\documentclass{article}

\usepackage[latin1]{inputenc}
\usepackage{amssymb}
\usepackage{amsfonts}
\usepackage{amscd}
\usepackage{mathrsfs}
\usepackage[dvips]{graphicx}
\usepackage[T1]{fontenc}
\usepackage[english]{babel}
\usepackage{marvosym}
\usepackage{amsthm}
\usepackage{multicol}
\usepackage[all]{xy}
\usepackage{todonotes}

\usepackage{pb-diagram}

\usepackage{bm}

\usepackage{wasysym}

\usepackage{color}

\theoremstyle{plain}
\newtheorem{thm}{Theorem}[section]
\newtheorem{lem}[thm]{Lemma}

\newtheorem{prop}[thm]{Proposition}
\newtheorem{conj}[thm]{Conjecture}
\newtheorem{coro}[thm]{Corollary}

\theoremstyle{definition}
\theoremstyle{remark}
\newtheorem{rk}[thm]{Remark}
\newtheorem{df}[thm]{Definition}

\def\ui{{\mathbf{i}}}
\def\uj{{\mathbf{j}}}

\def\bbC{\mathbb{C}}

\def\bbN{\mathbb{N}}

\def\bbQ{\mathbb{Q}}

\def\bbV{\mathbb{V}}

\def\bbZ{\mathbb{Z}}

\def\scrF{\mathscr{F}}

\def\scrR{\mathscr{R}}

\def\frakS{\mathfrak{S}}

\def\calA{\mathcal{A}}

\def\calC{\mathcal{C}}

\def\calK{\mathcal{K}}

\def\calO{\mathcal{O}}
\def\calP{\mathcal{P}}

\def\calX{\mathcal{X}}

\def\frakg{\mathfrak{g}}
\def\frakh{\mathfrak{h}}

\def\frakp{\mathfrak{p}}

\def\bfa{\mathbf{a}}

\def\bfg{\mathbf{g}}
\def\bfh{\mathbf{h}}
\def\bfk{\mathbf{k}}

\def\bfp{\mathbf{p}}

\def\bfA{\mathbf{A}}

\newcommand\restr[2]{{
  \left.\kern-\nulldelimiterspace 
  #1 
  \vphantom{\big|} 
  \right|_{#2} 
  }}

\def\homo{\operatorname{\it \mathscr{H}\kern-.25em om}}
\def\ext{\operatorname{\it \mathscr{E}\kern-.25em xt}}
\def\edo{\operatorname{\it \mathscr{E}\kern-.25em nd}}
\def\der{\operatorname{\it \mathscr{D}\kern-.25em er}}

\def\mod{\mathrm{mod}}
\def\proj{\mathrm{proj}}
\def\grmod{\mathrm{grmod}}
\def\bmod{\mathrm{bmod}}

\def\Hom{\mathrm{Hom}}
\def\RHom{\mathrm{RHom}}
\def\End{\mathrm{End}}

\def\Ext{\mathrm{Ext}}
\def\Ker{\mathrm{Ker}}
\def\Im{\mathrm{Im}}
\def\Id{\mathrm{Id}}
\def\KZ{\mathrm{KZ}}

\def\inc{\mathrm{inc}}
\def\tr{\mathrm{tr}} 
\def\Zuc{\mathrm{Zuc}}

\def\Com{\mathrm{Com}}

\def\Res{\mathrm{Res}}
\def\Ind{\mathrm{Ind}}

\def\bfwt{\mathrm{\mathbf{wt}}}

\title{\Huge Affine category O, Koszul duality and Zuckerman functors}
\author{Ruslan Maksimau  \\\\
Institut Montpelli\'erain Alexander Grothendieck (CNRS: UMR 5149),\\
Universit\'e Montpellier 2,\\
Case Courrier 051,\\
Place Eug\`ene Bataillon,\\
34095 MONTPELLIER Cedex,\\
FRANCE\\
ruslmax@gmail.com, ruslan.maksimau@umontpellier.fr.
}

\date{}

\begin{document}

\maketitle
\setcounter{tocdepth}{2}

\begin{abstract}
The parabolic category $\calO$ for affine ${\mathfrak{gl}}_N$ at level $-N-e$ admits a structure of a categorical representation of $\widetilde{\mathfrak{sl}}_e$ with respect to some endofunctors $E$ and $F$. This category contains a smaller category $\bfA$ that categorifies the higher level Fock space. We prove that the functors $E$ and $F$ in the category $\bfA$ are Koszul dual to Zuckerman functors. 

The key point of the proof is to show that the functor $F$ for the category $\bfA$ at level $-N-e$ can be decomposed in terms of the components of the functor $F$ for the category $\bfA$ at level $-N-e-1$. To prove this, we use the following fact from \cite{Mak-categ}: a category with an action of $\widetilde{\mathfrak sl}_{e+1}$ contains a (canonically defined) subcategory with an action of $\widetilde{\mathfrak sl}_{e}$.

We also prove a general statement that says that in some general situation a functor that satisfies a list of axioms is automatically Koszul dual to some sort of Zuckerman functor.

\end{abstract}

\tableofcontents

\section{Introduction}
Let $O^\nu_{-e}$ be the parabolic category $\calO$ with parabolic type $\nu$ of the affine version of the Lie algebra $\mathfrak{gl}_N$ at level $-N-e$. In \cite{RSVV}, a categorical representation of the affine Kac-Moody algebra $\widetilde{\mathfrak{sl}}_e$ in $O^\nu_{-e}$ is considered. In particular, this means that there are exact biadjoint functors $E_i,F_i\colon O^\nu_{-e}\to O^\nu_{-e}$ for $i\in[0,e-1]$ which induce a representation of the Lie algebra $\widetilde{\mathfrak{sl}}_e$ on the Grothendieck group $[O^\nu_{-e}]$ of $O^\nu_{-e}$.
The precise definition of a categorical representation is given in Section \ref{ch3:subs_categ-action}.
The category $O^\nu_{-e}$ admits a decomposition
$$
O^\nu_{-e}=\bigoplus_{\mu\in\bbZ^e}O^\nu_{\mu}
$$
that lifts the decomposition of the $\widetilde{\mathfrak{sl}}_e$-module $[O^\nu_{-e}]$ in a direct sum of weight spaces.

The category $O_\mu^\nu$ is Koszul by \cite{SVV}. Its Koszul dual category is the category $O_{\nu,+}^\mu$ defined similarly to $O_\nu^\mu$ at a positive level. In particular, the Koszul duality exchanges the parameter $\nu$ (\emph{the parabolic type}) with the parameter $\mu$ (\emph{the singular type}). The Koszul duality yields an equivalence of bounded derived categories $D^b(O^\nu_\mu)\simeq D^b({O}_{\nu,+}^\mu)$. More details about the Koszul duality can be found in \cite{BGS}.

Let $\alpha_0,\cdots,\alpha_{e-1}$ be the simple roots of $\widetilde{\mathfrak{sl}}_e$. We have
$$
E_i(O^\nu_\mu)\subset O^\nu_{\mu+\alpha_i}, \quad F_i(O^\nu_\mu)\subset O^\nu_{\mu-\alpha_i}.
$$
The aim of this paper is to prove that Koszul dual functors
$$
D^b(O^\mu_{\nu,+})\to D^b(O^{\mu+\alpha_i}_{\nu,+}),\quad D^b(O^\mu_{\nu,+})\to D^b(O^{\mu-\alpha_i}_{\nu,+})
$$
to the functors
$$
E_i\colon D^b(O_\mu^\nu)\to D^b(O_{\mu+\alpha_i}^\nu),\quad F_i\colon D^b(O_\mu^\nu)\to D^b(O_{\mu-\alpha_i}^\nu).
$$
are the Zuckerman functors.

Unfortunately, we cannot solve this problem for the full category $O$. But we are able to do this for a subcategory $\bfA$ of $O$. 



By definition, the Zuckerman functor is a composition of a parabolic inclusion functor with a parabolic truncation functor. Thus it is natural to try to decompose the functors $E_i$ and $F_i$ in two "smaller" functors. 
More precisely, we want to show that the functor $F_i\colon O^\nu_{\mu}\to O^\nu_{\mu'}$ (and similarly for $E_i$) between the blocks\footnote{We abuse the terminology using the word "block" here. We don't claim that $O^\nu_{\mu}$ is indecomposable. The word "block" is used here by analogy with the non-parabolic finite type category $\calO$, where similar subcategories are indeed blocks.} $O^\nu_{\mu_1}$ and $O^\nu_{\mu_2}$ can be decomposed as $O^\nu_{\mu_1}\to O^\nu_{\mu_{12}}\to O^\nu_{\mu_2}$, where the block $O^\nu_{\mu_{12}}$ is "more regular" than $O^\nu_{\mu_{1}}$ and than $O^\nu_{\mu_{2}}$. The additional difficulty is that there is no good candidate for such a block at the level $-N-e$, but there is one at another level: $-N-(e+1)$.


Now we describe more precisely our strategy of finding such a decomposition.
Let $\overline O^\nu_{\mu}$, $\overline E_i$, $\overline F_i$ be defined in the same way as $O^\nu_{\mu}$, $E_i$, $F_i$ with $e$ replaced by $e+1$. Let $\overline\alpha_0,\cdots,\overline\alpha_e$ be the simple roots of $\widetilde{\mathfrak{sl}}_{e+1}$. Fix $k\in[0,e-1]$. For an $e$-tuple $\mu=(\mu_1,\cdots,\mu_e)$ we set

$$
\overline\mu=
\left\{\begin{array}{ll}
(\mu_1,\cdots,\mu_k,0,\mu_{k+1},\cdots,\mu_e) & \mbox{ if }k\ne 0,\\
(0,\mu_1,\cdots,\mu_e) & \mbox{ if }k=0.
\end{array}\\
\right.
$$
Note that we have $\overline{(\mu-\alpha_k)}=\overline\mu-\overline\alpha_k-\overline\alpha_{k+1}$.

By \cite{Fie-str}, there is an equivalence of categories $\theta\colon O^\nu_{\mu}\to \overline O^\nu_{\overline\mu}$. The direct sum of such equivalences identifies the category $O_{-e}^\nu$ with a direct factor of the category $\overline O_{-(e+1)}^\nu$. We want to compare the $\widetilde{\mathfrak{sl}}_{e}$-action on $O_{-e}^\nu$ with the $\widetilde{\mathfrak{sl}}_{e+1}$-action on $\overline O_{-(e+1)}^\nu$. More precisely, we want to prove the following conjecture.

\smallskip
\begin{conj}
\label{ch3:conj-intro-art}
The following diagram of functors is commutative.
\begin{equation}
\label{ch3:diag-conj-intro}
\begin{diagram}
\node{\overline O^\nu_{\overline\mu}} \arrow{e,t}{\overline F_k}
\node{\overline O^\nu_{\overline\mu-\overline\alpha_k}} \arrow{e,t}{\overline F_{k+1}} 
\node{\overline O^\nu_{\overline\mu-\overline\alpha_k-\overline\alpha_{k+1}}} \arrow{s,r}{\theta^{-1}} \\
\node{O^\nu_{\mu}} \arrow{n,l}{\theta} 
\arrow[2]{e,b}{F_k} \node[2]{O^\nu_{\mu-\alpha_k}}
\end{diagram}
\end{equation}
\end{conj}

\smallskip
Now, if the conjecture is true, then it implies a decomposition that we expected. After that we could use an argument similar to \cite{MOS} to show that the functor $\overline F_k$ is Koszul dual to the parabolic inclusion functor and the functor $\overline F_{k+1}$ is Koszul dual to the parabolic truncation functor. Then we can deduce that $F_k$ is Koszul dual to the Zuckerman functor (which is the composition of the parabolic inclusion functor and the parabolic truncation functor). Thus the problem is reduced to the proof of this conjecture.

It is not hard to see that the diagram from Conjecture \ref{ch3:conj-intro-art} is commutative at the level of Grothendieck groups. In the case of the category $\calO$ of ${\mathfrak{gl}}_N$ (instead of affine ${\mathfrak{gl}}_N$) this is already enough to prove the analogue of Conjecture \ref{ch3:conj-intro-art}, using the theory of projective functors. Indeed, \cite[Thm.~3.4]{BG} implies that two projective functors are isomorphic if their actions on the Grothendieck group coincide. Unfortunately, there is no satisfactory theory of projective functors for the affine case (an attempt to develop such a theory was given in \cite{FrMal}).

We choose another strategy to prove this conjecture. It is based on the main result of \cite{Mak-categ}, relating the notion of a categorical representaion of $\widetilde{\mathfrak{sl}}_e$ with the notion of a categorical representation of $\widetilde{\mathfrak{sl}}_{e+1}$. 
Let us fix the following inclusion of Lie algebras $\widetilde{\mathfrak{sl}}_e\subset\widetilde{\mathfrak{sl}}_{e+1}$ 
$$
e_r\mapsto
\left\{\begin{array}{rl}
e_r &\mbox{ if }r\in[0,k-1],\\
{[e_k,e_{k+1}]} &\mbox{ if }r=k,\\
e_{r+1} &\mbox{ if }r\in[k+1,e-1],
\end{array}\right.
$$
$$
f_r\mapsto
\left\{\begin{array}{rl}
f_r &\mbox{ if }r\in[0,k-1],\\
{[f_{k+1},f_k]} &\mbox{ if }r=k,\\
f_{r+1} &\mbox{ if }r\in[k+1,e-1].
\end{array}\right.
$$

The Lie algebra $\widetilde{\mathfrak{sl}}_{e}$ has a categorical representation in the category $O_{-e}^\nu$ while the Lie algebra $\widetilde{\mathfrak{sl}}_{e+1}$ has a categorical representation in the category $\overline O_{-(e+1)}^\nu$. By \cite[Thm.~3.5]{Mak-categ}, each category $\overline\calC$ with a categorical action of $\widetilde{\mathfrak{sl}}_{e+1}$ contains (under some assumptions) a (canonically defined) subcategory $\calC\subset \overline\calC$ that inherits a categorical action of $\widetilde{\mathfrak{sl}}_{e}$ from the categorical action of $\widetilde{\mathfrak{sl}}_{e+1}$ on $\overline\calC$. In particular, if we take $\overline\calC=\overline O_{-(e+1)}^\nu$, then the subcategory $\calC$ can be easily identified with $\overline O_{-e}^\nu$ using equivalences $\theta$ like in (\ref{ch3:diag-conj-intro}).

We get two categorical representations of $\widetilde{\mathfrak{sl}}_e$ in $O^{\nu}_{-e}$:
\begin{itemize}
  \item the original one,
  \item the $\widetilde{\mathfrak{sl}}_e$-categorical representation structure induced from the $\widetilde{\mathfrak{sl}}_{e+1}$-categorical representation structure in $\overline O^{\nu}_{-(e+1)}$.
\end{itemize}
To prove Conjecture \ref{ch3:conj-intro-art}, it is enough to prove that these two categorical representation structures are the same.

\smallskip

Unfortunately we cannot apply the uniqueness theorem for categorical representations because the $\widetilde{\mathfrak{sl}}_e$-module categorified by $O^\nu_{-e}$ is not simple. However, we can obtain a weaker version of Conjecture \ref{ch3:conj-intro-art}. The category $O^\nu_{-e}$ contains subcategories $\bfA^\nu[\alpha]$, parameterized by $\alpha\in Q^+_e$, where $Q^+_e$ is the positive part of the root lattice of $\widetilde{\mathfrak{sl}}_{e}$. The direct sum of such categories categorifies the Fock space.\footnote{To get a categorification of the Fock space, we sum by $\alpha$, but the values of $N$ and $\nu$ are not fixed. For each $\alpha$ we need $N$ and $\nu$ such that $\nu_r\geqslant |\alpha|$. See \cite[Sec.~7.4]{RSVV} for more details.} In this case we can use the technique similar to one used in \cite{RSVV}.
This technique allows to prove in some cases that two categorical representations that categorify the Fock space are the same. We get the following.

For $i\in[0,e-1]$ we have $F_i(\bfA^\nu[\alpha])\subset \bfA^\nu[\alpha+\alpha_i]$. Let $\overline\bfA^\nu[\overline\alpha]$ be defined in the same way as $\bfA^\nu[\alpha]$ with respect to the parameter $e+1$ instead of $e$.
Let $|\alpha|$ be the height of $\alpha$. For each $\alpha\in Q^+_e$, we construct $\overline\alpha=\phi(\alpha)\in Q^+_{e+1}$, where the map $\phi\colon Q_e\to Q_{e+1}$ is defined in Section \ref{ch3:subs_not-quiv-I-Ibar} (see also Section \ref{ch3:subs_not-e-e+1}).  

The main result of Section \ref{ch3:sec_catO} is the following theorem.

\smallskip
\begin{thm}
\label{ch3:thm_intro-main-decomp-functors}
Assume that $e>2$ and $\nu=(\nu_1,\cdots,\nu_l)$ satisfies $\nu_r>|\alpha|$ for each $r\in[1,l]$. There exists $\beta\in Q^+_{e+1}$ such that for each $\alpha\in Q^+_{e}$ as above there are equivalences of categories $\theta'_{\alpha}\colon \bfA^\nu[\alpha]\to \overline\bfA^\nu[\beta+\overline\alpha]$ and $\theta'_{\alpha+\alpha_k}\colon \bfA^\nu[\alpha+\alpha_k]\to \overline\bfA^\nu[\beta+\overline\alpha+\overline\alpha_k+\overline\alpha_{k+1}]$ such that the following diagram is commutative
$$
\begin{CD}
\overline\bfA^\nu[\beta+\overline\alpha]@>{\overline F_{k+1}\overline F_k}>> \overline\bfA^\nu[\beta+\overline\alpha+\overline\alpha_k+\overline\alpha_{k+1}]\\
@A{\theta'_\alpha}AA                 @A{\theta'_{\alpha+\alpha_k}}AA\\
\bfA^\nu[\alpha]@>{F_k}>> \bfA^\nu[\beta+\alpha+\alpha_k].
\end{CD}
$$
\qed
\end{thm}

\smallskip

The paper has the following structure. 

In Section \ref{ch3:sec_catO}, we use the categorical representations to decompose the functor $F$ in the category $\bfA$ (Theorem \ref{ch3:thm_intro-main-decomp-functors}). We do this in the following way. We consider a category $\calA$ that is equivalent to $\bfA$ as a category, but the functors $E_i$ and $F_i$ on $\calA$ are defined with respect to another categorical action that comes from level $-N-(e+1)$. We want to find a (probably different) equivalence between $\bfA$ and $\calA$ that identifies the functors. The paper \cite{RSVV} compares $\bfA$ (together with $E_i$ and $F_i$) with the category $\calO$ for the rational Cherednik algebra. We use a similar argument to compare $\calA$ with the same category $\calO$. As a consequence, we manage to compare $\bfA$ with $\calA$ (together with the functors $E_i$ and $F_i$).

In Section \ref{ch3:sec_gr-lifts} we prove that in some cases the functors $E$ and $F$ for the category $O$ admit graded lifts. For this we use Soergel's functor $\bbV$. 

In Section \ref{ch3:sec-Koszul} we prove that the functors $E$ and $F$ for the category $\bfA$ are Koszul dual to Zuckerman functors. In fact, in Section \ref{ch3:subs_key-lem} we prove a more general and more abstract statement that says that in some general situation a functor that satisfies a list of axioms is automatically Koszul dual to some sort of Zuckerman functor. The proof of this statement uses the approach of \cite{MOS}. 


The technique of restriction of categorical representations in \cite{Mak-categ} is developed for solving the problem in the present paper. However, this technique has an independent interest: another application is given in \cite[Sec.~7]{RW}.

It is important to emphasize the relation between the present paper and the preprint \cite{Mak-Zuck}. The preprint \cite{Mak-Zuck} is expected to be published as two different papers: the first of them is \cite{Mak-categ}, the second one is the present paper. The first part contains general results about KLR algebras and categorical representations. The second part is an application of the first part. The present paper is rewritten (compared to \cite{Mak-Zuck}) in a way that we never use KLR algebras explicitly. This makes the paper more independent from \cite{Mak-categ}.

\section{The category $O$}
\label{ch3:sec_catO}

For a noetherian ring $A$ we denote by $\mod(A)$ the abelian category of left finitely generated $A$-modules. We denote by $\bbN$ the set of non-negative integers. By a commutative diagram of functors we always mean a diagram that commutes up to an isomorphism of functors.

\subsection{Kac-Moody algebras associated with a quiver}
\label{ch3:subs_KM-quiv}
Let $\Gamma=(I,H)$ be a quiver without $1$-loops with the set of vertices $I$ and the set of arrows $H$. For $i,j\in I$ let $h_{i,j}$ be the number of arrows from $i$ to $j$ and set also $a_{i,j}=2\delta_{i,j}-h_{i,j}-h_{j,i}$. Let $\frakg_I$ be the Kac-Moody algebra over $\bbC$ associated with the matrix $(a_{i,j})$. More precisely, the algebra $\frakg_I$ is generated by elements $e_i$, $f_i$, $h_i$ for $i\in I$ satisfying the following relations:
$$
\begin{array}{rcll}
[e_i,f_j]&=&\delta_{ij}h_i,\\
{[h_i,e_j]}&=&a_{i,j}e_j,\\
{[h_i,f_j]}&=&-a_{i,j}f_j,\\
{[h_i,h_j]}&=&0,\\
{\rm ad}(e_i)^{1-a_{ij}}(e_j)&=&0 &\mbox{ if }i\ne j,\\
{\rm ad}(f_i)^{1-a_{ij}}(f_j)&=&0 &\mbox{ if }i\ne j.\\
\end{array}
$$

For each $i\in I$ let $\alpha_i$ be the simple root corresponding to $e_i$.
Set
$$
Q_I=\bigoplus_{i\in I}\bbZ\alpha_i,\quad Q^+_I=\bigoplus_{i\in I}\bbN\alpha_i. 
$$
Let $X_I$ be the free abelian group with basis $\{\varepsilon_i;~i\in I\}$. Set also
$$
X^+_I=\bigoplus_{i\in I}\bbN\varepsilon_i.
$$

For $\alpha\in Q^+_I$ denote by $|\alpha|$ its height, i.e., for $\alpha=\sum_{i\in I}d_i\alpha_i$, $d_i\in \bbN$, we have $|\alpha|=\sum_{i\in I}d_i$.  Set $I^\alpha=\{\ui=(i_1,\cdots,i_{|\alpha|})\in I^{|\alpha|};~ \sum_{r=1}^{|\alpha|}\alpha_{i_r}=\alpha\}$. 

Let $\Gamma_\infty=(I_\infty,H_\infty)$ be the quiver with the set of vertices $I_\infty=\bbZ$ and the set of arrows $H_\infty=\{i\to i+1;~i\in I_\infty\}$. In this case we will simply write ${\mathfrak{sl}}_\infty$ for the Lie algebra $\frakg_{I_\infty}$. We stress that here we get a "two-sided" ${\mathfrak{sl}}_\infty$, the generators of ${\mathfrak{sl}}_\infty$ are parameterized by $\bbZ$ and not by $\bbN$.

 Assume that $e>1$ is an integer. Let $\Gamma_e=(I_e,H_e)$ be the quiver with the set of vertices $I_e=\bbZ/e\bbZ$ and the set of arrows $H_e=\{i\to i+1;~i\in I_e\}$. Then $\frakg_{I_e}$ is the Lie algebra $\widetilde{\mathfrak{sl}}_e=\mathfrak{sl}_e\otimes\bbC[t,t^{-1}]\oplus\bbC \bm{1}$.

Assume that $\Gamma=(I,H)$ is a quiver whose connected components are of the form $\Gamma_e$, with $e\in\bbN$, $e>1$ or $e=\infty$. For $i\in I$ denote by $i+1$ and $i-1$ the (unique) vertices in $I$ such that there are arrows $i\to i+1$ and $i-1\to i$.

Let us also consider the following additive map
$$
\iota\colon Q_I\to X_I,\quad\alpha_i\mapsto \varepsilon_i-\varepsilon_{i+1}.
$$

Fix a formal variable $\chi$ and set $X_I^\chi=X_I\oplus\bbZ\chi$.
We can lift the $\bbZ$-linear map $\iota$ to a $\bbZ$-linear map
$$
\iota^\chi\colon Q_I\to X_I^\chi,\quad \alpha_i\mapsto \varepsilon_i-\varepsilon_{i+1}-\chi.
$$

Note that the map $\iota^\chi\colon Q_I\to X^\chi_I$ is injective (while $\iota$ is not injective).
We may omit the symbols $\iota$, $\iota^\chi$ and write $\alpha$ instead of $\iota(\alpha)$ or $\iota^\chi(\alpha)$.

We will also abbreviate
$$
Q_e=Q_{I_e},\quad X_e=X_{I_e},\quad X^\chi_e=X^\chi_{I_e}. 
$$

\subsection{Doubled quiver}
\label{ch3:subs_not-quiv-I-Ibar}

Now we recall the notion of a doubled quiver introduced in \cite[Sec.~2B]{Mak-categ}.

Let $\Gamma=(I,H)$ be a quiver without $1$-loops. Fix a decomposition $I=I_0\sqcup I_1$ such that there are no arrows between the vertices in $I_1$.  In this section we define a \emph{doubled quiver} $\overline\Gamma=(\overline I,\overline H)$ associated with $(\Gamma,I_0,I_1)$. The motivation of this definition is that there is a relation between categorical representations of $\frakg_I$ and of $\frakg_{\overline I}$, see \cite{Mak-categ} for more details. Later, we will apply this relation to the categorical representation in the category $\calO$.

 The idea of how we get $\overline\Gamma$ from $\Gamma$ is to "double" each vertex in the set $I_1$ (we do not touch the vertices from $I_0$). We replace each vertex $i\in I_1$ by a couple of vertices $i^1$ and $i^2$ with an arrow $i^1\to i^2$. Each arrow entering to $i$ should be replaced by an arrow entering to $i^1$, each arrow coming from $i$ should be replaced by an arrow coming from $i^2$. See  \cite[Sec.~2B]{Mak-categ} for a more formal definition of the quiver $\overline\Gamma$. This construction will be used in the present paper for two special types of quivers mentioned in Section \ref{ch3:subs_not-e-e+1}.


Set $I^\infty=\coprod_{d\in \bbN} I^d$, $\overline I^\infty=\coprod_{d\in \bbN} \overline I^d$, where $I^d$, $\overline I^d$ are the cartesian products.
The concatenation yields a monoid structure on $I^\infty$ and $\overline I^\infty$.
Let $\phi\colon I^\infty\to \overline I^\infty$ be the unique morphism of monoids such that for $i\in I\subset I^\infty$ we have
$$
\phi(i)=
\left\{\begin{array}{ll}
i^0 &\mbox{ if }i\in I_0,\\
(i^1,i^2) &\mbox{ if } i\in I_1.
\end{array}\right.
$$

There is a unique $\bbZ$-linear map $\phi\colon Q_I\to Q_{\overline I}$ such that $\phi(I^\alpha)\subset I^{\phi(\alpha)}$ for each $\alpha\in Q^+_I$. It is given by
$$
\phi(\alpha_{i})=
\left\{\begin{array}{ll}
\alpha_{i^0} &\mbox{ if }i\in I_0,\\
\alpha_{i^1}+\alpha_{i^2}&\mbox{ if }i\in I_1.
\end{array}\right.
$$
Let $\phi$ denote also the unique additive embedding
\begin{equation}
\label{ch3:eq_phi(mu)}
\phi\colon X_I\to X_{\overline I}, \quad\varepsilon_i\mapsto \varepsilon_{i'},
\end{equation}
where
\begin{equation}
\label{ch3:eq_i'}
i'=
\left\{\begin{array}{ll}
i^0 &\mbox{ if }i\in I_0,\\
i^1&\mbox{ if }i\in I_1.
\end{array}\right.
\end{equation}

\subsection{Special quivers}
\label{ch3:subs_not-e-e+1}
We will use the construction from the previous section only for two special types of quivers.

\medskip
First, consider the quiver $\Gamma=\Gamma_{e}$, for $e>1$.  We have $I=I_e=\bbZ/e\bbZ$. Fix $k\in [0,e-1]$ and set $I_1=\{k\}$, $I_0=I\backslash\{k\}$. In this case the quiver $\overline\Gamma$ is isomorphic to $\Gamma_{e+1}$. The isomorphism $\overline{\Gamma_e}\simeq \Gamma_{e+1}$ at the level of vertices is 
$$
\begin{array}{rcll}
i^0&\mapsto & i &\mbox{ if } i\in [0;k-1],\\
k^1&\mapsto & k,&\\
k^2&\mapsto & k+1,&\\
i^0&\mapsto & i+1 &\mbox{ if } i\in [k+1,e-1].
\end{array}
$$

To avoid confusion, for $i\in \overline I=I_{e+1}$ we will write $\overline\alpha_i$ and $\overline\varepsilon_i$ instead of $\alpha_i$ and $\varepsilon_i$ respectively.


\smallskip
Let $\Upsilon\colon\bbZ\to\bbZ$ be the map given for $a\in\bbZ,b\in[0,e-1]$ by
\begin{equation}
\label{ch3:eq_upsilon}
\Upsilon(ae+b)=
\left\{\begin{array}{ll}
a(e+1)+b &\mbox{ if }b\in[0,k],\\
a(e+1)+b+1 &\mbox{ if }b\in[k+1,e-1].
\end{array}\right.
\end{equation}
Note that we have the following commutative diagram
$$
\begin{CD}
\bbZ @>{\Upsilon}>> \bbZ\\
@VVV           @VVV\\
I_e @>>>     I_{e+1},
\end{CD}
$$
where the bottom map is $i\mapsto i'$, see (\ref{ch3:eq_i'}).

\medskip
Now we describe the second quiver that we are interested in. Let $\widetilde\Gamma=(\widetilde I,\widetilde H)$ be the disjoint union of $l$ copies of $\Gamma_\infty$. 
 
 Write $\widetilde\alpha_{i}$ and $\widetilde\varepsilon_{i}$ instead of $\alpha_{i}$ and $\varepsilon_{i}$ respectively for each $i\in \widetilde I$.
We identify an element of $\widetilde I$ with an element $(a,b)\in I_\infty\times [1,l]$ in the obvious way.
Consider the decomposition $\widetilde I=\widetilde I_0\sqcup \widetilde I_1$ such that $(a,b)\in \widetilde I_1$ if and only if $a\equiv k~\mod~e$.
In this case the quiver $\overline{\widetilde\Gamma}$ is isomorphic to $\widetilde\Gamma$. More precisely, in this case we have
$$
\begin{array}{ll}
(a,b)^0=(\Upsilon(a),b),\\
(a,b)^1=(\Upsilon(a),b),\\
(a,b)^2=(\Upsilon(a)+1,b).\\
\end{array}
$$
To avoid confusion, we will always write $\widetilde\phi$ for any of the maps $\widetilde\phi\colon\widetilde I^\infty\to\widetilde I^\infty$, $Q_{\widetilde I}\to Q_{\widetilde I}$, $X_{\widetilde I}\to X_{\widetilde I}$ in Section \ref{ch3:subs_not-quiv-I-Ibar}. 



\medskip
Consider the quiver homomorphism $\pi_e\colon \widetilde \Gamma\to \Gamma_e$ such that
$$
\pi_e\colon\widetilde I\to I,~(a,b)\mapsto a~\mod~e.
$$
Similarly,  we have a quiver homomorphism $\pi_{e+1}\colon\widetilde\Gamma\to\Gamma_{e+1}$.
They yield $\bbZ$-linear maps
$$
\pi_e\colon Q_{\widetilde I}\to Q_{e}, \quad \pi_e\colon X_{\widetilde I}\to X_{e},\quad \pi_{e+1}\colon Q_{\widetilde I}\to Q_{e+1}, \quad \pi_{e+1}\colon X_{\widetilde I}\to X_{e+1}.
$$

The following diagrams are commutative for $\alpha\in Q^+_e$, $\widetilde\alpha\in Q^+_{\widetilde I}$ such that $\pi_e(\widetilde\alpha)=\alpha$,
$$
\begin{CD}
Q_{\widetilde I} @>{\widetilde\phi}>> Q_{\widetilde I}\\
@V{\pi_e}VV                                  @V{\pi_{e+1}}VV\\
Q_e              @>{\phi}>>           Q_{e+1}\\
\end{CD}
\qquad\qquad
\begin{CD}
X_{\widetilde I} @>{\widetilde\phi}>> X_{\widetilde I}\\
@V{\pi_e}VV                                  @V{\pi_{e+1}}VV\\
X_e              @>{\phi}>>           X_{e+1}\\
\end{CD}
\qquad\qquad
\begin{CD}
\widetilde I^{\widetilde\alpha} @>{\widetilde\phi}>> \widetilde I^{\widetilde\phi(\widetilde\alpha)}\\
@V{\pi_e}VV                                  @V{\pi_{e+1}}VV\\
I_e^{\alpha}         @>{\phi}>>           I_{e+1}^{\phi(\alpha)}\\
\end{CD}
$$

\subsection{Deformation rings}
\label{ch3:subs_def-ring}

In this section we introduce some general definitions from \cite{RSVV} for a later use. Let us fix an integer $e>1$.

We call \emph{deformation ring} $(R,\kappa,\tau_1,\cdots,\tau_l)$ a regular commutative noetherian $\bbC$-algebra $R$ with $1$ equipped with a homomorphism $\bbC[\kappa^{\pm 1},\tau_1,\cdots,\tau_l]\to R$. Let $\kappa,\tau_1,\cdots,\tau_l$ denote also the images of $\kappa,\tau_1,\cdots,\tau_l$ in $R$.
A deformation ring is \emph{in general position} if any two elements of the set
$$
\{\tau_u-\tau_v+a\kappa+b,\kappa-c;~a,b\in\bbZ,c\in \bbQ,u\ne v\}
$$
have no common non-trivial divisors.
A \emph{local deformation ring} is a deformation ring which is a local ring and such that $\tau_1,\cdots,\tau_l, \kappa-e$ belong to the maximal ideal of $R$. 
Note that each $\bbC$-algebra that is a field has a \emph{trivial} local deformation ring structure, i.e., such that $\tau_1=\cdots=\tau_l=0$ and $\kappa=e$. We always consider $\bbC$ as a local deformation ring with a trivial deformation ring structure.
A $\bbC$-algebra $R$ is called \emph{analytic} if it is a localization of the ring of germs of holomorphic functions on some compact polydisc $D\subset \bbC^d$ for some $d\geqslant 1$.

We will write $\overline\kappa=\kappa(e+1)/e$, $\overline\tau_r=\tau_r(e+1)/e$.
We will abbreviate $R$ for $(R,\kappa,\tau_1,\cdots,\tau_l)$ and $\overline R$ for $(R,\overline\kappa,\overline\tau_1,\cdots,\overline\tau_l)$.

\medskip
Let us give names to two special assumptions on $R$ that will be often used in the paper.


\medskip
\textbf{Assumption 1.} $R$ is a local analytic deformation ring in general
position of dimension $\leqslant 2$.

\medskip
We denote by $\bfk$ the residue field and by $K$ the field of fractions of $R$. Assumption 1 is necessary to be able to study categorical representations in the affine category $\calO$, see \cite{RSVV}.


Now, we introduce a modified version of this assumption. This is necessary to be able to state results that are true over $R$, over $K$ and over $\bfk$ without speaking about $K$ and $\bfk$ separately, see for example Proposition \ref{ch3:prop_functors-on-O-gen}.

\medskip
\textbf{Assumption 2.} The ring $R$ is either as in Assumption 1 or is the fraction field or the residue field of a ring satisfying Assumption 1.

\medskip
If $R$ satisfies the second part of Assumption 2, we simply mean $\bfk=K=R$.

\medskip
Let $R$ be as in Assumption 1.
Consider the element $q_{e}=\exp(2\pi \sqrt{-1}/\kappa)$ 
in $R$. This element specializes to $\zeta_{e}=\exp(2\pi \sqrt{-1}/e)$ 
in $\bfk$. 
If $R$ satisfies Assumption 2, then it has a deformation ring structure. The element $q_e$ 
still makes sense in $R$. If $R$ is a resudue field of a ring satisfying Assumption 1, then we mean that $q_e$ is $\zeta_e$. 


\subsection{The Lie algebra $\widehat{\mathfrak{gl}}_N$}
\label{ch3:subs_aff-Lie}

Fix positive integers $N$, $l$ and $e$ such that $e>1$. Let $R$ be a deformation ring, see Section \ref{ch3:subs_def-ring}. Set
$$
\bfg_R=\mathfrak{gl}_N(R),\quad \widehat\bfg_R=\widehat{\mathfrak{gl}}_N(R)=\mathfrak{gl}_N(R)[t,t^{-1}]\oplus R\bm{1}\oplus R\partial.
$$
For $i,j\in[1,N]$ let $e_{i,j}\in\bfg_R$ denote the matrix unit. Let $\bfh_R\subset\bfg_R$ be the Cartan subalgebra generated by the $e_{i,i}$'s, and $\epsilon_1,\cdots,\epsilon_N$ be the basis of $\bfh_R^*$ dual to $e_{1,1},\cdots,e_{N,N}$. Let $P=\bbZ\epsilon_1\oplus\cdots\oplus\bbZ\epsilon_N$ be the weight lattice of $\bfg_R$. We identify $P$ with $\bbZ^N$.
Let $\Pi,~\widehat \Pi\subset \bfh_R^*$ be the sets of simple roots of $\bfg_R$ and $\widehat \bfg_R$. Let $W=\frakS_N$ be the Weyl group of $\bfg_R$ and $\widetilde W=W\ltimes \bbZ\Pi$, $\widehat W=W\ltimes P$ be the affine and the extended affine Weyl groups.

Then we define the element $\bfwt_e(\lambda)\in X_e$ given by
$$
\bfwt_e(\lambda)=\sum_{s=1}^N\varepsilon_{\lambda_r},
$$
where we write $\varepsilon_{\lambda_r}$ for $\varepsilon_{(\lambda_r~\mod~e)}$.

We will abbreviate
\begin{equation}
\label{ch3:eq_def-P[mu]}
P[\mu]=\{\lambda\in P ; ~\bfwt_e(\lambda)=\mu\}.
\end{equation}

Similarly, we consider the weight
$$
\bfwt^\chi_e(\lambda)=\sum_{r=1}^N\varepsilon_{\lambda_r}+(\sum_{r=1}^N\lambda_r)\chi\in X^\chi_e.
$$

Finally, let $X_e[N]\subset X_e$ be the subset given by
$$
X_e[N]=\{\mu=\sum_{r=1}^e\mu_r\varepsilon_r\in X_e;~ \mu_r\geqslant 0,~ \sum_{r=1}^e\mu_r=N\}.
$$
We may identify $\mu$ with the tuple $(\mu_1,\cdots,\mu_e)$ if no confusion is possible.


Now, consider the Cartan subalgebra $\widehat\bfh_R=\bfh_R\oplus R\bm{1}\oplus R\partial$ of $\widehat\bfg_R$. Let $\alpha_0, \alpha_1,\cdots,\alpha_{N-1}\subset \widehat\bfh_R^*$ and $\check\alpha_0, \check\alpha_1,\cdots,\check\alpha_{N-1}\subset \widehat\bfh_R$ be the simple roots and coroots of $\widehat\bfg_R$ respectively. Let $\Lambda_0$ and $\delta$ be the elements of $\widehat\bfh_R^*$ defined  by
$$
\delta(\partial)=\Lambda_0(\bm{1})=1,\quad \delta(\bfh_R\oplus R\bm{1})=\Lambda_0(\bfh_R\oplus R\partial)=0.
$$
Let $(\bullet,\bullet)\colon \widehat\bfh_R^*\times \widehat\bfh_R^*\to R$ be the bilinear form such that
$$
\lambda(\check\alpha_r)=(\lambda,\alpha_r),\quad \lambda(\partial)=(\lambda,\Lambda_0),\qquad \forall\lambda\in \widehat\bfh^*_R.
$$

Set $P_R=P\otimes_{\bbZ}R$. Given an $l$-tuple of positive integers $\nu=(\nu_1,\cdots,\nu_l)$ such that $\sum_{r=1}^l \nu_r=N$, we define
$$
\begin{array}{lll}
\rho&=&(0,-1,\cdots,-N+1),\\
\rho_\nu&=&(\nu_1,\nu_1-1\cdots,1,\nu_2,\cdots,1,\cdots,\nu_l,\cdots,1),\\
\tau&=&(\tau_1^{\nu_1},\cdots,\tau_l^{\nu_l}),
\end{array}
$$
where $\tau_r^{\nu_r}$ means $\nu_r$ copies of $\tau_r$.
Set also
\begin{equation}
\label{ch3:eq_lambda-tilde}
\widehat\rho=\rho+N\Lambda_0, \quad
\widetilde\lambda=\lambda+\tau+z_\lambda\delta-(N+\kappa)\Lambda_0,
\end{equation}
where $z_\lambda=(\lambda,2\rho+\lambda)/2{\kappa}$. Denote by $\widehat{\bfp}_{R,\nu}$ the parabolic subalgebra of $\widehat{\bfg}_R$ of parabolic type $\nu$.
For a $\nu$-dominant weight $\lambda\in P$ let $\Delta(\lambda)_R$ be the parabolic Verma module with highest weight $\widetilde\lambda$ and $\Delta_R^{\lambda}=\Delta(\lambda-\rho)_R$. 
We will also skip the subscript $R$ when $R=\bbC$.

\subsection{Affine Weyl groups}
\label{ch3:subs_ext-aff}
Assume that $R=\bbC$.
In this section we discuss some combinatorial aspects of the $\widehat W$-action on $\widehat\bfh^*$.

The group $\widehat W$ is generated by $\{\pi,s_i;~i\in\bbZ/N\bbZ\}$ modulo the relations
$$
\begin{array}{ccccc}
s_i^2&=&1,\\
s_is_j&=&s_js_i \quad&\forall i\ne j\pm 1,\\
s_is_{i+1}s_i&=&s_{i+1}s_is_{i+1},\\
\pi s_{i+1}&=&s_{i}\pi.
\end{array}
$$

Let $\widetilde W$ be the subgroup of $\widehat W$ generated by $\{s_i;~i\in \bbZ/N\bbZ\}$.
The group $\widehat W$ acts on $P$ in the following way:
\begin{itemize}
    \item[\textbullet] $s_r$ switches of the $r$th and $(r+1)$th components of $\lambda$ if $r\ne 0$,
    \item[\textbullet] $s_0(\lambda_1,\cdots,\lambda_N)=(\lambda_N-e,\lambda_2,\cdots,\lambda_{N-1},\lambda_1+e)$,
    \item[\textbullet]
$\pi(\lambda_1,\cdots,\lambda_N)=(\lambda_2,\cdots,\lambda_N,\lambda_1+e)$.
\end{itemize}
We will call this action of $\widehat W$ on $P$ the negative $e$-\emph{action}. We will always consider only negative actions of $\widehat W$ on $P$ up to Section \ref{ch3:subs_dual-funct-in-O}. So we can skip the word "negative". We may write $P^{(e)}=P$ to stress that we consider the $e$-action of $\widehat W$ on $P$. The map
$$
P^{(e)}\to \widehat\bfh^*,\quad \lambda\mapsto\widetilde{\lambda-\rho}+\widehat\rho
$$
is $\widehat W$-invariant. This means that the weights $\lambda_1,\lambda_2\in P$ are in the same $\widehat W$-orbit if and only if the highest weights of the Verma modules $\Delta^{\lambda_1}$ and $\Delta^{\lambda_2}$ are linked with respect to the Weyl group $\widehat W$, see \cite[Sec.~3.2]{SVV} and \cite[Sec.~2.3]{Fie-str} for more details about linkage.
Note that $P=\coprod_{\mu\in X_e[N]} P[\mu]$ is the decomposition of $P$ into $\widehat W$-orbits with respect to the $e$-action. An element $\lambda\in P$ is $e$-\emph{anti-dominant} if $\lambda_1\leqslant \lambda_2\leqslant\cdots\leqslant\lambda_N\leqslant \lambda_1+e$.

Recall the map $\Upsilon\colon\bbZ\to\bbZ$ from (\ref{ch3:eq_upsilon}). Applying $\Upsilon$ coordinate by coordinate to the elements of $P$ we get a map $\Upsilon\colon P^{(e)}\to P^{(e+1)}$.

\smallskip
\begin{lem}
The map $\Upsilon\colon P^{(e)}\to P^{(e+1)}$ is $\widehat W$-invariant and takes $e$-anti-dominant weights to $(e+1)$-anti-dominant weights.
\qed
\end{lem}

\subsection{Hecke algebras}
\label{ch3:subs_Hecke}

Let $R$ be a  commutative ring with $1$. 
Fix an invertible element $q\in R$. 

\smallskip
\begin{df}
The \emph{affine Hecke algebra} $H_{d,R}(q)$ is the $R$-algebra generated by $T_1,\cdots,T_{d-1}$ and the invertible elements $X_1,\cdots,X_d$ modulo the following defining relations
$$
\begin{array}{lllll}
X_rX_s=X_sX_r, &T_rX_r=X_rT_r ~\mbox{ if }|r-s|>1,\\
T_{r}T_{s}=T_{s}T_{r} \mbox{ if }|r-s|>1,& T_{r}T_{r+1}T_{r}=T_{r+1}T_{r}T_{r+1},\\
 T_rX_{r+1}=X_rT_r+(q-1)X_{r+1}, &T_rX_{r}=X_{r+1}T_r-(q-1)X_{r+1},\\
(T_r-q)(T_r+1)=0.\\
\end{array}
$$
\end{df}

Fix an $l$-tuple  $Q=(Q_1,\cdots,Q_l)\in R^l$.

\smallskip
\begin{df}
The \emph{cyclotomic Hecke algebra} $H^Q_{d,R}(q)$ is the quotient of $H_{d,R}(q)$ by the two-sided ideal generated by $(X_1-Q_1)\cdots(X_1-Q_l)$.
\end{df}

\subsection{Categorical representations}
\label{ch3:subs_categ-action}

Let $R$ be a $\bbC$-algebra. Fix an invertible element $q\in R$, $q\ne 1$. Let $\calC$ be an exact $R$-linear category.

\smallskip
\begin{df}
A \emph{representation datum} in $\calC$ is a tuple $(E,F,X,T)$ where $(E,F)$ is a pair of exact functors $\calC\to\calC$ and $X\in\End(F)^{\rm op}$, $T\in\End(F^2)^{\rm op}$ are endomorphisms of functors such that
for each $d\in\bbN$, there is an $R$-algebra homomorphism $\psi_d\colon H_{d,R}(q)\to \End(F^d)^{\rm op}$ given by
$$
\begin{array}{ll}
X_r\mapsto F^{d-r}XF^{r-1} &\forall r\in[1,d],\\
T_r\mapsto F^{d-r-1}TF^{r-1} &\forall r\in[1,d-1].
\end{array}
$$
 \end{df}
\smallskip
Now, assume that $R=\bfk$ is a field. Assume that $\calC$ is a $\Hom$-finite abelian category.

\smallskip
\begin{rk}
\label{ch3:rk_df-repl-F-by-E}
Assume that we have a representation datum in a $\bfk$-linear category $\calC$ such that the functors $E$ and $F$ are biadjoint. Then by adjointness we have an algebra isomorphism $\End(E^d)\simeq \End(F^d)^{\rm op}$. In particular we get an algebra homomorphism $H_{d,\bfk}(q)\to \End(E^d)$.
\end{rk}

Let $\scrF$ be a subset of $\bfk^\times$. We view $\scrF$ as the vertex set of a quiver $\Gamma_\scrF$ with an arrow $i \to j$ if and only if $j = qi$.

\smallskip
\begin{df}
\label{ch3:def-categ_action-Hecke}
An $\mathfrak{g}_{\scrF}$-categorical representation in $\calC$ is the datum of a representation datum $(E,F,X,T)$ and a decomposition $\calC=\bigoplus_{\mu\in X_\scrF}\calC_\mu$ satisfying the conditions $(a)$ and $(b)$ below. For $i\in\scrF$ let $E_i$, $F_i$ be endofunctors of $\calC$ such that for each $M\in\calC$ the objects $E_i(M)$, $F_i(M)$ are the generalized $i$-eigenspaces of $X$ acting on $E(M)$ and $F(M)$ respectively, see also Remark \ref{ch3:rk_df-repl-F-by-E}.
We assume that
\begin{itemize}
    \item[$(a)$] $F=\bigoplus_{i\in\scrF}F_i$ and $E=\bigoplus_{i\in\scrF}E_i$,
    \item[$(b)$] $E_i(\calC_\mu)\subset \calC_{\mu+\alpha_i}$ and  $F_i(\calC_\mu)\subset \calC_{\mu-\alpha_i}$.
\end{itemize}
\end{df}


For exemple, in the case when the quiver $\Gamma_\scrF$ is isomorphic to $\Gamma_e$, then we have $\frakg_\scrF= \widetilde{\mathfrak{sl}}_e$. So we will say "an $\widetilde{\mathfrak{sl}}_e$-categorical representation" instead of "a $\mathfrak{g}_{\scrF}$-categorical representation". However, we should remember that this definition depends on the choice of the set $\scrF$.

\subsection{The category $O$}
\label{ch3:subs_cat-O}

Let $R$ be a deformation ring. Fix an $l$-tuple of positive integers $\nu=(\nu_1,\cdots,\nu_l)$ such that $\sum_{r=1}^l \nu_r=N$.
First we define an $R$-deformed version of the parabolic category $\calO$ for $\widehat{\mathfrak{gl}}_N$.
Recall that we identify the weight lattice $P$ with $\bbZ^N$. We say that $\lambda\in P$ is $\nu$-\emph{dominant} if $\lambda_r>\lambda_{r+1}$ for each $r\in[1,N-1]\backslash\{\nu_1,\nu_1+\nu_2,\cdots,\nu_1+\cdots+\nu_l\}$.
Let $P^\nu$ be the set of $\nu$-dominant weights of $P$. Set also $P^\nu[\mu]=P^\nu\cap P[\mu]$, where $P[\mu]$ is as in (\ref{ch3:eq_def-P[mu]}). Let $O^\nu_R$ be the $R$-linear abelian category of finitely generated $\widehat\bfg_R$-modules $M$ which are weight $\widehat\bfh_R$-modules, and such that the $\widehat\bfp_{R,\nu}$-action on $M$ is locally finite over $R$, and the highest weight of any subquotient of $M$ is of the form $\widetilde\lambda$ with $\lambda\in P^\nu$, where $\widetilde\lambda$ is defined in (\ref{ch3:eq_lambda-tilde}).
Let $O_{\mu,R}^\nu$ be the Serre subcategory of $O^\nu_R$ generated by the modules $\Delta_R^\lambda$ for all $\lambda\in P^\nu[\mu]$. Let $O^{\nu,\Delta}_{\mu,R}\subset O^\nu_{\mu,R}$ be the full subcategory of $\Delta$-filtered modules. 

We will omit the upper index $\nu$ if $\nu=(1,1,\cdots,1)$. Assume $\lambda\in P$. In the case if $R=\bfk$ is a field we denote by $L(\lambda)_\bfk$ the simple quotient of $\Delta(\lambda)_\bfk$. 
In the case if $R$ is local with residue fields $\bfk$, the simple module $L(\lambda)_\bfk\in O_\bfk$ has a simple lift $L(\lambda)_R\in O_R$ such that $L(\lambda)_\bfk=\bfk\otimes_RL(\lambda)_R$ (see \cite[Sec.~2.2]{Fie-cen}). Set also $L^\lambda_R=L(\lambda-\rho)_R$.

\subsection{The choice of $\scrF$}
\label{ch3:subs_param-Hecke}
In this section we define some sets $\scrF$ and $\scrF_\bfk$. 
We will see later that these sets are related with the categorical representations in the categories $O^\nu_{-e,K}$ and $O^\nu_{-e,\bfk}$. 

We assume that $R$ is as in Assumption 1.
As above, we fix an $l$-tuple $\nu=(\nu_1,\cdots,\nu_l)$ of positive integers. Put 
\begin{equation}
\label{ch3:eq_Qr}
Q_r=\exp(2\pi \sqrt{-1}(\nu_r+\tau_r)/\kappa),\qquad r\in[1,l].
\end{equation} 
The canonical homomorphism $R\to \bfk$ maps $q_e$ to $\zeta_e$ and $Q_r$ to
$\zeta_e^{\nu_r}$, where $q_e$ and $\zeta_e$ are as in Section \ref{ch3:subs_def-ring}.

Now, consider the subset $\scrF$ of $R$ given by
$$
\scrF=\bigcup_{r\in \bbZ,t\in[1,l]}\{q_e^rQ_t\}.
$$
Denote by $\scrF_\bfk$ the image of $\scrF$ in $\bfk$ with respect to the surjection $R\to \bfk$. More precisely, we have $\scrF_\bfk=\{\zeta_e^r;~r\in\bbZ\}$ (i.e., the set $\scrF_\bfk$ is the set of $e$th roots of unity in $\bfk$).
Recall from Section \ref{ch3:subs_categ-action} that we consider $\scrF$ (and $\scrF_\bfk)$ as a vertex set of a quiver. The set $\scrF$ is a vertex set of a quiver that is a disjoint union of $l$ infinite linear quivers. The set $\scrF_\bfk$ is a vertex set of a cyclic quiver of length $e$. 


We fix the following identifications
$$
\begin{array}{rcll}
I_e&\simeq& \scrF_\bfk,\quad &i\mapsto \zeta_e^i,\\
\widetilde I&\simeq& \scrF,\quad &(a,b)\mapsto \exp(2\pi \sqrt{-1} (a+\tau_b)/\kappa).\\
\end{array}
$$
In particular, we identify the quivers $\Gamma_e$ and $\widetilde\Gamma$ with the quivers $\Gamma_{\scrF_\bfk}$ and $\Gamma_{\scrF}$ respectively.
We have $\frakg_{\scrF_\bfk}=\widetilde{\mathfrak{sl}}_e$ 
On the other hand, the Lie algebras $\frakg_{\scrF}$ 
is isomorphic to $({\mathfrak{sl}}_\infty)^{\oplus l}$.

Now we consider some special cyclotomic Hecke algebras. Set $H^\nu_{d,R}(q_e)=H^Q_{d,R}(q_e)$, where $Q=(Q_1,\ldots,Q_l)$ is the $l$-tuple defined by (\ref{ch3:eq_Qr}). Similarly, we define the algebras $H^\nu_{d,K}(q_e)$ and $H^\nu_{d,\bfk}(\zeta_e)$ (in the last case, we replace $Q$ by its image in $\bfk$).



It is useful to think of the algebras $H^\nu_{d,K}(q_e)$ and $H^\nu_{d,\bfk}(\zeta_e)$ as cyclotomic KLR algebras defined with respect to the quivers $\Gamma_\scrF$ and $\Gamma_{\scrF_\bfk}$  respectively, see \cite[Cor.~2.18]{Mak-Zuck}. However, we don't use this point of view explicitly in the present paper.

\subsection{The standard representation of $\widetilde{\mathfrak{sl}}_e$}
\label{ch3:subs_stand-rep-aff}

Let $e_i$, $f_i$, $h_i$, be the generators of the complex Lie algebra $\widetilde{\mathfrak{sl}}_e=\mathfrak{sl}_e\otimes\bbC[t,t^{-1}]\oplus\bbC \bm{1}$, here $i\in I_e$.
Let $V_e$ be a $\bbC$-vector spaces with canonical basis $\{v_1,\cdots,v_e\}$ and set $U_e=V_e\otimes \bbC[z,z^{-1}]$. The vector space $U_e$ has a basis $\{u_r;~r\in\bbZ\}$ where $u_{a+eb}=v_a\otimes z^{-b}$ for $a\in[1,e]$, $b\in\bbZ$. It has a structure of an $\widetilde{\mathfrak{sl}}_e$-module such that
$$
f_i(u_r)=\delta_{i\equiv r}u_{r+1},\quad e_i(u_r)=\delta_{i\equiv r-1}u_{r-1}.
$$
Let $\{v'_1,\cdots,v'_{e+1}\}$, $\{u'_r;r\in\bbZ\}$ denote the bases of $V_{e+1}$ and $U_{e+1}$.

Fix an integer $0\leqslant k<e$. Consider the following inclusion of vector spaces
$$
V_e\subset V_{e+1}, ~v_r\mapsto
\left\{\begin{array}{ll}
v'_r &\mbox{ if }r\leqslant k,\\
v'_{r+1} &\mbox{ if }r>k.
\end{array}\right.
$$
It yields an inclusion $\mathfrak{sl}_e\subset\mathfrak{sl}_{e+1}$ such that
$$
e_r\mapsto
\left\{\begin{array}{rl}
e_r &\mbox{ if }r\in[1,k-1],\\
{[e_k,e_{k+1}]} &\mbox{ if }r=k,\\
e_{r+1} &\mbox{ if }r\in[k+1,e-1],
\end{array}\right.
$$
$$
f_r\mapsto
\left\{\begin{array}{rl}
f_r &\mbox{ if }r\in[1,k-1],\\
{[f_{k+1},f_k]} &\mbox{ if }r=k,\\
f_{r+1} &\mbox{ if }r\in[k+1,e-1],
\end{array}\right.
$$
$$
h_r\mapsto
\left\{\begin{array}{rl}
h_r &\mbox{ if }r\in[1,k-1],\\
h_k+h_{k+1} &\mbox{ if }r=k,\\
h_{r+1} &\mbox{ if }r\in[k+1,e-1].
\end{array}\right.
$$

This inclusion lifts uniquely to an inclusion  $\widetilde{\mathfrak{sl}}_e\subset\widetilde{\mathfrak{sl}}_{e+1}$ such that
$$
e_0\mapsto
\left\{\begin{array}{rl}
e_0 &\mbox{ if }k\ne 0,\\
{[e_0,e_1]} &\mbox{ else},\\
\end{array}\right.
$$
$$
f_0\mapsto
\left\{\begin{array}{rl}
f_0 &\mbox{ if }k\ne 0,\\
{[f_1,f_0]} &\mbox{ else},\\
\end{array}\right.
$$
$$
h_0\mapsto
\left\{\begin{array}{rl}
h_0 &\mbox{ if }k\ne 0,\\
{h_0+h_1} &\mbox{ else}.\\
\end{array}\right.
$$

Consider the inclusion $U_e\subset U_{e+1}$ such that $u_{r}\mapsto u'_{\Upsilon(r)}$.

\smallskip
\begin{lem}
The embeddings $V_e\subset V_{e+1}$ and $U_e\subset U_{e+1}$ are compatible with the actions of $\mathfrak{sl}_e\subset \mathfrak{sl}_{e+1}$ and $\widetilde{\mathfrak{sl}}_e\subset \widetilde{\mathfrak{sl}}_{e+1}$ respectively.
\qed
\end{lem}

Set $\wedge^\nu U_e=\wedge^{\nu_1}U_e\otimes\cdots\otimes \wedge^{\nu_l} U_e$.
For each
$\lambda\in P^\nu$ define the following element in $\wedge^\nu U_e$:
$$
\wedge_\lambda^\nu=(u_{\lambda_1}\wedge\cdots\wedge u_{\lambda_{\nu_1}})\otimes\cdots\otimes(u_{\lambda_1+\cdots+\lambda_{l-1}+1}\wedge\cdots\wedge u_{\lambda_{\nu_1}+\cdots+\lambda_{\nu_l}}).
$$

The obvious $\widetilde{\mathfrak{sl}}_e$-action on $U_e$ yields an   $\widetilde{\mathfrak{sl}}_e$-action on $\wedge^\nu U_e$. We identify the abelian group $X_e/\bbZ(\varepsilon_1+\cdots+\varepsilon_e)$ with the weight lattice of $\mathfrak{sl}_e$. In particular each element $\mu\in X_e$ yields a weight of $\mathfrak{sl}_e$. For each $\mu\in X_e[N]$ let $(\wedge^\nu U_e)_\mu$ be the weight space in $\wedge^\nu U_e$ corresponding to $\mu$.

\subsection{Categorical representation in the category $O$}
\label{ch3:sec-categ-rep-O}

Set $O^\nu_{-e,R}=\bigoplus_{\mu\in X_e[N]}O^\nu_{\mu,R}$ and similarly for $O^{\nu,\Delta}_{-e,R}$. Now we define a representation datum in the category $O^{\nu,\Delta}_{-e,R}$. See \cite[Sec.~5.4]{RSVV} for more details. In Sections \ref{ch3:sec-categ-rep-O}-\ref{ch3:subs-cat-calA}  we assume that $R$ is as in Assumption 2.
For an exact category $\calC$ denote by $[\calC]$ its complexified Grothendieck group. The following proposition holds, see \cite{RSVV}.
\begin{prop}
\label{ch3:prop_functors-on-O-gen}
There is a pair of exact endofunctors $E$, $F$ of $O^{\nu,\Delta}_{-e,R}$ such that the following properties hold.
\begin{itemize}
    \item[$(a)$] The functors $E$, $F$
commute with the base changes $K\otimes_R\bullet$, $\bfk\otimes_R\bullet$.
    \item[$(b)$] $(O^{\nu,\Delta}_{-e,R},E,F)$ admits a representation datum structure (with respect to $q=q_e$).
    \item[$(c)$] The pair of functors $(E,F)$ is biadjoint. It extends to a pair of biadjoint functors $O^{\nu}_{-e,R}\to O^{\nu}_{-e,R}$ if $R$ is a field.
    \item[$(d)$] There are decompositions $E=\bigoplus_{i\in I_e}E_i$, $F=\bigoplus_{i\in I_e}F_i$ such that
$$
E_i(O_{\mu,R}^{\nu,\Delta})\subset O_{\mu+\alpha_i,R}^{\nu,\Delta},\qquad F_i(O_{\mu,R}^{\nu,\Delta})\subset O_{\mu-\alpha_i,R}^{\nu,\Delta}.
$$
    \item[$(e)$] There is a vector space isomorphism $[O^{\nu,\Delta}_{\mu,R}]\simeq (\wedge^\nu U_e)_\mu$ such that the functors $F_i$, $E_i$ act on $[O^{\nu,\Delta}_{-e,R}]=\bigoplus_{\mu\in X_e[N]}[O^{\nu,\Delta}_{\mu,R}]$ as the standard generators $e_i$, $f_i$ of $\widetilde{\mathfrak{sl}}_e$.
    \item[$(f)$] If $R=\bfk$ with the trivial deformation ring structure, then $E_i$, $F_i$ yield a categorical representation of $\widetilde{\mathfrak{sl}}_e$ in $O^\nu_{-e,\bfk}$ (with respect to the set $\scrF$ as in Section \ref{ch3:subs_param-Hecke}).
\qed
\end{itemize}

\end{prop}

\smallskip
Fix $k\in [0,e-1]$. Recall the map $\Upsilon\colon P\to P$ from Section \ref{ch3:subs_ext-aff} and the map $\phi\colon X_e\to X_{e+1}$ from (\ref{ch3:eq_phi(mu)}). 
Set $\mu'=\mu-\alpha_k$ and $\overline\mu=\phi(\mu)$.
Set, $\overline\mu^0=\overline\mu-\overline\alpha_{k}$ and $\overline\mu'=\overline\mu-\overline\alpha_{k}-\overline\alpha_{k+1}$. Note that $\Upsilon(P[\mu])\subset P[\overline\mu]$.
For $k\ne 0$ we have
$$
\begin{array}{llllll}
\mu&=&(\mu_1,\cdots,\mu_k,&&\mu_{k+1},&\cdots,\mu_e),\\
\mu '&=&(\mu_1,\cdots,\mu_k-1,&&\mu_{k+1}+1,&\cdots,\mu_e),\\
\overline\mu&=&(\mu_1,\cdots,\mu_k,&0,&\mu_{k+1},&\cdots,\mu_e),\\
\overline\mu^0&=&(\mu_1,\cdots,\mu_k-1,&1,&\mu_{k+1},&\cdots,\mu_e),\\
\overline\mu '&=&(\mu_1,\cdots,\mu_k-1,&0,&\mu_{k+1}+1,&\cdots,\mu_e).\\
\end{array}
$$
For an $e$-tuple $\bfa=(a_1,\cdots,a_e)$ of non-negative integers we set $1_\bfa=(1^{a_1},\cdots,e^{a_e})$. Note that we have
\begin{equation}
\label{ch3:eq_Upsilon(1)}
\Upsilon(1_{\mu})=1_{\overline\mu},\qquad \Upsilon(1_{\mu'})=1_{\overline\mu'}.
\end{equation}

\smallskip
\begin{rk}
\label{ch3:rk_choice-of-1_mu}
The set $P[\bfa]$ is a $\widehat W$-orbit in $P^{(e)}$. It is a union of $\widetilde W$-orbits and each of them contains a unique $e$-anti-dominant weight. By definition, the weight $1_\bfa\in P[\bfa]$ is $e$-anti-dominant. However, there is no canonical way to choose an anti-dominant element in $P[\bfa]$.
In the case $k=0$ we need to change our convention and to set $1_\bfa=(0^{a_e},1^{a_e},\cdots,(e-1)^{a_{e-1}})$. This change is necessary to have (\ref{ch3:eq_Upsilon(1)}).
\end{rk}

First, assume that $l=N$ and $\nu=(1,1,\cdots,1)$.

\begin{lem}
There is an equivalence of categories $\theta_{\mu}^{\overline\mu}\colon O_{\mu,R}\to O_{\overline\mu,\overline R}$ such that $\theta_\mu^{\overline\mu}(\Delta_{R}^{\lambda})\simeq\Delta_{\overline R}^{\Upsilon(\lambda)}$. It restricts to an equivalence of categories $\theta_{\mu}^{\overline\mu}\colon O^{\Delta}_{\mu,R}\to O^{\Delta}_{\overline\mu,\overline R}$.
\end{lem}
\begin{proof}[Proof]
For each $n\in \bbZ$ the weight $\pi^n(1_\mu)$ is $e$-anti-dominant.
Let $\calO_{\pi^n(1_\mu),R}\subset O_{\mu,R}$ be the Serre subcategory generated by the Verma modules of the form $\Delta_{R}^{w\pi^n(1_\mu)}$ with $w\in\widetilde W$. 
We have
\begin{equation}
\label{ch3:eq_dec-O-What-Wtilde}
O_{\mu,R}=\bigoplus_{n\in\bbZ}\calO_{\pi^n(1_\mu),R}.
\end{equation}
The weights $\pi^n(1_\mu)\in P^{(e)}$ and $\pi^n(1_{\overline\mu})\in P^{(e+1)}$ have the same stabilizers in $\widetilde W$. Thus by \cite[Thm.~11]{Fie-str} (see also \cite[Prop.~5.24]{RSVV}) we have an equivalence of categories
$$
\calO_{\pi^n(1_\mu),R}\simeq \calO_{\pi^n(1_{\overline\mu}),\overline R},\qquad \Delta_{R}^{w\pi^n(1_\mu)}\mapsto \Delta_{\overline R}^{w\pi^n(1_{\overline\mu})} ~~\forall w\in\widetilde W.
$$
Taking the sum by all $n\in\bbZ$ we get an equivalence of categories
$$
\theta_{\mu}^{\overline\mu}\colon O_{\mu,R}\simeq O_{\overline\mu,\overline R},\qquad \Delta_{R}^{w(1_\mu)}\mapsto \Delta_{\overline R}^{w(1_{\overline\mu})} ~~\forall w\in\widehat W.
$$
Recall that we have $\Upsilon(1_\mu)=1_{\overline\mu}$.
Thus by $\widehat W$-invariance of $\Upsilon$ we get
$$
\theta_\mu^{\overline\mu}(\Delta_{R}^{\lambda})\simeq\Delta_{\overline R}^{\Upsilon(\lambda)}~~\forall \lambda\in P[\mu].
$$
\end{proof}

\begin{rk}
Notice that \cite{Fie-str} yields an equivalence of categories over a field. It is explained in \cite{RSVV} how to get from it an equivalence of categories $O^\Delta_R$. First, comparing the endomorphisms of projective generators one gets an equivalence of the abelian categories $O_R$. Then, comparing the highest weight structure in both sides, we deduce an equivalence of additive categories $O^\Delta_R$.

\end{rk}

\smallskip
The equivalence $\theta_{\mu}^{\overline\mu}$ restricts to equivalences $O^\nu_{\mu,R}\simeq O^\nu_{\overline\mu,\overline R}$ and $O^{\nu,\Delta}_{\mu,R}\simeq O^{\nu,\Delta}_{\overline\mu,\overline R}$ for each parabolic type $\nu$, see \cite[Sec.~5.7.2]{RSVV}. We will also call this equivalence $\theta_{\mu}^{\overline\mu}$. We obtain equivalences of categories $\theta_{\overline\mu'}^{\mu'}\colon O^\nu_{\overline\mu',\overline R}\simeq O^\nu_{\mu',R}$ and $\theta_{\overline\mu'}^{\mu'}\colon O^{\nu,\Delta}_{\overline\mu',\overline R}\simeq O^{\nu,\Delta}_{\mu',R}$ in a similar way.

\smallskip
\begin{conj}
\label{ch3:conj_F_k-decomp}
There are the following commutative diagrams
$$
\begin{diagram}
\node{O^{\nu,\Delta}_{\overline\mu,\overline R}} \arrow{e,t}{\overline F_k}
\node{O^{\nu,\Delta}_{\overline\mu^0,\overline R}} \arrow{e,t}{\overline F_{k+1}} 
\node{\overline O^{\nu,\Delta}_{\mu',\overline R}} \arrow{s,r}{\theta_{\overline\mu'}^{\mu'}} \\
\node{O^{\nu,\Delta}_{\mu,R}} \arrow{n,l}{\theta_\mu^{\overline\mu}} 
\arrow[2]{e,b}{F_k} \node[2]{O^{\nu,\Delta}_{\mu',R}}
\end{diagram}
$$
and
$$
\begin{diagram}
\node{O^{\nu,\Delta}_{\overline\mu,\overline R}} \arrow{s,l}{\theta^\mu_{\overline\mu}}
\node{O^{\nu,\Delta}_{\overline\mu^0,\overline R}} \arrow{w,t}{\overline E_k}   
\node{O^{\nu,\Delta}_{\overline\mu',\overline R}} \arrow{w,t}{\overline E_{k+1}}  \\
\node{O^{\nu,\Delta}_{\mu,R}}     \node[2]{O^{\nu,\Delta}_{\mu',R}} \arrow{n,r}{\theta^{\overline\mu'}_{\mu'}} \arrow[2]{w,b}{E_k} 
\end{diagram}.
$$
\end{conj}

\subsection{The commutativity in the Grothendieck groups}
\label{ch3:subs_comm-Groth}

We have the following commutative diagram of vector spaces
$$
\begin{CD}
[O^{\nu,\Delta}_{-(e+1),\overline R}] @>>> \wedge^\nu U_{e+1}\\
@A{\oplus_\mu\theta_\mu^{\overline\mu}}AA @AAA\\
[O^{\nu,\Delta}_{-e,R}] @>>> \wedge^\nu U_e,
\end{CD}
$$
where the horizontal maps are respectively the isomorphisms of $\widetilde{\mathfrak{sl}}_e$-modules and $\widetilde{\mathfrak{sl}}_{e+1}$-modules from Proposition \ref{ch3:prop_functors-on-O-gen} $(e)$, the right vertical map is given by the injection $U_{e}\to U_{e+1}$ in Section \ref{ch3:subs_stand-rep-aff}. Moreover, the right vertical map is a morphism of $\widetilde{\mathfrak{sl}}_e$-modules where $\wedge^\nu U_{e+1}$ is viewed as an $\widetilde{\mathfrak{sl}}_e$-module via the inclusion $\widetilde{\mathfrak{sl}}_e\subset \widetilde{\mathfrak{sl}}_{e+1}$ introduced in Section \ref{ch3:subs_stand-rep-aff}. Thus $\oplus_\mu\theta_\mu^{\overline\mu}\colon [O^{\nu,\Delta}_{-e,R}] \to [O^{\nu,\Delta}_{-(e+1),\overline R}]$ is a morphism of $\widetilde{\mathfrak{sl}}_e$-modules which intertwines

\smallskip
\begin{itemize}
\item[\textbullet] $[E_r]$ with $[\overline E_r]$,\quad $[F_r]$ with $[\overline F_r]$ if $r\in[1,k-1]$,
\item[\textbullet] $[E_k]$ with $[\overline E_k\overline E_{k+1}]-[\overline E_{k+1}\overline E_{k}]$,\quad $[F_k]$ with $[\overline F_{k+1}\overline F_k]-[\overline F_{k}\overline F_{k+1}]$,
\item[\textbullet] $[E_r]$ with $[\overline E_{r+1}]$,\quad $[F_r]$ with $[\overline F_{r+1}]$ if $r\in[k+1,e-1]$.

\smallskip
In particular, we see that the diagrams from Conjecture \ref{ch3:conj_F_k-decomp} commute at the level of Grothendieck groups. Since there is no good notion of projective functors in the affine category $\calO$, this is not enough to prove our conjecture.

\end{itemize}

\subsection{Partitions}
A \emph{partition} of an integer $n\geqslant 0$ is a tuple of positive integers $(\lambda_1,\cdots,\lambda_s)$ such that $\lambda_1\geqslant \lambda_2\geqslant\cdots\geqslant \lambda_s$ and $\sum_{t=1}^s\lambda_t=n$. Denote by $\calP_n$ the set of all partitions of $n$ and set $\calP=\coprod_{n\in\bbN}\calP_n$. For a partition $\lambda=(\lambda_1,\cdots,\lambda_s)$ of $n$, we set $|\lambda|=n$ and $\ell(\lambda)=s$.
An $l$-\emph{partition} of an integer $n\geqslant 0$ is an $l$-tuple $\lambda=(\lambda^{1},\cdots,\lambda^{l})$ of partitions of integers $n_1,\cdots,n_l\geqslant 0$ such that $\sum_{t=1}^ln_t=n$. Let $\calP^l_n$ be the set of all $l$-partitions of $n$ and set $\calP^l=\coprod_{n\in\bbN}\calP^l_n$.

A partition $\lambda$ can be represented by a Young diagram $Y(\lambda)$ and an $l$-partition $\lambda=(\lambda^1,\cdots,\lambda^l)$ by an $l$-tuple of Young diagrams $Y(\lambda)=(Y(\lambda^1),\cdots,Y(\lambda^l))$.

Let $\lambda\in\calP^l$ be an $l$-partition. For a box $b\in Y(\lambda)$ situated in the $i$th row, $j$th column of the $r$th component we define its \emph{residue} $\Res_\nu(b)\in I$ as $\nu_r+j-i$ $\mod~e$ and its \emph{deformed residue} $\widetilde\Res_\nu(b)\in \widetilde I$ as $(\nu_r+j-i,r)$. 
Set
$$
\Res_\nu(\lambda)=\sum_{b\in Y(\lambda)}\alpha_{\Res_\nu(b)}\in Q^+_e,\quad\widetilde\Res_\nu(\lambda)=\sum_{b\in Y(\lambda)}\widetilde\alpha_{\widetilde\Res_\nu(b)}\in Q^+_{\widetilde I}.
$$

Now for $\alpha\in Q^+_e$ and $\widetilde\alpha\in Q^+_{\widetilde I}$ set
$$
\calP^l_\alpha=\{\lambda\in\calP^l;~\Res_\nu(\lambda)=\alpha\},\qquad \calP^l_{\widetilde\alpha}=\{\lambda\in\calP^l;~\widetilde\Res_\nu(\lambda)=\widetilde\alpha\}.
$$
This notation depends on $\nu$. We may write $\calP^l_{\alpha,\nu}$ and $\calP^l_{\widetilde\alpha,\nu}$ to specify $\nu$.
We have decompositions
$$
\calP^l_d=\bigoplus_{\alpha\in Q^+_e, |\alpha|=d}\calP^l_\alpha,\qquad \calP^l_\alpha=\bigoplus_{\widetilde\alpha\in Q^+_{e+1}, \pi_e(\widetilde\alpha)=\alpha}\calP^l_{\widetilde\alpha}.
$$

\subsection{The category $\bfA$}
\label{ch3:subs_cat-bfA}
Let $\calP^\nu_d\subset \calP^l_d$ be the subset of the elements $\lambda=(\lambda^1,\cdots,\lambda^l)$ such that $\ell(\lambda^r)\leqslant \nu_r$ for each $r\in[1,l]$.
We can view each $\lambda\in \calP^\nu_d$ as the weight in $P$ given by
$$
(\lambda^1_1,\cdots,\lambda^1_{\ell(\lambda^1)},0^{\nu_1-\ell(\lambda^1)},\lambda^2_1,\cdots,\lambda^2_{\ell(\lambda^2)},0^{\nu_2-\ell(\lambda^2)},\cdots,\lambda^l_1,\cdots,\lambda^l_{\ell(\lambda^l)},0^{\nu_l-\ell(\lambda^l)}).
$$
We abbreviate $\Delta[\lambda]_R=\Delta^{\lambda+\rho_\nu}_R$.

\smallskip
\begin{df}
Let $\bfA^\nu_R[d] \subset O^\nu_{-e,R}$ be the Serre subcategory generated by the modules $\Delta[\lambda]_R$ with $\lambda\in \calP^\nu_d$, see Section \ref{ch3:subs_aff-Lie}. Denote by $\bfA^{\nu,\Delta}_R[d]$ the full subcategory of $\Delta$-filtered modules in $\bfA^\nu_R[d]$.
\end{df}

\smallskip
The restriction of the functor $F$ to the subcategory $\bfA^{\nu,\Delta}_R[d]$ yields a functor $F\colon\bfA^{\nu,\Delta}_R[d]\to\bfA^{\nu,\Delta}_R[d+1]$. However, it is not true that $E(\bfA^{\nu,\Delta}_R[d+1])\subset \bfA^{\nu,\Delta}_R[d]$. Nevertheless, we can define a functor $E\colon\bfA^{\nu,\Delta}_R[d+1]\to \bfA^{\nu,\Delta}_R[d]$ that is left adjoint to $F\colon \bfA^{\nu,\Delta}_R[d]\to\bfA^{\nu,\Delta}_R[d+1]$, see \cite[Sec.~5.9]{RSVV}. This can be done in the following way. Let $h$ be the inclusion functor from $\bfA^{\nu,\Delta}_R[d]$ to ${O^{\nu,\Delta}_{-e,R}}$. Abusing the notation, we will use the same symbol for the inclusion functor from $\bfA^{\nu,\Delta}_R[d+1]$ to ${O^{\nu,\Delta}_{-e,R}}$. Let $h^*$ be the left adjoint functor to $h$. We define the functor $E$ for the category $\bfA^{\nu,\Delta}_R$ as $h^*Eh$.

There is a decomposition $\bfA^\nu_R[d]=\bigoplus_{\alpha\in Q^+_e,|\alpha|=d}\bfA^\nu_R[\alpha]$, where $\bfA^\nu_R[\alpha]$ is the Serre subcategory of $\bfA^\nu_R[d]$ generated by the Verma modules $\Delta[\lambda]_R$ such that $\lambda\in\calP^l_\alpha$. The functors $E$, $F$ admit decompositions
$$
E=\bigoplus_{i\in I_e}E_i,\qquad F=\bigoplus_{i\in I_e}F_i
$$
such that for each $\alpha\in Q^+_e$ and $i\in I_e$ we have
$$
E_i(\bfA^{\nu,\Delta}_R[\alpha])\subset \bfA^{\nu,\Delta}_R[\alpha-\alpha_i], \qquad F_i(\bfA^{\nu,\Delta}_R[\alpha])\subset \bfA^{\nu,\Delta}_R[\alpha+\alpha_i].
$$

Note that the functor $E_i$ for the category $\bfA^{\nu,\Delta}_R$ is the restriction of the functor $E_i$ for the category $O^{\nu,\Delta}_{-e,R}$ if $i\ne 0$. Thus for $i\ne 0$ the pair of functors $(E_i,F_i)$ for the category $\bfA^{\nu,\Delta}_R$ is biadjoint. But we have only a one-side adjunction $(E_0,F_0)$. Note also that if $R$ is a field, then we can define the functors as above (with the same adjunction properties) for the category $\bfA^{\nu}_R$ instead of $\bfA^{\nu,\Delta}_R$.

\smallskip
Let us write $\emptyset$ for the empty $l$-partition. Note that $\Delta[\emptyset]_R=\Delta^{\rho_\nu}_R$ is the Verma module of highest weight $\widetilde{\rho_\nu-\rho}$. Since, $\rho_\nu$ lies in $P[\bfwt_e(\rho_\nu)]$, we have $\Delta[\emptyset]_R\in O^\nu_{\bfwt_e(\rho_\nu),R}$. More generally, fix an element $\alpha=\sum_{i\in I_e}d_i\alpha_i$ in $Q^+_e$. Put $\mu=\bfwt_e(\rho_\nu)-\alpha\in X_e$. See \cite[Sec.~2.3]{RSVV} for the definition of a highest weight category over a local ring. The following proposition holds, see \cite[Sec.~5.5]{RSVV}.
\begin{prop}
The category $\bfA^\nu_R[\alpha]$ is a full subcategory of $O^\nu_{\mu,R}$ that is a highest weight category.
\qed
\end{prop}
For $\lambda\in\calP^l_d$ let $P[\lambda]_{R}$, $\nabla[\lambda]_{R}$ and $T[\lambda]_{R}$ be the projective, costandard and the tilting objects in $\bfA_R^\nu[d]$ with parameter $\lambda$, see \cite[Prop.~2.1]{RSVV}.

\subsection{The change of level for $\bfA$}
\label{ch3:subs_rank-ch-A}

For $\lambda_1,\lambda_2\in P$ we write $\lambda_1\geqslant\lambda_2$ if $(\lambda_1)_r\geqslant (\lambda_2)_r$ for each $r\in[1,N]$. Here, $(\lambda_i)_r$ is the $r$th entry of $\lambda_i$ for each $r$. We identify $Q_e$ with a sublattice of $X^\chi_e$ via the map $\iota^\chi$ defined in Section \ref{ch3:subs_KM-quiv}.

\smallskip
\begin{lem}
\label{ch3:lem_lam1>=lam2}
$(a)$ For each $\lambda_1,\lambda_2\in P$ we have $\bfwt^\chi_e(\lambda_1)-\bfwt^\chi_e(\lambda_2)\in Q_e$.

$(b)$ If we also have $\lambda_1\leqslant\lambda_2$, then $\bfwt^\chi_e(\lambda_1)-\bfwt^\chi_e(\lambda_2)\in Q^+_e$.
\end{lem}
\begin{proof}[Proof]
It is enough to assume that we have $\lambda_1=\lambda_2-\epsilon_r$ for some $r\in[1,N]$. In this case we have $\bfwt^\chi_e(\lambda_1)-\bfwt^\chi_e(\lambda_2)=\alpha_i$, where $i\in I_e$ is the residue of the integer $(\lambda_1)_r$ modulo $e$.
\end{proof}

Let $\phi\colon Q_e\to Q_{e+1}$ and $\phi\colon X_e\to X_{e+1}$ be as in Section \ref{ch3:subs_not-quiv-I-Ibar} (see also Section \ref{ch3:subs_not-e-e+1}) and $\Upsilon$ be as in (\ref{ch3:eq_upsilon}). 
Set $\overline\alpha=\phi(\alpha)\in Q_{e+1}$, $\overline\mu=\phi(\mu)\in X_{e+1}$ and  $\beta=\bfwt^\chi_{e+1}(\rho_\nu)-\bfwt^\chi_{e+1}(\Upsilon(\rho_\nu))$. By Lemma \ref{ch3:lem_lam1>=lam2}, we have $\beta\in Q^+_{e+1}$.

\smallskip
\begin{prop}
\label{ch3:prop_equiv-A-e-e+1}
The equivalence of categories $\theta_{\mu}^{\overline\mu}$ takes the subcategory $\bfA_R^\nu[\alpha]$ of $O^\nu_{\mu,R}$ to the subcategory $\bfA^\nu_{\overline R}[\beta+\overline\alpha]$ of $O^\nu_{\overline\mu,\overline R}$.
\qed
\end{prop}

\smallskip
First, we prove the following lemma.

\smallskip
\begin{lem}
\label{ch3:lem_difference-weights}
If $\lambda_1,\lambda_2\in P$, then
$$
\bfwt_{e+1}^\chi(\Upsilon(\lambda_1))-\bfwt_{e+1}^\chi(\Upsilon(\lambda_2))=\phi(\bfwt_{e}^\chi(\lambda_1)-\bfwt_{e}^\chi(\lambda_2)).
$$

\end{lem}
\begin{proof}[Proof of Lemma \ref{ch3:lem_difference-weights}]
It is enough to prove the statement in the case where we have $\lambda_1=\lambda_2-\epsilon_r$ for some $r\in[1,N]$. In this case we have $\bfwt_{e}^\chi(\lambda_1)-\bfwt_{e}^\chi(\lambda_2)=\alpha_i$, where $i$ is the residue of $(\lambda_1)_r$ modulo $e$. If $i\ne k$ then we have $\bfwt_{e+1}^\chi(\Upsilon(\lambda_1))-\bfwt_{e+1}^\chi(\Upsilon(\lambda_2))=\overline\alpha_{i'}=\phi(\alpha_{i})$, where $i'$ is as in (\ref{ch3:eq_i'}). 
If $i=k$ then we have $\bfwt_{e+1}^\chi(\Upsilon(\lambda_1))-\bfwt_{e+1}^\chi(\Upsilon(\lambda_2))=\overline\alpha_{k}+\overline\alpha_{k+1}=\phi(\alpha_k)$.
\end{proof}
\begin{proof}[Proof of Proposition \ref{ch3:prop_equiv-A-e-e+1}]
By definition, $\bfA^\nu_R[\alpha]\subset O^\nu_{\mu,R}$ is the Serre subcategory of $O_{\mu,R}^\nu$ generated by $\Delta_{R}^\lambda$ such that the weight $\lambda\in P^\nu$ satisfies $\lambda\geqslant\rho_\nu$ and $\bfwt^\chi_e(\rho_\nu)-\bfwt^\chi_e(\lambda)=\alpha$. Here $\geqslant$ is the order defined before Lemma \ref{ch3:lem_lam1>=lam2}.

As $\theta_{\mu}^{\overline\mu}(\Delta_R^\lambda)$ is isomorphic to $\Delta^{\Upsilon(\lambda)}_{\overline R}$, Lemma \ref{ch3:lem_difference-weights} implies that $\theta_{\mu}^{\overline\mu}(\bfA^\nu_R[\alpha])$ is the Serre subcategory of $O^\nu_{\overline\mu,\overline R}$ generated by $\Delta^{\overline\lambda}_{\overline R}$ for $\overline\lambda\in P^\nu$ such that $\overline\lambda\geqslant\Upsilon(\rho_\nu)$ and $\bfwt^\chi_{e+1}(\Upsilon(\rho_\nu))-\bfwt^\chi_{e+1}(\overline\lambda)=\overline\alpha$.

Moreover, for each module $\Delta^{\overline\lambda}_{\overline R}\in O^\nu_{\overline\mu,\overline R}$, the weight $\overline\lambda$ has no coordinates that are congruent to $k+1$ modulo $e+1$.
Then $\overline\lambda$ satisfies $\overline\lambda\geqslant \rho_\nu$ if and only if it satisfies $\overline\lambda\geqslant \Upsilon(\rho_\nu)$. We have $\bfwt^\chi_{e+1}(\rho_\nu)-\bfwt^\chi_{e+1}(\Upsilon(\rho_\nu))=\beta$. Thus $\theta_{\mu}^{\overline\mu}(\bfA^\nu_R[\alpha])$ is the Serre subcategory of $O^\nu_{\overline\mu,\overline R}$ generated by the modules $\Delta^{\overline\lambda}_{\overline R}$ where $\overline\lambda$ runs over the set of all $\overline\lambda\in P^\nu$ such that $\overline\lambda\geqslant\rho_\nu$ and $\bfwt^\chi_{e+1}(\rho_\nu)-\bfwt^\chi_{e+1}(\overline\lambda)=\overline\alpha+\beta$. This implies $\theta_{\mu}^{\overline\mu}(\bfA_R^\nu[\alpha])=\bfA_{\overline R}^\nu[\overline\alpha+\beta]$.
\end{proof}

\subsection{The category $\calA$}
\label{ch3:subs-cat-calA}

From now on, to avoid cumbersome notation we will use the following abbreviations.
First, for each $\alpha\in Q^+_e$ we set
$$
\calA^\nu_{R}[\alpha]=\bfA^\nu_{\overline R}[\beta+\overline\alpha],\qquad \calA^\nu_R[d]=\bigoplus_{|\alpha|=d}\calA^\nu_R[\alpha], \qquad \calA^\nu_R=\bigoplus_{d\in\bbN}\calA^\nu_R[d].
$$
Next, we define the endofunctors $E_0,\cdots,E_{e-1}$, $F_0,\cdots, F_{e-1}$ of $\calA^{\nu,\Delta}_R$ (or of $\calA^{\nu}_R$ is $R$ is a field) by
$$
F_0=\restr{\overline F_0}{\calA^{\nu,\Delta}_R},\cdots, F_{k-1}=\restr{\overline F_{k-1}}{\calA^{\nu,\Delta}_R},\quad F_k=\restr{\overline F_{k+1}\overline F_k}{\calA^{\nu,\Delta}_R},
$$
\begin{equation}
\label{ch3:eq_E-from-E-bar-A}
F_{k+1}=\restr{\overline F_{k+2}}{\calA^{\nu,\Delta}_R}, \cdots , F_{e-1}=\restr{\overline F_e}{\calA^{\nu,\Delta}_R},
\end{equation}
$$
E_0=\restr{\overline E_0}{\calA^{\nu,\Delta}_R},\cdots, E_{k-1}=\restr{\overline E_{k-1}}{\calA^{\nu,\Delta}_R}, \quad E_k=\restr{\overline E_{k}\overline F_{k+1}}{\calA^{\nu,\Delta}_R},
$$
$$
E_{k+1}=\restr{\overline E_{k+2}}{\calA^{\nu,\Delta}_R}, \cdots , E_{e-1}=\restr{\overline E_e}{\calA^{\nu,\Delta}_R}.
$$

We precise that here we use the functor $\overline E_0$ for the category  $\bfA^\nu_{\overline R}$. This functor is not just the naive restriction of the functor $\overline E_0$ for the category $O^\nu_{-(e+1),\overline R}$, see Section \ref{ch3:subs_cat-bfA}.

By definition, we have $E_i(\calA^{\nu,\Delta}_R[\alpha])\subset \calA^{\nu,\Delta}_R[\alpha-\alpha_i]$ and $F_i(\calA^{\nu,\Delta}_R[\alpha])\subset \calA^{\nu,\Delta}_R[\alpha+\alpha_i]$. Consider the endofunctors $E=\bigoplus_{i\in I_e}E_i$ and $F=\bigoplus_{i\in I_e}F_i$ of $\calA^{\nu,\Delta}_R$. We have $E(\calA^{\nu,\Delta}_R[d])\subset \calA^{\nu,\Delta}_R[d-1]$ and $F(\calA^{\nu,\Delta}_R[d])\subset \calA^{\nu,\Delta}_R[d+1]$.

Let $\theta_\alpha\colon\bfA^{\nu}_R[\alpha]\to\calA^{\nu}_R[\alpha]$ be the equivalence of categories in Proposition \ref{ch3:prop_equiv-A-e-e+1}. Taking the sum over $\alpha$'s, we get an equivalence $\theta\colon\bfA^{\nu}_R\to \calA_R^{\nu}$ and an equivalence $\theta_d\colon\bfA^{\nu}_R[d]\to \calA_R^{\nu}[d]$. Moreover, we have the following commutative diagrams of Grothendieck groups
\begin{equation}
\label{ch3:eq_comm-diag-Groth-A}
\begin{CD}
[\bfA^{\nu,\Delta}_R] @>{F_i}>> [\bfA^{\nu,\Delta}_R]\\
@V{\theta}VV         @V{\theta}VV\\
[\calA^{\nu,\Delta}_R] @>{F_i}>> [\calA^{\nu,\Delta}_R],\\
\end{CD}
\qquad
\begin{CD}
[\bfA^{\nu,\Delta}_R] @<{E_i}<< [\bfA^{\nu,\Delta}_R]\\
@V{\theta}VV         @V{\theta}VV\\
[\calA^{\nu,\Delta}_R] @<{E_i}<< [\calA^{\nu,\Delta}_R],\\
\end{CD}
\end{equation}
see Section \ref{ch3:subs_comm-Groth}.

For $\lambda\in\calP^\nu_d$ we set $\overline\Delta[\lambda]_R=\Delta_{\overline R}^{\Upsilon(\rho_\nu+\lambda)}\in\calA_R^\nu[d]$. By construction we have $\theta_d(\Delta[\lambda]_R)\simeq \overline\Delta[\lambda]_R$.
Let $\overline P_R[\lambda]$,  $\overline\nabla_R[\lambda]$ and $\overline T_R[\lambda]$ be the projective, the costandard and the tilting object with parameter $\lambda$ in $\calA^\nu_R[d]$.

\subsection{The categorical representation in the category $O$ over the field $K$}
Assume that $R$ is as in Assumption 1.
In this section we compare the categorical representation in $O^\nu_{-e,K}$ with the representation datum in $O^{\nu,\Delta}_{-e,R}$ introduced above. Recall that the categorical representation in $O^{\nu}_{-e,K}$ is defined with respect to the set $\scrF$ in Section \ref{ch3:subs_param-Hecke}.

 First, for each $\lambda\in P$ we define the following weight in $X^+_{\widetilde I}$
$$
\widetilde\bfwt_{e}(\lambda)=\sum_{r=1}^l\left(\sum_{t=\nu_1+\cdots+\nu_{r-1}+1}^{\nu_1+\cdots+\nu_{r}}\widetilde\varepsilon_{(\lambda_t,r)}\right).
$$

For $\widetilde\mu\in X^+_{\widetilde I}$ let $O^\nu_{\widetilde\mu,K}$ be the Serre subcategory of $O^\nu_{-e,K}$ generated by the Verma modules $\Delta^\lambda_K$ such that $\widetilde\bfwt_{e}(\lambda)=\widetilde\mu$. This decomposition is a refinement of the decomposition $O_{-e,K}^\nu=\bigoplus_{\mu\in X_e}O^\nu_{\mu,K}$ introduced in Section \ref{ch3:subs_cat-O}. More precisely, we have
$$
O^\nu_{\mu,K}=\bigoplus_{\widetilde\mu\in X^+_{\widetilde I},\pi_e(\widetilde\mu)=\mu}O^\nu_{\widetilde\mu,K}.
$$
Similarly, there are decompositions

$$
E=\bigoplus_{j\in\widetilde I}E_j, \qquad F=\bigoplus_{j\in\widetilde I}F_j
$$
such that $E_j$ and $F_j$ map $O_{\widetilde\mu,K}$
 to $O_{\widetilde\mu+\widetilde\alpha_j,K}$ and $O_{\widetilde\mu-\widetilde\alpha_j,K}$ respectively. For $i\in I_e$, we set
$$
E_i=\bigoplus_{j\in \widetilde I,\pi_e(j)=i}E_j,\qquad F_i=\bigoplus_{j\in
\widetilde I,\pi_e(j)=i}F_j.
$$
We have commutative diagrams
$$
\begin{CD}
O^{\nu,\Delta}_{-e,R} @>{E_i}>> O^{\nu,\Delta}_{-e,R}\\
@V{K\otimes_R\bullet}VV         @V{K\otimes_R\bullet}VV\\
O^{\nu,\Delta}_{-e,K} @>{E_i}>> O^{\nu,\Delta}_{-e,K},\\
\end{CD}
\qquad
\begin{CD}
O^{\nu,\Delta}_{-e,R} @>{F_i}>> O^{\nu,\Delta}_{-e,R}\\
@V{K\otimes_R\bullet}VV         @V{K\otimes_R\bullet}VV\\
O^{\nu,\Delta}_{-e,K} @>{F_i}>> O^{\nu,\Delta}_{-e,K}.\\
\end{CD}
$$

For each element $\widetilde\alpha\in Q^+_{\widetilde I}$ let $\bfA^\nu_K[\widetilde\alpha]$ be the Serre subcategory of $\bfA^\nu_K$ generated by the Verma modules $\Delta[\lambda]_K$ such that $\lambda\in \calP^l_{\widetilde\alpha,\nu}$. Similarly to Section \ref{ch3:subs_cat-bfA}, we have 
$$
E_j(\bfA^\nu_K[\widetilde\alpha])\subset \bfA^\nu_K[\widetilde\alpha-\widetilde\alpha_j],\quad F_j(\bfA^\nu_K[\widetilde\alpha])\subset \bfA^\nu_K[\widetilde\alpha+\widetilde\alpha_j].
$$
See \cite[Sec.~7.4]{RSVV} for details.

Similarly, for $j\in {\overline{\widetilde I}}\simeq\widetilde I$ we can define the endofunctor $\overline E_j$, $\overline F_j$ of the categories $O^\nu_{-(e+1),\overline K}$ and $\overline \bfA^\nu_{\overline K}$. 

The decomposition $\bfA^\nu_K[\alpha]=\bigoplus_{\pi_e(\widetilde\alpha)=\alpha}\bfA^\nu_K[\widetilde\alpha]$ yields a decomposition $\calA^\nu_K[\alpha]=\bigoplus_{\pi_e(\widetilde\alpha)=\alpha}\calA^\nu_K[\widetilde\alpha]$.
We also consider the endofunctors $E_j$, $F_j$ of $\calA^\nu_{K}$ such that for $j=(a,b)\in \widetilde I$ we have the following analogue of (\ref{ch3:eq_E-from-E-bar-A}):
$$
E_j=
\left\{\begin{array}{lll}
\restr{\overline E_{(\Upsilon(a),b)}}{\calA^\nu_K} & \mbox{ if }\pi_{e}(j)\ne k,\\
 \restr{\overline E_{(\Upsilon(a),b)}\overline E_{(\Upsilon(a)+1,b)}}{{\calA^\nu_K}}& \mbox{ if }\pi_{e}(j)= k,
\end{array}\right.
$$
$$
F_j=
\left\{\begin{array}{lll}
\restr{\overline F_{(\Upsilon(a),b)}}{\calA^\nu_K} & \mbox{ if }\pi_{e}(j)\ne k,\\
\restr{\overline F_{(\Upsilon(a)+1,b)}\overline F_{(\Upsilon(a),b)}}{\calA^\nu_K}& \mbox{ if }\pi_{e}(j)= k.
\end{array}\right.
$$
We have $E_j(\calA^\nu_K[\widetilde\alpha])\subset \calA^\nu_K[\widetilde\alpha-\widetilde\alpha_j]$ and  $F_j(\calA^\nu_K[\widetilde\alpha])\subset \calA^\nu_K[\widetilde\alpha+\widetilde\alpha_j]$.

\subsection{The modules $T_{d,R}$, $\overline T_{d,R}$}
Assume that $R$ is as in Assumption 2.

Consider the module $T_{d,R}=F^d(\Delta[\emptyset]_R)$ in $\bfA^\nu_R[d]$ and the module $\overline T_{d,R}=F^d(\overline\Delta[\emptyset]_R)$ in $\calA^\nu_R[d]$. The commutativity of the diagram (\ref{ch3:eq_comm-diag-Groth-A}) implies that we have the following equality of classes $[\theta_d(T_{d,R})]=[\overline T_{d,R}]$ in $[\calA^\nu_R[d]]$. The modules $T_{d,R}\in\bfA^\nu_R[d]$ and $\overline T_{d,R}\in\calA^\nu_R[d]$ are tilting because $\Delta[\emptyset]_R\in \bfA^\nu_R[0]$ and $\overline\Delta[\emptyset]_R\in\calA_R^\nu[0]$ are tilting and the functor $F$ preserves tilting modules, see \cite[Lem.~8.33,~Lem.~5.16~(b)]{RSVV}. Since a tilting module is characterized by its class in the Grothendieck group, we deduce that there is an isomorphism of modules
\begin{equation}
\label{ch3:eq_theta-T-Tbar-d}
\theta_d(T_{d,R})\simeq\overline T_{d,R}.
\end{equation}

The representation datum in $O^{\nu,\Delta}_{-e,R}$ given in Proposition \ref{ch3:prop_functors-on-O-gen}, we have a homomorphism $H_{d,R}(q_e)\to \End(T_{d,R})^{\rm op}$. Similarly, we have homomorphisms $H_{d,\bfk}(\zeta_e)\to \End(T_{d,\bfk})^{\rm op}$ and $H_{d,K}(q_e)\to \End(T_{d,K})^{\rm op}$. The given homomorphisms commute with the base changes $\bfk\otimes_R\bullet$ and $K\otimes_R\bullet$, see \cite[Prop.~8.30]{RSVV}.

Now we are going to construct a similar homomorphism 
\begin{equation}
\label{ch3:eq_Hd-act-onTbar}
H_{d,R}(q_e)\to \End(\overline T_{d,R})^{\rm op}
\end{equation} 
using the representation datum in $O^{\nu,\Delta}_{-(e+1),\overline R}$. To avoid confusion, we stress that we are not going to use (\ref{ch3:eq_theta-T-Tbar-d}) to define (\ref{ch3:eq_Hd-act-onTbar}). We are going to construct (\ref{ch3:eq_Hd-act-onTbar}) using the representation datum in $O^{\nu,\Delta}_{-(e+1),\overline R}$ and not the representation datum in $O^{\nu,\Delta}_{-e,R}$.

Now, assume that $R$ is as in Assumption 1.
A notion of a categorical representation of $(\widetilde{\mathfrak{sl}}_e,{\mathfrak{sl}}_\infty^{\oplus l})$ is given in \cite[Def.~B.17]{Mak-categ}. The triple $(O^{\nu,\Delta}_{-e,R},O^{\nu}_{-e,\bfk},O^{\nu}_{-e,K})$ matchs this definition. (The only difference is that the categories that we consider here do not satisfy the $\Hom$-finiteness condition that is assumed in \cite{Mak-categ}. However, it is possible to truncate the categories to make them $\Hom$-finite.) 
Similarly, the triple $(O^{\nu,\Delta}_{-(e+1),\overline R},O^{\nu}_{-(e+1),\overline\bfk},O^{\nu}_{-(e+1),\overline K})$ is a categorical representation of $(\widetilde{\mathfrak{sl}}_{e+1},{\mathfrak{sl}}_\infty^{\oplus l})$. Apply \cite[Thm.~B.18]{Mak-categ}, we get a categorical representation of $(\widetilde{\mathfrak{sl}}_e,{\mathfrak{sl}}_\infty^{\oplus l})$ on some triple of subcategories. In particular this yields a homomorphism (\ref{ch3:eq_Hd-act-onTbar}). 

Similarly, we have homomorphisms $H_{d,\bfk}(\zeta_e)\to \End(\overline T_{d,\bfk})^{\rm op}$ and $H_{d,K}(q_e)\to \End(\overline T_{d,K})^{\rm op}$. The given homomorphisms commute with the base changes $\bfk\otimes_R\bullet$ and $K\otimes_R\bullet$.

\begin{lem}
\label{ch3:lem_Tbar-fact-cycl-bfk}
The homomorphism $H_{d,\bfk}(\zeta_e)\to \End(\overline T_{d,\bfk})$ factors through a homomorphism $\overline\psi_{d,\bfk}\colon H^\nu_{d,\bfk}(\zeta_e)\to \End(\overline T_{d,\bfk})$.
\end{lem}
\begin{proof}
The statement follows from the lemma below.
\end{proof}

Only in the lemma below, we assume that $\bfk$ is an arbitrary field and $\scrF\subset \bfk^\times$ is an arbitrary subset.
\begin{lem}
\label{ch3:lem_categ-rep-fact-cycl-general}
Let $\calC=\bigoplus_{\mu\in X_\scrF}\calC_\mu$ be a categorical representation of $\frakg_\scrF$ over $\bfk$, see Definition \ref{ch3:def-categ_action-Hecke}.
Let $M\in\calC$ be an object such that there are non-negative integers $t_i$ for $i\in \scrF$ such that $t_i$ is non-zero only for finitely many $i$ and
\begin{itemize}
    \item[\textbullet] $\End(M)\simeq \bfk$,
    \item[\textbullet] $E_iF_i(M)\simeq M^{\oplus t_i},\quad \forall~i\in \scrF$.
\end{itemize}
Then the polynomial $\prod_{i\in\scrF}(X_1-i)^{t_i}$ is in the kernel of the homomorphism $H_{d,\bfk}(q)\to \End(F^d(M))^{\rm op}$ for each $d\in\bbN$.
\end{lem}
\begin{proof}
It is enough to prove the statement for $d=1$.
By adjointness we have the following isomorphisms of vector spaces
$$
\Hom(F_i(M),F_i(M))\simeq \Hom(E_iF_i(M),M)\simeq \Hom(M,M)^{\oplus t_i}\simeq \bfk^{t_i}.
$$

Each $F_i(M)$ is killed by $(X_1-i)^{t_i}$. Then $F(M)$ is killed by $\prod_{i\in\scrF}(X_1-i)^{t_i}$.
\end{proof}

Now, we get statements similar to Lemma \ref{ch3:lem_Tbar-fact-cycl-bfk} for $R$ and $K$.

\medskip
\begin{lem}
\label{ch3:lem_Tbar-fact-cycl-K}
The homomorphism $H_{d,K}(q_e)\to \End(\overline T_{d,K})$ factors through a homomorphism $\overline\psi_{d,K}\colon H^\nu_{d,K}(q_e)\to \End(\overline T_{d,K})$.
\end{lem}
\begin{proof}
This statement also follows from Lemma \ref{ch3:lem_categ-rep-fact-cycl-general}.
\end{proof}

\medskip
\begin{lem}
\label{ch3:lem_Tbar-fact-cycl-R}
The homomorphism $H_{d,R}(q_e)\to \End(\overline T_{d,R})$ factors through a homomorphism $\overline\psi_{d,R}\colon H^\nu_{d,R}(q_e)\to \End(\overline T_{d,R})$.
\end{lem}
\begin{proof}
This statement follows from Lemma \ref{ch3:lem_Tbar-fact-cycl-K} and from the commutativity of the following diagram 
$$
\begin{CD}
H_{d,R}(q_{e}) @>>>\End(\overline T_{d,R})^{\rm op}\\
@VVV @VVV\\
H_{d,K}(q_{e}) @>>>\End(\overline T_{d,K})^{\rm op}.
\end{CD}
$$
\end{proof}

\subsection{The proof of invertibility} 
We still assume that $R$ is as in Assumption 1.
From now on, we assume $\nu_r\geqslant d$ for each $r\in[1,l]$.
The goal of this section is to prove that under this assumption the homomorphism $\overline\psi_{d,R}$ in Lemma \ref{ch3:lem_Tbar-fact-cycl-R} is an isomorphism. 


Consider the functors
$$
\begin{array}{ll}
\Psi^\nu_d\colon \bfA^\nu_R[d]\to \mod(H^\nu_{d,R}(q_e)),& M\mapsto \Hom(T_{d,R},M),\\
\overline\Psi^\nu_d\colon \calA^\nu_R[d]\to \mod(H^\nu_{d,R}(q_e)),& M\mapsto \Hom(\overline T_{d,R},M),\\
\end{array}
$$
where $\Hom(T_{d,R},M)$ and $\Hom(\overline T_{d,R},M)$ are considered as $H^{\nu}_{d,R}(q_e)$-modules with respect to the homomorphisms $\psi_{d,R}$ and $\overline \psi_{d,R}$.

Let us abbreviate $\Psi=\Psi^\nu_{d}$, $\overline\Psi=\overline\Psi^\nu_{d}$, $T_{R}=T_{d,R}$ and $\overline T_R=\overline T_{d, R}$. We may write $\Psi_R$, $\overline\Psi_R$ to specify the ring $R$. For $\lambda\in\calP^l_d$ denote by $S[\lambda]_R$ the Specht module of $H^\nu_{d,R}(q_e)$. We will use similar notation for $\bfk$ or $K$ instead of $R$. See also \cite[Sec.~2.4.3]{RSVV}.







Let us identify $\widetilde I\simeq \scrF\subset K$ as in Section \ref{ch3:subs_param-Hecke}.
For each $\uj=(j_1,\cdots,j_n)\in \widetilde{I}^d$, the Hecke algebra $H^\nu_{d,K}(q_e)$ contains an idempotent $e(\uj)$ such that for each finite dimensional $H^\nu_{d,K}(q_e)$-module $M$ the subspace $1_\uj M\subset M$ is the generalized eigenspace of the commuting operators $X_1,X_2,\cdots,X_n$ with eigenvalues $j_1,j_2,\cdots,j_n$ respectively. The idempotent $e(\uj)$ acts on $T_{K}=F^d(\overline\Delta[\emptyset]_K)$ by projection onto $F_{j_d}F_{j_{d-1}}\cdots F_{j_1}(\overline\Delta[\emptyset]_K)$.

Now, for each $\widetilde\alpha\in Q^+_{\widetilde I}$, set $e(\widetilde\alpha)=\sum_{\uj\in \widetilde I^{\widetilde\alpha}}e(\uj)$
and
$$
F_{\widetilde\alpha}=\bigoplus_{\uj=(j_1,\cdots,j_d)\in \widetilde I^{\widetilde\alpha}}F_{j_1}F_{j_2}\cdots F_{j_d}.
$$ 
We obviously have 
$$
F^d=\bigoplus_{\widetilde\alpha\in Q^+_{\widetilde I},|\widetilde\alpha|=d}F_{\widetilde\alpha}.
$$
The idempotent $e(\widetilde\alpha)$ acts on $T_{K}=F^d(\overline\Delta[\emptyset]_K)$ by projection onto $F_{\widetilde\alpha}(\overline\Delta[\emptyset]_K)$.

\smallskip
\begin{lem}
\label{ch3:lem_isom-K}
$(a)$ The homomorphism $\overline\psi_{d,K}\colon H^{\nu}_{d,K}(q_e)\to \End(\overline T_K)^{\rm op}$ is an isomorphism.

$(b)$ For each $\lambda\in\calP^l_d$ we have $\overline\Psi(\overline\Delta[\lambda]_K)\simeq S[\lambda]_K$.
\end{lem}
\begin{proof}[Proof]

First, we prove that the homomorphism $\overline\psi_{d,K}$ is injective.
The algebra $H^\nu_{d,K}(q_e)$ is finite dimensional and semisimple. Its center is spanned by the idempotents $e(\widetilde\alpha)$ such that $\widetilde\alpha\in Q^+_{\widetilde I}$ and $|\widetilde\alpha|=d$. 

Since the idempotent $e(\widetilde\alpha)$ acts on $T_{K}$ by projection onto $F_{\widetilde\alpha}(\overline\Delta[\emptyset]_K)$, to prove the injectivity of $\overline\psi_{d,K}$ we need to check that $F_{\widetilde\alpha}(\overline\Delta[\emptyset]_K)$ is nonzero whenever $e(\widetilde\alpha)$ is nonzero. Similarly to the argument in Section \ref{ch3:subs_comm-Groth}, we see that the equivalence $\theta\colon \bfA^\nu_K\simeq \calA^\nu_K$ yields an isomorphism of Grothendieck groups $[\bfA^\nu_K]\simeq [\calA^\nu_K]$ that commutes with functors $F_j$. Thus the module $F_{\widetilde\alpha}(\overline\Delta[\emptyset]_K)\in\calA^\nu_K$ is nonzero if and only if the module $F_{\widetilde\alpha}(\Delta[\emptyset]_K)\in\bfA^\nu_K$ is nonzero. By \cite[Prop.~5.22~(d)]{RSVV}, the module $F_{\widetilde\alpha}(\Delta[\emptyset]_K)\in\bfA^\nu_K$ is nonzero whenever $e(\widetilde\alpha)$ is nonzero. Thus $\overline\psi_{d,K}$ is injective.


Thus it is also surjective because
$$
\dim_K H^\nu_{d,K}(q_e)=\dim_K\End(T_K)^{\rm op}=\dim_K\End(\overline T_K)^{\rm op},
$$
where the first equality holds by \cite[Prop.~5.22~(d)]{RSVV} and the second holds by (\ref{ch3:eq_theta-T-Tbar-d}). This implies part $(a)$.

The discussion above implies that $\overline T_K$ contains each $\overline\Delta[\lambda]_K$, $\lambda\in \calP^l_d$ as a direct factor. In particular $\overline T_K$ is a projective generator of $\calA^\nu_K[d]$. Thus $\overline\Psi_K$ is an equivalence of categories. It must take $\overline\Delta[\lambda]_K$ to $S[\lambda]_K$ because $S[\lambda]_K$ is the unique simple module in the block $\mod(H^\nu_{\widetilde\alpha,K}(q_e))$ of $\mod(H^\nu_{d,K}(q_e))$.
\end{proof}

\smallskip
\begin{lem}
\label{ch3:lem_isom-R}
$(a)$ The homomorphism $\overline\psi_{d,K}\colon H^{\nu}_{d,R}(q_e)\to \End(\overline T_R)^{\rm op}$ is an isomorphism.

$(b)$ For each $\lambda\in\calP^l_d$ we have $\overline\Psi(\overline\Delta[\lambda]_R)\simeq S[\lambda]_R$.
\end{lem}
\begin{proof}[Proof]
Consider the endomorphism $u$ of $H^\nu_{d,R}(q_e)$ obtained from the following chain of homomorphisms
$$
u\colon H^\nu_{d,R}(q_e)\stackrel{\overline\psi_{d,R}}{\to}\End_{\calA^\nu}(\overline T_R)^{\rm op}\stackrel{\theta_{d}^{-1}}{\to}\End_{\bfA^\nu}(T_R)^{\rm op}\stackrel{\psi_{d,R}^{-1}}{\to}H^\nu_{d,R}(q_e).
$$
The invertibility of $\overline\psi_{d,R}$ is equivalent to the invertibility of $u$. By \cite[Prop.~2.23]{RSVV} to prove that $u$ is an isomorphism it is enough to show that its localization $Ku\colon H^\nu_{d,K}(q_e)\to H^\nu_{d,K}(q_e)$ is an isomorphism and that $Ku$ induces the identity map on Grothendieck groups $[\mod(H^\nu_{d,K}(q_e))]\to [\mod(H^\nu_{d,K}(q_e))]$. The bijectivity of $Ku$ follows from Lemma \ref{ch3:lem_isom-K} $(a)$.

Now we check the condition on the Grothendieck group. We already know from \cite[Prop.~5.22~(c)]{RSVV} and the proof of Lemma \ref{ch3:lem_isom-K} that $\Psi_K$ and $\overline\Psi_K$ are equivalences of categories. Thus, by semisimplicity of the categories $\bfA^\nu_K[d]$, $\calA^\nu_K[d]$ and $\mod(H^\nu_{d,K}(q_e))$, we have an isomorphism of functors $\Psi_K\simeq \overline \Psi_K\circ \theta_{d}$ because $\Psi_K(M)\simeq \overline\Psi_K\circ \theta_d(M)$ for each $M\in\bfA_K^\nu[d]$. This implies that $Ku$ is the identity on the Grothendieck group. This proves part $(a)$.

Part $(b)$ follows from Lemma \ref{ch3:lem_isom-K} and the characterization of Specht modules, see \cite[Sec.~2.4.3]{RSVV}.
\end{proof}

\smallskip
\begin{rk}
There is no reason why the automorphism $u\colon H^\nu_{d,R}(q_e)\to H^\nu_{d,R}(q_e)$ in the proof of Lemma \ref{ch3:lem_isom-R} should be identity. Because of this, the functor $\Psi$ has no reason to coincide with $\overline\Psi\circ\theta_d$. However the automorphism $u$ of $H^\nu_{d,R}(q_e)$ induces an autoequivalence $u^*$ of $\mod(H^\nu_{d,R}(q_e))$ such that we have
\begin{equation}
\label{ch3:eq_isom-Psi-Psibar-u*}
\Psi= u^*\circ\overline\Psi\circ\theta_d.
\end{equation}
\end{rk}

\smallskip
Now, specializing to $\bfk$, we obtain the following.

\smallskip
\begin{coro}
\label{ch3:coro_isom-C}
$(a)$ The homomorphism $\overline\psi_{d,\bfk}\colon H^\nu_{d,\bfk}(\zeta_e)\to \End(\overline T_\bfk)^{\rm op}$ is an isomorphism.

$(b)$ For each $\lambda\in\calP^l_{d}$ we have $\overline\Psi_\bfk(\overline\Delta[\lambda]_\bfk)\simeq S[\lambda]_\bfk$.
\qed
\end{coro}


\subsection{Rational Cherednik algebras}
Let $R$ be a local commutative $\bbC$-algebra with residue field $\bbC$. Let $W$ be a complex reflection group. Denote by $S=S(W)$ and $\calA$ the set of pseudo-reflections in $W$ and the set of reflection hyperplanes respectively. Let $\frakh$ be the reflection representation of $W$ over $R$. Let $c\colon S\to R$ be a map which is constant on the $W$-conjugacy classes.

Denote by $\langle\bullet,\bullet\rangle$ the canonical pairing between $\frakh^*$ and $\frakh$. For each $s\in S$ fix a generator $\alpha_s\in \frakh^*$ of ${\rm Im}(\restr{s}{\frakh^*}-1)$ and a generator $\check\alpha_s\in\frakh$ of ${\rm Im}(\restr{s}{\frakh}-1)$ such that $\langle\alpha_s,\check{\alpha}_s\rangle=2$.

\smallskip
\begin{df}
The \emph{rational Cherednik algebra} $H_c(W,\frakh)_R$ is the quotient of the smash product $RW\ltimes T(\frakh\oplus\frakh^*)$ by the relations
$$
[x,x']=0,\qquad [y,y']=0,\qquad [y,x]=\langle x,y\rangle-\sum_{s\in S}c_s\langle\alpha_s,y\rangle\langle x,\check{\alpha}_s\rangle s,
$$
for each $x,x'\in\frakh^*$, $y,y'\in\frakh$. Here $T(\bullet)$ denotes the tensor algebra.
\end{df}

Denote by $\calO_c(W,\frakh)_R$ the category $\calO$ of $H_c(W,\frakh)_R$, see \cite[Sec.~3.2]{GGOR} and \cite[Sec.~6.1.1]{RSVV}.
Let $E$ be an irreducible representation of $\bbC W$.

\smallskip
\begin{df}
A \emph{Verma module} associated with $E$ is the following module in $\calO_c(W,\frakh)_R$
$$
\Delta_R(E):=\Ind_{RW\ltimes R[\frakh^*]}^{H_c(W,\frakh)_R}(RE).
$$
Here $\frakh\subset R[\frakh^*]$ acts on $RE$ by zero.
\end{df}
The category $\calO_c(W,\frakh)_R$ is a highest weight category over $R$ with standard modules $\Delta_R(E)$.

We call a subgroup $W'$ of $W$ \emph{parabolic} if it is a stabilizer of some point of $b\in\frakh$. In this case $W'$ is a complex reflection group with reflection representation $\frakh/\frakh^{W'}$, where $\frakh^{W'}$ is the set of $W'$-stable points in $\frakh$. Moreover, the map $c\colon S(W)\to R$ restricts to a map $c\colon S(W')\to R$. There are induction and restriction functors
$$
^\calO\Ind_{W'}^W\colon \calO_c(W',\frakh/\frakh^{W'})_R\to \calO_c(W,\frakh)_R,\quad ^\calO\Res_{W'}^W\colon \calO_c(W,\frakh)_R\to \calO_c(W',\frakh/\frakh^{W'})_R,
$$
see \cite{BE}. The definitions of these functors depend on $b$ but their isomorphism classes are independent of the choice of $b$.

The following lemma holds.

\smallskip
\begin{lem}
\label{ch3:lem_proj-conj-par-same}
Assume that $W'$ and $W''$ are conjugated parabolic subgroups in $W$. Let $P\in\calO_c(W,\frakh)_R$ be a projective module. Then the following conditions are equivalent
 \begin{itemize}
    \item[\textbullet] the module $P$ is isomorphic to a direct factor of the module $^\calO\Ind_{W'}^W(P')$ for some projective module $P'\in \calO_c(W',\frakh/\frakh^{W'})_R$,
    \item[\textbullet] the module $P$ is isomorphic to a direct factor of a module $^\calO\Ind_{W''}^W(P'')$ for some projective module $P''\in \calO_c(W'',\frakh/\frakh^{W''})_R$.
\end{itemize}
\end{lem}
\begin{proof}[Proof]
Let $w$ be an element of $W$ such that $wW'w^{-1}=W''$. The conjugation by $w$ yields an isomorphism $W'\simeq W''$. Hence, the element $w$ takes $\frakh^{W'}$ to $\frakh^{W''}$. Thus we get an algebra isomorphism $H_c(W',\frakh/\frakh^{W'})_R\simeq H_c(W'',\frakh/\frakh^{W''})_R$ and an equivalence of categories $\calO_c(W',\frakh/\frakh^{W'})_R\simeq \calO_c(W'',\frakh/\frakh^{W''})_R$. Moreover, the conjugation by $w$ yields an automorphism $t$ of $H_c(W,\frakh)_R$ such that for each $x\in\frakh^*$, $y\in\frakh$, $u\in W$ we have
$$
t(x)=w(x),\quad t(y)=w(y),\quad t(u)=wuw^{-1}.
$$
The following diagram of functors is commutative up to equivalence of functors
$$
\begin{CD}
\calO_c(W,\frakh)_R @<{t^*}<< \calO_c(W,\frakh)_R\\
@A{^\calO\Ind_{W'}^W}AA     @A{^\calO\Ind_{W''}^W}AA\\
\calO_c(W',\frakh/\frakh^{W'})_R @<<<   \calO_c(W'',\frakh/\frakh^{W''})_R
\end{CD}
$$

To conclude, we need only to prove that the pull-back $t^*$ induces the identity map on the Grothendieck group of $\calO_c(W,\frakh)_R$ (and thus it maps each projective module to an isomorphic one). This is true because $t^*$ maps each Verma module $\Delta(E)_R$ to an isomorphic one because the representation $E$ of $W$ does not change the isomorphism class when we twist the $W$-action by an inner automorphism.
\end{proof}

\subsection{Cyclotomic rational Cherednik algebras}
From now on, we assume that $R=\bbC$ or that $R$ is as in Assumption 1 with residue field $\bfk=\bbC$.

Let $\Gamma\simeq\bbZ/l\bbZ$ be the group of complex $l$th roots of unity and set $\Gamma_d=\frakS_d\ltimes \Gamma^d$. For $\gamma\in\Gamma$, $r\in[1,l]$ denote by $\gamma_r$ the element of $\Gamma^d$ having $\gamma$ at the position $r$ and $1$ at other positions. Let $s_{r,t}$ be the transposition in $\frakS_d$ exchanging $r$ and $t$. For $\gamma\in\Gamma$, $r,t\in[1,l]$ set $s_{r,t}^\gamma:=s_{r,t}\gamma_r\gamma^{-1}_t\in\Gamma_d$. From now on we suppose that the group $W$ is $\Gamma_d$ and $\frakh=R^d$ is the obvious reflection representation of $\Gamma_d$. Assume also that $h,h_1,\cdots,h_{l-1}$ are some elements of $R$ and set $h_{-1}=h_{l-1}$. Let us chose the parameter $c$ in the following way
$$
c(s_{r,t}^\gamma)=-h \qquad \mbox{ for each $r,t\in [1,t]$, $r\ne t$, $\gamma\in\Gamma$},
$$
$$
c(\gamma_r)=-\frac{1}{2}\sum_{p=0}^{l-1}\gamma^{-p}(h_p-h_{p-1})\qquad \mbox{ for each $r\in[1,l]$, $\gamma\in\Gamma$, $\gamma\ne 1$}.
$$

Let $\nu_1,\cdots,\nu_l$ be as above. We set
$$
h=-1/\kappa,\quad h_p=-(\nu_{p+1}+\tau_{p+1})/\kappa-p/l, \quad p\in[1,l-1].
$$
Let us abbreviate $\calO^\nu_R[d]=\calO_c(\Gamma_d,R^d)_R$. Consider the $\KZ$-functor $\KZ^\nu_d\colon \calO^\nu_R[d]\to \mod(H^\nu_{d,R}(q_e))$ introduced in \cite[Sec.~6]{RSVV}. Denote by $^*\calO^\nu_R[d]$ the category defined in the same way as $\calO^\nu_R[d]$ with replacement of $(\nu_1,\cdots,\nu_l)$ by $(-\nu_l,\cdots,-\nu_1)$ and $(\tau_1,\cdots,\tau_l)$ by $(-\tau_l,\cdots,-\tau_1)$. Similarly, denote by $^*H^\nu_{d,R}(q_e)$ the affine Hecke algebra defined in the same way as $H^\nu_{d,R}(q_e)$ with the replacement of parameters as above. There is also a $\KZ$-functor $^*\KZ^\nu_d\colon {^*\calO}^\nu_R[d]\to \mod(^*H^\nu_{d,R}(q_e))$.

The simple $\bbC\Gamma_d$-modules are labeled by the set $\calP^l_d$. We write $E(\lambda)$ for the simple module corresponding to $\lambda$. Set $\Delta[\lambda]_R=\Delta(E(\lambda))_R$. Similarly, write $P[\lambda]_R$ and $T[\lambda]_R$ for the projective and tilting object in $\calO^\nu_R[d]$ with index $\lambda$.

The category $^*\calO^{\nu}_R[d]$ is the Ringel dual of the category $\calO^{\nu}_R[d]$, see \cite[Sec.~6.2.4]{RSVV}. In particular we have an equivalence between the categories of standardly filtered objects $\scrR\colon{^*\calO}^{\nu}_R[d]^{\Delta}\to (\calO^{\nu}_R[d]^{\Delta})^{\rm op}$. Let $\proj(R)$ be the category of projective finitely generated $R$-modules. There is an algebra isomorphism
$$
\iota\colon H^\nu_{d,R}(q_e)\to(^*H^\nu_{d,R}(q_e))^{\rm op},\quad T_r\mapsto -q_eT_r^{-1},~X_r\mapsto X_r^{-1},
$$
see \cite[Sec.~6.2.4]{RSVV}. It induces an equivalence
$$
\scrR_H=\iota^*(\bullet^\vee)\colon \mod(^*H^\nu_{d,R}(q_e))\cap\proj(R)\to(\mod(H^\nu_{d,R}(q_e))\cap\proj(R))^{\rm op},
$$
where $\bullet^\vee$ is the dual as an $R$-module.
By \cite[(6.3)]{RSVV}, 
the following diagram of functors is commutative
\begin{equation}
\label{ch3:eq_diag-scrR}
\begin{CD}
^*\calO^\nu_R[d]^\Delta @>{\scrR}>> (\calO^\nu_R[d]^\Delta)^{\rm op}\\
@V{^*\KZ^\nu_{d}}VV     @V{\KZ^\nu_{d}}VV\\
\mod({^*H}^\nu_{d,R}(q_e))\cap\proj(R) @>{\scrR_H}>>(\mod({H}^\nu_{d,R}(q_e))\cap\proj(R))^{\rm op}.
\end{CD}
\end{equation}


\smallskip
\begin{lem}
\label{ch3:lem_d=1-Cher}
Assume that $d=1$.
For each $l$-partition of $1$ $\lambda$ we have an isomorphism of $H^\nu_{1,R}(q_e)$-modules $\KZ^\nu_1(P[\lambda]_R)\simeq \overline \Psi^\nu_1(\overline T[\lambda]_R)$.
\end{lem}
\begin{proof}[Proof]
The proof is similar to the proof of \cite[Prop.~6.7]{RSVV}. The commutativity of the diagram (\ref{ch3:eq_diag-scrR}) implies
$$
\KZ^\nu_1(P[\lambda]_R)=\KZ^\nu_1(\scrR(T[\lambda]_R))=\scrR_H({^*\KZ^\nu_1}(T[\lambda]_R)).
$$

To conclude, we just need to compare the highest weight covers $\scrR_H\circ{^*\KZ^\nu_1}$ and $\overline\Psi^\nu_1$ of $H^\nu_{1,R}(q_e)$ using Lemma \ref{ch3:lem_isom-K} $(b)$ and \cite[Prop.~2.21]{RSVV}.

\end{proof}

Let $O^+_{\mu,R}$ be the affine parabolic category $\calO$ associated with the parabolic type consisting of the single block of size $\nu_1+\cdots+\nu_l$.
We define the categories $\bfA^+_{R}[d]$, $\calA^+_R[d]$ and $O^+_R[d]$ similarly. In this case we will also write the upper index $+$ in the notation of modules and functors (for example $\Delta^+[\lambda]_R$, $T^+_{d,R}$, $\KZ^+_{d}$, etc.) Let also $H^+_{d,R}(q_e)$ be the cyclotomic Hecke algebra with $l=1$. It is isomorphic to the Hecke algebra of $\frakS_d$.

We can prove the following result.

\smallskip
\begin{lem}
\label{ch3:lem_d=2,l=1-Cher}
For each $\lambda\in\calP^1_2$ we have $\KZ^+_2(P^+[\lambda]_R)\simeq \overline \Psi^+_2(\overline T^+[\lambda]_R)$.\qed
\end{lem}
\begin{proof}[Proof]
Similarly to the proof of Lemma \ref{ch3:lem_d=1-Cher} we compare the highest weight covers $\scrR_H\circ{^*\KZ^+_2}$ and $\overline\Psi^+_2$ of $H^+_{2,R}(q_e)$ using Lemma \ref{ch3:lem_isom-K} $(b)$ and \cite[Prop.~2.21]{RSVV}.

\end{proof}

Denote by $\Ind_{d,+}^{d,\nu}$ the induction functor with respect to the inclusion $H^{+}_{d,R}(q_e)\subset H^{\nu}_{d,R}(q_e)$.
We will also need the following lemma.

\smallskip
\begin{lem}
\label{ch3:lem_d=1-to-d=2,l=1}
Assume $\nu_r\geqslant 2$ for each $r\in[1,l]$. Assume also that $e>2$. For each $\lambda\in\calP^1_2$ there exists a tilting module $\overline T_{\lambda,R}\in\calA^\nu_R[2]$ such that $\Psi^\nu_2(\overline T_{\lambda,R})\simeq \Ind_{2,+}^{2,\nu}(\overline\Psi^+_2(\overline T^+[\lambda]_R))$.
\end{lem}
\begin{proof}[Proof]
Set $\lambda^+=(2)$, $\lambda^-=(1,1)$.
We have $\zeta_e\ne -1$ because $e>2$. In this case the algebra $H^+_{2,\bfk}(\zeta_e)$ is semisimple. The category $\calA^+_\bfk[2]$ is also semisimple. This implies
$$
\overline T^+_{2,R}\simeq \overline\Delta[\lambda^+]_R\oplus\overline\Delta[\lambda^-]_R= \overline T[\lambda^+]_R\oplus \overline T[\lambda^-]_R.
$$
By definition, we have $\overline\Psi^+_{2}(\overline T^+_{2,R})\simeq H^+_{2,R}(q_e)$ and $\overline\Psi^\nu_{2}(\overline T_{2,R})\simeq H^\nu_{2,R}(q_e)$. This implies
$$
\overline\Psi^\nu_{2}(\overline T_{2,R})\simeq \Ind_{2,+}^{2,\nu}(\overline\Psi^+_{2}(\overline T^+_{2,R})).
$$
By the proof of \cite[Prop.~6.8]{RSVV},
the functor $\Psi_2^\nu$ takes indecomposable factors of $T_{2,R}$ to indecomposable modules. Thus, by (\ref{ch3:eq_isom-Psi-Psibar-u*}), the functor $\overline\Psi_2^\nu$ takes indecomposable factors of $\overline T_{2,R}$ to indecomposable modules. Thus there is a decomposition $\overline T_{2,R}=\overline T_{\lambda^+,R}\oplus \overline T_{\lambda^-,R}$ such that $\overline T_{\lambda^+,R}$, $\overline T_{\lambda^-,R}$ satisfy the required properties.


\end{proof}

\subsection{Proof of Theorem \ref{ch3:thm_intro-main-decomp-functors}}
In this section we finally give a proof of over main result.

A priori there is no reason to have the following isomorphism of functors $\Psi^\nu_\alpha\simeq \overline \Psi^\nu_\alpha\circ \theta_\alpha$. However, we can modify the equivalence $\theta_\alpha$ to make this true.

For $d_1<d_2$ we have an inclusion $H^\nu_{d_1,R}(q_e)\subset H^\nu_{d_2,R}(q_e)$. Let $\Ind_{d_1}^{d_2}\colon \mod(H^\nu_{d_1,R}(q_e))\to \mod(H^\nu_{d_2,R}(q_e))$ be the induction with respect to this inclusion. The following lemma can be proved similarly to \cite[Lem.~5.41]{RSVV}.

\smallskip
\begin{lem}
\label{ch3:lem_com-F-calA-Hecke}
Assume that $\nu_r> d$ for each $r\in[1,l]$. Then the following diagram of functors is commutative.
$$
\begin{CD}
\calA^\nu_R[d] @>{F}>> \calA^\nu_R[d+1]\\
@V{\overline\Psi^\nu_d}VV                @V{\overline\Psi^\nu_{d+1}}VV\\
\mod(H^\nu_{d,R}(q_e))@>{\Ind_d^{d+1}}>>\mod(H^\nu_{d+1,R}(q_e))
\end{CD}
$$
\qed
\end{lem}

For a partition $\lambda$ denote by $\lambda^*$ the transposed partition. For an $l$-partition $\lambda=(\lambda_1,\cdots,\lambda_l)$ set $\lambda^*=((\lambda_l)^*,\cdots,(\lambda_1)^*)$.
There is an algebra isomorphism
$$
{\rm IM}\colon H^\nu_{d,R}(q_e)\to {^*H}^\nu_{d,R}(q_e), \quad  T_r\mapsto -q_eT_r^{-1},~X_r\mapsto X_r^{-1},
$$
see \cite[Sec.~6.2.4]{RSVV}. Let ${\rm IM}^*\colon \mod({^*H}^\nu_{d,R}(q_e))\to \mod(H^\nu_{d,R}(q_e))$ be the induced equivalence of categories. We have
\begin{equation}
\label{ch3:eq_IM-on-Specht}
{\rm IM}^*(S[\lambda^*]_R)\simeq S[\lambda]_R.
\end{equation}

The following proposition is proved in \cite[Thm.~6.9]{RSVV}.
\begin{prop}
\label{ch3:prop_equiv-Cher-bfA}
Assume that $\nu_r\geqslant d$ for each $r\in[1,l]$. Then there is an equivalence of categories $\gamma_d\colon{^*\calO}^{\nu}_R[d]\simeq \bfA^\nu_R[d]$ taking $\Delta[\lambda^*]_R$ to $\Delta[\lambda]_R$. Moreover,  we have the following isomorphism of functors $\Psi^\nu_d\circ\gamma_d\simeq {\rm IM}^*\circ{^*\KZ}^{\nu}_d$.
\end{prop}

Now, we prove a similar statement for $\calA^\nu_R[d]$. For each reflection hyperplane $H$ of the complex reflection group $\Gamma_d$ let $W_H\subset \Gamma_d$ be the pointwise stabilizer of $H$. 

\smallskip
\begin{prop}
\label{ch3:prop_constr-of-equiv}
Assume that $\nu_r\geqslant d$ for each $r\in[1,l]$ and $e>2$.
There is an equivalence of categories $\overline\gamma_d\colon {^*\calO}^{\nu}_R[d]\simeq \calA^\nu_R[d]$, taking $\Delta[\lambda^*]_R$ to $\overline\Delta[\lambda]_R$. Moreover, we have the following isomorphism of functors $\overline\Psi^\nu_d\circ\overline\gamma_d\simeq {\rm IM}^*\circ{^*\KZ}^{\nu}_d$.
\end{prop}
\begin{proof}[Proof]
The proof is similar to the proof of \cite[Thm. 6.9]{RSVV}.
We set $\calC={^*\calO}^{\nu}_R[d]$, $\calC'=\calA^\nu_R[d]$. Consider the following functors
$$
\begin{array}{lll}
Y\colon \calC\to \mod(H^\nu_{d,R}(q_e)),\quad &Y={\rm IM}^*\circ {^*\KZ}^{\nu}_d,\\
Y'\colon \calC'\to \mod(H^\nu_{d,R}(q_e)),\quad &Y'=\overline\Psi^\nu_d.
\end{array}
$$
By \cite[Prop.~2.20]{RSVV} it is enough to check the following four conditions.
\begin{itemize}
\item[(1)] We have $Y(\Delta[\lambda^*]_R)\simeq Y'(\Delta[\lambda]_R)$ and the bijection $\Delta[\lambda^*]_R\mapsto \Delta[\lambda]_R$ between the sets of standard objects in $\calC$ and $\calC'$ respects the highest weight orders.
\item[(2)] The functor $Y$ is fully faithful on $\calC^\Delta$ and $\calC^\nabla$.
\item[(3)] The functor $Y'$ is fully faithful on $\calC'^\Delta$ and $\calC'^\nabla$.
\item[(4)] For each reflection hyperplane $H$ of $\Gamma_d$ and each projective module $P\in\calO(W_H)_R$ we have
$$
\KZ^\nu_d(^\calO\Ind^{\Gamma_d}_{W_H}P)\in F'(\calC'^{\rm tilt}).
$$
\end{itemize}

It is explained in the proof of \cite[Thm.~6.9]{RSVV} that condition $(4)$ announced here implies the fourth condition in \cite[Prop.~2.20]{RSVV}.

We have $Y(\Delta[\lambda^*]_R)\simeq Y'(\Delta[\lambda]_R)$ by Lemma \ref{ch3:lem_isom-R} $(b)$, \cite[Lem.~6.6]{RSVV} and (\ref{ch3:eq_IM-on-Specht}). The composition of the equivalence $\theta_d\colon\bfA^\nu_R[d]\to\calA^\nu_R[d]$ with the equivalence $\gamma_d\colon{^*\calO}^\nu_R[d]\simeq\bfA^\nu_R[d]$ yields an equivalence of highest weight categories $\calC\simeq \calC'$ that takes $\Delta[\lambda^*]_R$ to $\Delta[\lambda]_R$. This implies $(1)$.

Condition $(2)$ is already checked in \cite[Sec.~6.3.2]{RSVV}.

The functor $\Psi^\nu_d$ is fully faithful on $\bfA^{\nu,\Delta}_R[d]$ and $\bfA^{\nu,\nabla}_R[d]$ by \cite[Thm.~5.37~(c)]{RSVV}. Thus the functor $\overline\Psi^\nu_d$ is fully faithful on $\calA^{\nu,\Delta}_R[d]$ and $\calA^{\nu,\nabla}_R[d]$ by (\ref{ch3:eq_isom-Psi-Psibar-u*}). This implies $(3)$.


Let us check condition $(4)$.
There are two possibilities for the hyperplane $H$.
\begin{itemize}
    \item[\textbullet] The hyperplane is $\Ker(\gamma_r-1)$ for $r\in[1,d]$. 
By Lemma \ref{ch3:lem_proj-conj-par-same}, we can assume that $H=\Ker(\gamma_1-1)$. By Lemma \ref{ch3:lem_d=1-Cher} there exists a tilting module $\overline T\in \calA^\nu_R[1]$ such that $\KZ^\nu_1(P)\simeq \overline\Psi^\nu_1(\overline T)$. 
We get
$$
\KZ^\nu_d(^\calO\Ind^{\Gamma_d}_{W_H}P)\simeq\Ind_1^d(\KZ^\nu_1(P))\simeq \Ind_1^d(\overline \Psi^\nu_1(\overline T))\simeq \overline\Psi^\nu_d(F^{d-1}(\overline T)).
$$
Here the first isomorphism follows from \cite[(6.1)]{RSVV}, the third isomorphism follows from Lemma \ref{ch3:lem_com-F-calA-Hecke}.
    \item[\textbullet] The hyperplane is $\Ker(s_{r,t}^\gamma-1)$ for $r,t\in[1,d]$, $\gamma\in\Gamma$. 
By Lemma \ref{ch3:lem_proj-conj-par-same}, we can assume that $H=\Ker(s_{1,2})$. By Lemma \ref{ch3:lem_d=2,l=1-Cher} there is a tilting module $\overline T^+\in\calA_R^+[2]$ such that $\KZ^+_2(P)\simeq \overline \Psi^+_2(\overline T^+)$. By Lemma \ref{ch3:lem_d=1-to-d=2,l=1} there is a tilting module $\overline T\in\calA^\nu_R[2]$ such that $\Ind_{2,+}^{2,\nu}(\overline\Psi^+_2(\overline T^+))\simeq \overline\Psi^\nu_2(\overline T)$. Thus we get $\Ind_{2,+}^{2,\nu}\KZ^+_2(P)\simeq \overline\Psi^\nu_2(\overline T)$.
\end{itemize}

We obtain
$$
\KZ^\nu_d(^\calO\Ind^{\Gamma_d}_{W_H}P)\simeq \Ind_{2,+}^{d,\nu}(\KZ^+_2(P))\simeq \Ind_{2,\nu}^{d,\nu}(\overline\Psi^\nu_2(\overline T))\simeq \Psi^\nu_d(F^{d-2}(\overline T)).
$$
Here the first isomorphism follows from \cite[(6.1)]{RSVV}, the third isomorphism follows from Lemma \ref{ch3:lem_com-F-calA-Hecke}.
\end{proof}

Now, composing the equivalences of categories in Propositions \ref{ch3:prop_equiv-Cher-bfA}, \ref{ch3:prop_constr-of-equiv} we obtain the following result.

\smallskip
\begin{coro}
\label{ch3:coro_constr-of-equiv-A-A}
Assume that $\nu_r\geqslant d$ for each $r\in[1,l]$ and $e>2$.
There is an equivalence of categories $\theta'_d\colon \bfA^\nu_R[d]\simeq \calA^\nu_R[d]$ such that we have the following isomorphism of functors $\overline\Psi^\nu_d\circ\theta'_d\simeq\Psi^\nu_d$.\qed
\end{coro}

For each $\alpha\in Q^+_e$ such that $|\alpha|=d$ let $\theta'_\alpha\colon \bfA^\nu_R[\alpha]\simeq \calA^\nu_R[\alpha]$ be the restriction of $\theta'_d$.

From now on we work over the field $\bbC$. 
The following lemma can be proved by the method used in \cite[Sec.~5.9]{RSVV}.

\smallskip
\begin{lem}
\label{ch3:lem_F-comm-KZ-e}
Assume that $\nu_r>d$ for each $r\in[1,l]$.
The following diagrams are commutative modulo an isomorphism of functors.
$$
\begin{CD}
\bfA^\nu[d]@>{F}>> \bfA^\nu[d+1]\\
@V{\Psi^\nu_d}VV               @V{\Psi^\nu_{d+1}}VV\\
\mod(H^\nu_d(\zeta_e))     @>{\Ind_d^{d+1}}>> \mod(H^\nu_{d+1}(\zeta_e))
\end{CD}
$$
$$
\begin{CD}
\calA^\nu[d]@>{F}>> \calA^\nu[d+1]\\
@V{\overline\Psi^\nu_d}VV               @V{\overline\Psi^\nu_{d+1}}VV\\
\mod(H^\nu_d(\zeta_e))     @>{\Ind_d^{d+1}}>> \mod(H^\nu_{d+1}(\zeta_e))
\end{CD}
$$
\qed
\end{lem}

Now, Theorem \ref{ch3:thm_intro-main-decomp-functors} follows from the following one.

\smallskip
\begin{thm}
\label{ch3:thm_decomp_Fk-A}
Assume that $\nu_r>d$ for each $r\in[1,l]$ and $e>2$.
Then the following diagram is commutative
$$
\begin{CD}
\calA^\nu[d]@>{F_k}>> \calA^\nu[d+1]\\
@A{\theta'_d}AA                 @A{\theta'_{d+1}}AA\\
\bfA^\nu[d]@>{F_k}>> \bfA^\nu[d+1].
\end{CD}
$$

In particular, for each $\alpha\in Q^+_e$ such that $|\alpha|=d$, the following diagram is commutative

$$
\begin{CD}
\calA^\nu[\alpha]@>{F_k}>> \calA^\nu[\alpha+\alpha_k]\\
@A{\theta'_\alpha}AA                 @A{\theta'_{\alpha+\alpha_k}}AA\\
\bfA^\nu[\alpha]@>{F_k}>> \bfA^\nu[\alpha+\alpha_k].
\end{CD}
$$
\end{thm}
\begin{proof}[Proof]
The result follows from Corollary \ref{ch3:coro_constr-of-equiv-A-A}, Lemma \ref{ch3:lem_F-comm-KZ-e} and an argument similar to \cite[Lem.~2.4]{Shan-Fock}.
\end{proof}

\section{Graded lifts of the functors}
\label{ch3:sec_gr-lifts}

\subsection{Graded categories}
For any noetherian ring $A$, let $\mod(A)$ be the category of finitely generated left $A$-modules. For any noetherian $\bbZ$-graded ring $A=\bigoplus_{n\in\bbZ}A_n$, let $\grmod(A)$ be the category of $\bbZ$-graded finitely generated left $A$-modules. The morphisms in $\grmod(A)$ are the morphisms which are homogeneous of degree zero. For each $M\in\grmod(A)$ and each $r\in\bbZ$ denote by $M_r$ the homogeneous component of degree $r$ in $M$. For $n\in\bbZ$ let $M\langle n \rangle$ be the $n$th shift of grading on $M$, i.e., we have $(M\langle n \rangle)_r=M_{r-n}$. For each $\bbZ$-graded finite dimensional $\bbC$-vector space $V$, let $\dim_q V\in\bbN[q,q^{-1}]$ be its graded dimension, i.e., $\dim_q V=\sum_{r\in \bbZ}(\dim V_r) q^r$.

The following lemma is proved in \cite[Lem.~2.5.3]{BGS}.

\smallskip
\begin{lem}
\label{ch3:lem-grad-unique}
Assume that $M$ is an indecomposable $A$-module of finite length. Then,
if $M$ admits a graded lift, this lift is unique up to grading shift and isomorphism.
\qed
\end{lem}

\smallskip
\begin{df}
A $\bbZ$-\emph{category} (or a \emph{graded category}) is an additive category $\widetilde\calC$ with a fixed auto-equivalence $T\colon\widetilde\calC\to \widetilde\calC$. We call $T$ the shift functor. For each $X\in\widetilde \calC$ and $n\in\bbZ$, we set $X\langle n \rangle=T^n(X)$. A functor of $\bbZ$-categories is a functor commuting with the shift functor.
\end{df}

\smallskip
For a graded noetherian ring $A$ the category $\grmod(A)$ is a $\bbZ$-category where $T$ is the shift of grading, i.e., for $M=\oplus_{n\in\bbZ}M_n\in \grmod(A), ~k\in\bbZ$, we have $T(M)_k=M_{k-1}$.

\smallskip
\begin{df}
\label{def_gr-ver}
Let $\calC$ be an abelian category. We say that an abelian $\bbZ$-category $\widetilde\calC$ is a \emph{graded version} of $\calC$ if there exists a functor $F_{\calC}\colon\widetilde \calC\to\calC$ and a graded noetherian ring $A$ such that we have the following commutative diagram, where the horizontal arrows are equivalences of categories and the top horizontal arrow is a functor of $\bbZ$-categories
$$
\begin{CD}
\widetilde \calC\ @>>> \grmod(A)\\
@V{F_\calC}VV                 @V\mbox{forget}VV\\
\calC            @>>> \mod(A).
\end{CD}
$$
\end{df}

\smallskip
In the setup of Definition \ref{def_gr-ver}, we say that an object $\widetilde X\in\widetilde\calC$ is a \emph{graded lift} of an object $X\in\calC$ if we have $F_\calC(\widetilde X)\simeq X$. For objects $X, Y\in\calC$ with fixed graded lifts $\widetilde{X},\widetilde{Y}$ the $\bbZ$-module $\Hom_\calC(X,Y)$ admits a $\bbZ$-grading given by $\Hom_\calC(X,Y)_n=\Hom_{\widetilde\calC}(\widetilde X\langle n\rangle,\widetilde Y)$. In the sequel we will often denote the object $X$ and its graded lift $\widetilde X$ by the same symbol.

\smallskip
\begin{df}
For two abelian categories $\calC_1$, $\calC_2$ with graded versions $\widetilde{\calC_1}$, $\widetilde{\calC_2}$ we say that the functor of $\bbZ$-categories $\widetilde\Phi\colon\widetilde\calC_1\to\widetilde\calC_2$ is a \emph{graded lift} of a functor  $\Phi\colon\calC_1\to\calC_2$ if $F_{\calC_2}\circ\widetilde\Phi=\Phi\circ F_{\calC_1}$.
\end{df}

\subsection{The truncated category $O$}

We can extend the Bruhat order $\leqslant$ on $\widetilde W$ to an order $\leqslant$ on $\widehat W$ in the following way. For each $w_1,w_2\in\widehat W$ we have $w_1\leqslant w_2$ if and only if there exists $n\in \bbZ$ such that $w_1\pi^n,w_2\pi^n\in \widetilde W$ and we have $w_1\pi^n\leqslant w_2\pi^n$ in $\widetilde W$. Note that the order on $\widehat W$ is defined in such a way that for $w_1,w_2\in \widehat W$ we can have $w_1\leqslant w_2$ only if $\widetilde W w_1=\widetilde W w_2$.

Fix $\mu=(\mu_1,\cdots,\mu_e)\in X_e[N]$. Let $W_\mu$ be the stabilizer of the weight $1_\mu\in P$ in $\widetilde W$ (or equivalently in $\widehat W$). 
Let $J_{\mu}$ (resp. $J_{\mu,+}$) be the set of shortest (resp. longest) representatives of the cosets $\widehat W/W_\mu$ in $\widehat W$. For each $v\in \widehat W$ put $^vJ_\mu=\{w\in J_\mu;~w\leqslant v\}$ and $^vJ_{\mu,+}=\{w\in J_{\mu,+};~w\leqslant v\}$. They are finite posets.

Assume that $R$ is a local deformation ring. Let $^vO_{\mu,R}$ be the Serre subcategory of $O_{\mu,R}$ generated by the modules $\Delta^{w(1_\mu)}_R$ with $w\in {^vJ_\mu}$. This is a highest weight category, see \cite[Lem.~3.7]{SVV}.
Note that the definition of the category $^vO_{\mu,R}$ does not change if we replace $v$ by the minimal length element in $vW_\mu$ (i.e., by an element of $J_\mu$). However, in some situations it will be more convenient to assume that $v$ is maximal in $vW_\mu$ (and not minimal).

Recall the decomposition
$$O_{\mu,R}=\bigoplus_{n\in\bbZ}\calO_{\pi^n(1_\mu),R}$$
in (\ref{ch3:eq_dec-O-What-Wtilde}). Note that the definition of the order on $\widehat W$ implies that the category $^vO_{\mu,R}$ lies in $\calO_{\pi^n(1_\mu),R}$, where $n\in\bbZ$ is such that $v\in \widetilde W\pi^n$.

\subsection{Linkage}
We still consider the non-parabolic category $O$. In particular we have $l=N$.

Let $\bfk$ be a deformation ring that is a field. Recall that the affine Weyl group $\widetilde W$ is generated by reflections $s_\alpha$, where $\alpha$ is a real affine root. Now we consider the following equivalence relation $\sim_\bfk$ on $P$. We define it as the equivalence relation generated by $\lambda_1\sim_\bfk\lambda_2$  when $\widetilde\lambda_1+\widehat\rho=s_\alpha(\widetilde\lambda_2+\widehat\rho)$ for some real affine root $\alpha$. The definition of $\sim_\bfk$ depends on $\bfk$ because the definitions of $\widetilde\lambda$ and $\widehat\rho$ depend on the elements $\tau_r,\kappa\in \bfk$.

Now, let $R$ be a deformation ring that is a local ring with residue field $\bfk$. Then for $\lambda_1,\lambda_2\in P$ we write $\lambda_1\sim_R\lambda_2$ if and only if we have $\lambda_1\sim_\bfk\lambda_2$. Note that the definition of the equivalence relation above is motivated by \cite[Thm.~3.2]{Fie-cen}.

In the particular case when $R$ is a local deformation ring, the equivalence relation $\sim_R$ coincides with the equivalence relation $\sim_\bbC$ because we have $\tau_r=0$ and $\kappa=e$ in the residue field of $R$. The relation $\sim_\bbC$ can be easily described in terms of the $e$-action of $\widehat W$ on $P$, introduced in Section \ref{ch3:subs_ext-aff}. We have $\lambda_1\sim_\bbC\lambda_2$ if and only the elements $\lambda_1+\rho$ and $\lambda_2+\rho$ of $P^{(e)}$ are in the same $\widetilde W$-orbit.

\smallskip
\begin{rk}
\label{ch3:rk-small-orbits}
Let $\bfk$ be as above.

$(a)$ Assume that for each $r,t\in[1,l]$ such that $r\ne t$ we have $\tau_r-\tau_t\not\in \bbZ$. In this case the equivalence relation $\sim_\bfk$ is the equality.

$(b)$ Assume that we have $\tau_r-\tau_t\in\bbZ$ for a unique couple $(r,t)$ as above. In this case each equivalence class with respect to $\sim_\bfk$ contains at most two elements.  

$(c)$ Let $R$ be as local deformation ring in general position with the field of fractions $K$. By $(a)$, the equivalence relation $\sim_K$ is just the equality. Now, let $\frakp$ be a prime ideal of height $1$ in $R$. In this case, each equivalence class with respect to $\sim_{R_\frakp}$ contains at most two elements (this follows from \cite[Prop.~5.22~$(a)$]{RSVV}, $(a)$ and $(b)$).  
\end{rk}

\smallskip
The relation $\sim_R$ yields a decomposition of the category $O_{-e,R}$ in a direct sum of subcategories, see \cite[Prop.~2.8]{Fie-cen}. More precisely, let $\Lambda$ be an $\sim_R$-equivalence class in $P$. Let $\calO_{\Lambda,R}$ be the Serre subcategory of $O_{-e,R}$ generated by $\Delta(\lambda)$ for $\lambda\in\Lambda$. Then we have
\begin{equation}
\label{ch3:eq_block-decomp-Fie}
O_{\mu,R}=\bigoplus_{\Lambda\subset P[\mu]-\rho}\calO_{\Lambda,R}.
\end{equation}
For example, if $R$ is a local deformation ring, then this decomposition coincides with (\ref{ch3:eq_dec-O-What-Wtilde}).
The following lemma explains what happens after the base change, see \cite[Lem.~2.9,~Cor.~2.10]{Fie-cen}.

\smallskip
\begin{lem}
The $R$ and $T$ be deformation rings that are local and let $R\to T$ be a ring homomorphism.

$(a)$ The equivalence relation $\sim_T$ is finer than the relation $\sim_R$.

$(b)$ Let $\Lambda$ be an equivalence class with respect to $\sim_R$. Then $T\otimes_R \calO_{\Lambda,R}$ is equal to $\bigoplus_{\Lambda'}\calO_{\Lambda',T}$, where the sum is taken by all $\sim_T$-equivalence classes $\Lambda'$ in $\Lambda$.
\qed
\end{lem}

\smallskip
\begin{df}
\label{ch3:def_generic}
We say that the category $\calO_{\Lambda,R}$ is \emph{generic} if $\Lambda$ contains a unique element and \emph{subgeneric} if it contains exactly two elements.
\end{df}
\smallskip

More details about the structure of generic and subgeneric categories can be found in \cite[Sec.~3.1]{Fie-str}.

\subsection{Centers}
We assume that $R$ is a deformation ring that is a local ring with the residue field $\bfk$ and the field of fractions $K$. Recall that we have $l=N$ because we consider the non-parabolic category $O$.

Let $\Lambda$ be an equivalence class in $P$ with respect to $\sim_R$. Consider the category $\calO_{\Lambda,R}$ as in (\ref{ch3:eq_block-decomp-Fie}). There is a partial order $\leqslant$ on $\Lambda$ such that $\lambda_1\leqslant \lambda_2$ when $\widetilde\lambda_2-\widetilde\lambda_1$ is a sum of simple roots. There exists an element $\lambda\in\Lambda$ such that $\Lambda$ is minimal in $\Lambda$ with respect to this order. Assume that $\Lambda$ is finite.

\smallskip
\begin{df}
The \emph{antidominant projective module in $\calO_{\Lambda,R}$} is the projective cover in $\calO_{\Lambda,R}$ of the simple module $L_R(\lambda)$, where $\lambda$ is the minimal element in $\Lambda$. (The existence of the protective cover as above is explained in \cite[Thm.~2.7]{Fie-cen}.)
\end{df}

\smallskip
This notion has no sense if $\Lambda$ is infinite. However we can consider the truncated version. Fix $v\in \widehat W$. We have a truncation of the decomposition (\ref{ch3:eq_block-decomp-Fie}):
\begin{equation}
\label{ch3:eq_block-decomp-Fie-trunc}
^vO_{\mu,R}=\bigoplus_{\Lambda}{^v\calO}_{\Lambda,R},
\end{equation}
where we put ${^v\calO}_{\Lambda,R}=\calO_{\Lambda,R}\cap{^vO}_{\mu,R}$.

By \cite[Thm.~2.7]{Fie-cen} each simple module in $^vO_{\mu,R}$ has a projective cover.
As above, we denote by $\lambda$ the element of $\Lambda$ that is minimal in $\Lambda$ with respect to the order $\leqslant$.

\smallskip
\begin{df}
The \emph{antidominant projective module in $^v\calO_{\Lambda,R}$} is the projective cover in $^v\calO_{\Lambda,R}$ of the simple module $L_R(\lambda)$.
\end{df}

\smallskip
From now on we assume that $R$ is a local deformation ring in general position, see Section \ref{ch3:subs_def-ring}. Let $\bfk$ and $K$ be the residue field and the field of fractions of $R$ respectively.  We set $h_0=\tau_l-\tau_1-\kappa+e$ and $h_r=\tau_{r+1}-\tau_{r}$ for $r\in[1,l-1]$. We have $h_r\ne 0$ for each $r\in[0,l-1]$ because the ring is assumed to be in general position. Under the assumption on $R$, the decomposition (\ref{ch3:eq_block-decomp-Fie-trunc}) contains only one term. Let $^vP^\mu_R$ be the antidominant projective module in $^vO_{\mu,R}$, i.e., $^vP^\mu_R$ is the projective cover of $L_R^{\pi^n(1_\mu)}$, where $n$ is such that we have $\pi^n\leqslant v$.

\smallskip
\begin{lem}
\label{ch3:lem_antid-proj}
$(a)$ The module $^vP^\mu_R$ has a $\Delta$-filtration such that each Verma module in the category $^vO_{\mu,R}$ appears exactly ones as a subquotient in this $\Delta$-filtration.

$(b)$ For each base change $R'\otimes_R\bullet$, where $R'$ is a deformation ring that is local, the module $R'\otimes_R {^vP}^\mu_{R}$ splits into a direct sum of anti-dominant projective modules in the blocks of the category $^vO_{\mu,R'}$.

\end{lem}
\begin{proof}[Proof]
The first part in $(a)$ holds by \cite[Thm.~2~(2)]{Fie-str} and the second part in $(a)$ holds by the proof of \cite[Lem.~4]{Fie-str}. Finally, $(b)$ follows from \cite[Rem.~5]{Fie-str}.
\end{proof}

\smallskip
We will need the following lemma.

\smallskip
\begin{lem}
\label{ch3:lem_center-of-modA}
Let $A$ be a ring. Then the center $Z(\mod(A))$ of the category $\mod(A)$ is isomorphic to the center $Z(A)$ of the ring $A$.
\end{lem}
\begin{proof}
There is an obvious injective homomorphism $Z(A)\to Z(\mod(A))$. We need only to check that it is also surjective.

Let $z$ be an element of the center of $\mod(A)$. By definition, $z$ consists of an endomorphism $z_M$ of $M$ for each $M\in\mod(A)$ such that these endomorphisms commute with all morphisms between the modules in $\mod(A)$. Then $z_A$ is an endomorphism of the $A$-module $A$ that commutes with each other endomorphism of the $A$-module $A$. Thus $z_A$ is a multiplication by an element $a$ in the center of $A$.

Now we claim that for each module $M\in\mod(A)$ the endomorphism $z_M$ is the multiplication by $a$. Fix $m\in M$. Let $\phi\colon A\to M$ be the morphism of $A$-modules sending $1$ to $m$. We have $\phi\circ z_A=z_M\circ\phi$. Then
$$
z_M(m)=z_M\circ \phi(1)=\phi\circ z_A(1)=\phi(a)=am.
$$
This completes the proof.
\end{proof}

\smallskip
Let $Z_{\mu,R}$ (resp. $^vZ_{\mu,R}$) be the center of the category $O_{\mu,R}$ (resp. $^vO_{\mu,R}$).

\smallskip
\begin{prop}
\label{ch3:prop_eval-on-proj}
The evaluation homomorphism $^vZ_{\mu,R}\to \End(^vP^\mu_R)$ is an isomorphism.
\qed
\end{prop}
\begin{proof}[Proof]
The statement is proved in \cite[Lem.~5]{Fie-str}. There are however some subtle points that we explain.

Firstly, the statement of \cite[Lem.~5]{Fie-str} announces the result for the non-truncated category $\calO$. But in fact, the main point of the proof of \cite[Lem.~5]{Fie-str} is to show the statement first in the truncated case.

Secondly, \cite[Lem.~5]{Fie-str} assumes that the deformation ring $R$ is the localization of the symmetric algebra $S(\widehat\frakh)$ at the maximal ideal generated by $\widehat\frakh$. Let us sketch the argument of \cite[Lem.~5]{Fie-str} to show that it works well for our assumption on $R$.

Denote by $ev_R\colon {^vZ}_{\mu,R}\to \End(^vP^\mu_R)$ the homomorphism in the statement.
Let $I(R)$ be the set of prime ideals of height $1$ in $R$.
We claim that we have
\begin{equation}
\label{ch3:eq_Z-as-inter-loc}
^vZ_{\mu,R}=\bigcap_{\frakp\in I(R)}{^vZ}_{\mu,R_\frakp},
\end{equation}
where the intersection is taken inside of $^vZ_{\mu,K}$.
Really, let $^vA_{\mu,R}$ be the endomorphism algebra of the minimal projective generator of $^vO_{\mu,R}$. We have an equivalence of categories $^vO_{\mu,R}\simeq \mod(^vA_{\mu,R})$. By Lemma \ref{ch3:lem_center-of-modA} we have an algebra isomorphism $^vZ_{\mu,R}\simeq Z(^vA_{\mu,R})$. The same is true if we replace $R$ by $R_\frakp$ or $K$. By \cite[Prop.~2.4]{Fie-cen} we have $^vA_{\mu,R_\frakp}\simeq R_\frakp \otimes_R{^vA}_{\mu,R}$, $^vA_{\mu,K}\simeq K \otimes_R{^vA}_{\mu,R}$. The algebra ${^vA}_{\mu,R}$ is free over $R$ as an endomorphism algebra of a projective module in $^vO_{\mu,R}$. Thus we have ${^vA}_{\mu,R}=\bigcap_{\frakp\in I(R)}{^vA}_{\mu,R_\frakp}$, where the intersection is taken in ${^vA}_{\mu,K}$. If we intersect each term in the previous formula with $^vZ_{\mu,K}=Z({^vA}_{\mu,K})$, we get (\ref{ch3:eq_Z-as-inter-loc}).

 Similarly, we have
$$
\End(^vP^\mu_{R})=\bigcap_{\frakp\in I(R)}\End(R_\frakp\otimes_R{^vP}^\mu_{R})
$$
incide of $\End(K\otimes_R{^vP}^\mu_{R})$.

To conclude, we only need to show that the evaluation homomorphisms
$$
ev_{R_\frakp}\colon{^vZ}_{\mu,R_\frakp}\to\End(R_\frakp\otimes_R{^vP}^{\mu}_{R}),\qquad ev_K\colon{^vZ}_{\mu,K}\to\End(K\otimes_R{^vP}^{\mu}_R)
$$
are isomorphisms for each $\frakp\in I(R)$ and that the following diagram is commutative
$$
\begin{CD}
\End(R_\frakp\otimes_R{^vP}^{\mu}_{R})@>>> \End(K\otimes_R{^vP}^{\mu}_R)\\
@A{ev_{R_\frakp}}AA @A{ev_K}AA\\
^vZ_{\mu,R_\frakp}@>>> ^vZ_{\mu,K}\\
\end{CD}
$$

The commutativity of the diagram is obvious. Since $R$ is in general position, the category $^vO_{\mu,K}$ is semisimple, see Remark \ref{ch3:rk-small-orbits}. Moreover, for each $\frakp\in I(R)$, the category $^vO_{\mu,R_\frakp}$ is a direct sum of blocks with at most two Verma modules in each, see Remark \ref{ch3:rk-small-orbits}. Similarly, by Lemma \ref{ch3:lem_antid-proj} $(b)$ the localisation $R_\frakp\otimes_R{^vP}^\mu_{R}$ of the antidominant projective module splits into a direct sum of antidominant projective modules. Now, the invertibility of $ev_{R_\frakp}$ and $ev_{K}$ follows from \cite[Lem.~2]{Fie-str}. 

\end{proof}

\smallskip
We will need the following lemma.

\smallskip
\begin{lem}
\label{ch3:lem_trunc-categ-gen}
Assume that $\calC_1$ is an abelian category and $\calC_2$ is an abelian subcategory of $\calC_1$. Let $i\colon \calC_2\to \calC_1$ be the inclusion functor. For each object $M\in\calC_1$ we assume that $M$ has a maximal quotient that is in $\calC_2$ and we denote this quotient by $\tau(M)$. Then we have the following.

$(a)$ The functor $\tau\colon\calC_2\to\calC_1$ is left adjoint to $i$.

$(b)$ Let $L$ be a simple object in $\calC_2$. Assume that $L$ has a projective cover $P$ in $\calC_1$. Then $\tau(P)$ is a projective cover of $L$ in $\calC_2$.
\end{lem}
\begin{proof}
Take $M\in\calC_1$ and $N\in\calC_2$. For each homomorphism $f\colon M\to N$ we have $M/\Ker f\simeq \Im f\in\calC_2$. Thus $\Ker f$ must contain the kernel of $M\to\tau(M)$. This implies that each homomorphism $f\colon M\to N$ factors through $\tau(M)$. This proves $(a)$.

Now, we prove $(b)$. We have a projective cover $p\colon P\to L$ in $\calC_1$. First, we clam that the object $\tau(P)$ is projective in $\calC_2$. Really, by $(a)$ the functors from $\calC_2$ to the category of $\bbZ$-modules  $\Hom_{\calC_2}(\tau(P),\bullet)$ and $\Hom_{\calC_1}(P,\bullet)$ are isomorphic. Thus the first of them should be exact because the second one is exact by the projectivity of $P$. This shows that $\tau(P)$ is projective in $\calC_2$. Denote by $\overline p$ the surjection $\overline p\colon \tau(P)\to L$ induced by $p\colon P\to L$. Let $t$ be the surjection $t\colon P\to\tau(P)$. We have $p=\overline p\circ t$. Now we must prove that each proper submodule $K\subset \tau(P)$ is in $\ker \overline p$. Really, if this is not true for some $K$, then $p(t^{-1}(K))$ is nonzero. Then we have $p(t^{-1}(K))=L$ because the module $L$ is simple. Then by the definition of a projective cover we must have $t^{-1}(K)=P$. This is impossible because $t$ is surjective and $K$ is a proper submodule of $\tau(P)$.
\end{proof}

\smallskip
\begin{rk}
\label{ch3:rk_grade-trunc-gen}
Let $A$ be a graded noetherian ring. Let $I\subset A$ be a homogeneous ideal. Put $\calC_1=\mod(A)$, $\calC_2=\mod(A/I)$, $\widetilde\calC_1=\grmod(A)$, $\widetilde\calC_2=\grmod(A/I)$. There is an obvious inclusion of categories $i\colon\calC_2\to\calC_1$ and it has an obvious graded lift $\widetilde i\colon\widetilde\calC_2\to\widetilde\calC_1$. The left adjoint functor $\tau$ to $i$ is defined by $\tau(M)=M/IM$ and the left adjoint functor $\widetilde\tau$ to $\widetilde i$ is also defined by $\widetilde\tau(M)=M/IM$. This implies that the functor $\widetilde\tau$ is a graded lift of $\tau$.
\end{rk}

\smallskip
Recall that we denote by $s_0,\cdots,s_{N-1}$ the simple reflections in $\widetilde W$.

\smallskip
\begin{prop}
\label{ch3:prop_eval-on-Verma-full}
We have an isomorphism
$$
Z_{\mu,R}\simeq\left\{(z_w)\in\prod_{w\in J_\mu}R;~z_{w}\equiv z_{s_r w}~\mod~h_r~\forall r\in[0,N-1],w\in {J_\mu}\cap {s_rJ_\mu}\right\}
$$
which maps an element $z\in {Z}_{\mu,R}$ to the tuple $(z_w)_{w\in {J_\mu}}$ such that $z$ acts on the Verma module $\Delta^{w(1_\mu)}_R$ by $z_w$.
\end{prop}
\begin{proof}[Proof]
The statement is proved in \cite[Thm.~3.6]{Fie-cen} in the case where $R$ is the localization of the symmetric algebra $S(\widehat\frakh)$ at the maximal ideal generated by $\widehat\frakh$. Similarly to Proposition \ref{ch3:prop_eval-on-proj}, the same proof works under our assumption on the ring $R$.
\end{proof}

\smallskip
Proposition \ref{ch3:prop_eval-on-Verma-full} has the following truncated version that can be proved in the same way using the approach of localizations at the prime ideals of height 1. (See, for example, the proof of Proposition \ref{ch3:prop_eval-on-proj}). For each such localization the result becomes clear by \cite[Cor.~3.5]{Fie-cen} and Remark \ref{ch3:rk-small-orbits}.

\smallskip
\begin{prop}
\label{ch3:prop_eval-on-Verma-trucn}
We have an isomorphism
\begin{equation}
\label{ch3:eq_eval-on-Verma-trunc}
^vZ_{\mu,R}\simeq\left\{(z_w)\in\prod_{w\in ^vJ_\mu}R;~z_{w}\equiv z_{s_r w}~\mod~h_r~\forall r\in[0,N-1],w\in {^vJ_\mu}\cap s_r{^vJ_\mu}\right\}
\end{equation}
which maps an element $z\in {^vZ}_{\mu,R}$ to the tuple $(z_w)_{w\in {^vJ_\mu}}$ such that $z$ acts on the Verma module $\Delta^{w(1_\mu)}_R$ by $z_w$.
\qed
\end{prop}



\smallskip
For each $v\in\widehat W$, set $^vJ=\{w\in \widehat W;~w\leqslant v\}$ and
$$
^vZ_{R}\simeq\left\{(z_w)\in\prod_{w\in {^vJ}}R;~z_{w}\equiv z_{s_r w}~\mod~h_{r}~\forall r\in[0,N-1],w\in {^vJ}\cap s_r{^vJ}\right\}.
$$
If $v$ is in $J_{\mu,+}$, then the group $W_\mu$ acts on $^vZ_R$ by $w(z)=z'$ where $z'_x=z_{xw^{-1}}$ for each $x\in {^vJ}$. Note that the algebra $^vZ^{W_\mu}_{R}$ of $W_\mu$-invariant elements in $^vZ_{R}$ is obviously isomorphic to the right hand side in (\ref{ch3:eq_eval-on-Verma-trunc}). Thus Proposition \ref{ch3:prop_eval-on-Verma-trucn} identifies the center $^vZ_{\mu,R}$ of $^vO_{\mu,R}$ with $^vZ^{W_\mu}_{R}$.

\subsection{The action on standard and projetive modules}
As above, we fix $k\in[0,e-1]$ and set $\mu'=\mu-\alpha_k$. 
From now on, we assume that $R$ is as in Assumption 1 with residue field $\bfk=\bbC$.

From now on we also always assume that we have
$W_\mu\subset W_{\mu'}$.
This happens if and only if we have $\mu_k=1$. In this case we have $J_{\mu'}\subset J_\mu$ and $J_{\mu',+}\subset J_{\mu,+}$. From now on we always assume that the element $v$ is in $J_{\mu',+}$ (thus $v$ is also in $J_{\mu,+}$). We have an inclusion of algebras ${^vZ}_{\mu',R}\subset {^vZ}_{\mu,R}$ because ${^vZ}_{\mu',R}\simeq {^vZ}_{R}^{W_{\mu'}}$ and ${^vZ}_{\mu,R}\simeq {^vZ}_{R}^{W_{\mu}}$. Let $\Res\colon \mod(^vZ_{\mu,R})\to \mod(^vZ_{\mu',R})$ and $\Ind\colon \mod(^vZ_{\mu',R})\to \mod(^vZ_{\mu,R})$ be the restriction and the induction functors. We may write $\Res_{\mu}^{\mu'}$ and $\Ind_{\mu'}^{\mu}$ to specify the parameters.

It is easy to see on Verma modules using two lemmas below that the functors $E_k$ and $F_k$ restrict to functors of truncated categories
$$
F_k\colon {^vO}^\Delta_{\mu,R}\to {^vO}^\Delta_{\mu',R},\qquad E_k\colon {^vO}^\Delta_{\mu',R}\to {^vO}^\Delta_{\mu,R}.
$$

\smallskip
\begin{lem}
\label{ch3:lem_Fk-Verma-spcase1}
For each $w\in\widehat W$, we have $F_k(\Delta^{w(1_\mu)}_R)\simeq\Delta^{w(1_{\mu'})}_R$.
\end{lem}
\begin{proof}
Since $\mu_k=1$, the weight $w(1_\mu)\in P$ has a unique coordinate containing an element congruent to $k$ modulo $e$.  Let $r\in[1,N]$ be the position number of this coordinate.
By Proposition \ref{ch3:prop_functors-on-O-gen} $(e)$, we have $[F_k(\Delta^{w(1_\mu)}_R)]=[\Delta^{w(1_\mu)+\epsilon_r}_R]$. The equality of classes in the Grothendieck group implies that we have an isomorphism of modules $F_k(\Delta^{w(1_\mu)}_R)\simeq \Delta^{w(1_{\mu})+\epsilon_r}_R$. Finally, since $w(1_\mu)+\epsilon_r=w(1_{\mu'})$, we get $F_k(\Delta^{w(1_\mu)}_R)\simeq\Delta^{w(1_{\mu'})}_R$.
\end{proof}

\smallskip
\begin{lem}
\label{ch3:lem_Ek-Verma-spcase1}
For each $w\in\widehat W$, we have $[E_k(\Delta^{w(1_{\mu'})}_R)]=\sum_{z\in W_{\mu'}/W_\mu}[\Delta^{wz(1_{\mu})}_R]$.
\end{lem}
\begin{proof}
By  Proposition \ref{ch3:prop_functors-on-O-gen} $(e)$, we have
\begin{equation}
\label{ch3:eq_Ek-on-Vermas-1}
[E_{k}(\Delta^{w(1_{\mu'})}_R)]=\sum_{r} [\Delta^{w(1_{\mu'})-\epsilon_r}_R],
\end{equation}
where the sum in taken by all indices $r\in[1,N]$ such that the $r$th coordinate of $w(1_\mu)$ is congruent to $k+1$ modulo $e$. For each such $r$ we have $w(1_{\mu'})-\epsilon_r=wz(1_\mu)$ for a unique element $z\in W_{\mu'}/W_{\mu}$. Moreover, the obtained map $r\mapsto z$ is a bijection from the set of possible indices $r$ to $W_{\mu'}/W_{\mu}$. Thus $(\ref{ch3:eq_Ek-on-Vermas-1})$ can be rewritten as
$$
[E_{k}(\Delta^{w(1_{\mu'})}_R)]=\sum_{z\in W_{\mu'}/W_{\mu}} [\Delta^{wz(1_{\mu})}_R].
$$
\end{proof}


\smallskip
\begin{lem} 
We have $E_{k}(^vP^{\mu'}_R)\simeq {^vP}^{\mu}_R$.
\end{lem}
\begin{proof}[Proof]
By Lemma \ref{ch3:lem_antid-proj} $(a)$, the class $[^vP^{\mu}_R]$ of $^vP^{\mu}_R$ in the Grothendieck group of $^vO_{\mu,R}^\Delta$ is the sum of all classes of Verma modules of the category $^vO^{\Delta}_{\mu,R}$ and similarly for $[^vP^{\mu'}_R]$. Taking the sum in the equality in Lemma \ref{ch3:lem_Ek-Verma-spcase1} over all $w\in {^vJ_{\mu'}}$, we get $[E_{k}(^vP^{\mu'}_R)]=[^vP^{\mu}_R]$. Finally, this yields an isomorphism $E_{k}(^vP^{\mu'}_R)\simeq {^vP}^{\mu}_R$ because the modules $E_{k}(^vP^{\mu'}_R)$ and $^vP^{\mu}_R$ are projective.
\end{proof}

\smallskip
Fix an isomorphism $E_{k}(^vP^{\mu'}_R)\simeq {^vP}^{\mu}_R$ as above. Then by functoriality it yields an algebra homomorphism $\End(^vP^{\mu'}_R)\to \End(^vP^{\mu}_R)$.

\smallskip
\begin{lem}
\label{ch3:lem_diag-cent-proj-mu'-mu}
The following diagram of algebra homomorphisms is commutative
$$
\begin{CD}
\End(^vP^{\mu'}_R)@>>> \End(^vP^{\mu}_R)\\
@AAA @AAA\\
^vZ_{\mu',R}@>>> {^vZ}_{\mu,R},
\end{CD}
$$
where the top horizontal map is as above, the bottom horizontal map is the inclusion and the vertical maps are the isomorphisms from Proposition \ref{ch3:prop_eval-on-proj}.
\end{lem}
\begin{proof}[Proof]
Note that each element in $\End(^vP^{\mu}_R)$ is induced by the center ${^vZ}_{\mu,R}$. In partilucar, each endomorphism of $^vP^{\mu}_R$ preserves each submodule of $^vP^{\mu}_R$. Moreover, by Lemma \ref{ch3:lem_antid-proj} $(a)$, each Verma module in $^vO_{\mu,R}$ is isomorphic to a subquotient of $^vP^{\mu}_R$. Thus, by Proposition \ref{ch3:prop_eval-on-proj} and Proposition \ref{ch3:prop_eval-on-Verma-trucn}, an element of $\End(^vP^{\mu}_R)$ is determined by its action on the subquotients of a $\Delta$-filtration of $^vP^{\mu}_R$.

Fix an element $z=(z_w)$ in $^vZ_{\mu',R}$, see Proposition \ref{ch3:prop_eval-on-Verma-trucn}. Fix also a $\Delta$-filtration of $^vP^{\mu'}_R$. The element $z$ acts on $^vP^{\mu'}_R$ in such a way that it preserves each component of the $\Delta$-filtration and the induced action on the subquotient $\Delta^{w(1_{\mu'})}_R$ of $^vP^{\mu'}_R$ is the multiplication by $z_w$.

For each $w\in \widehat W$, the module $E_{k}(\Delta^{w(1_{\mu'})}_R)$ is $\Delta$-filtered. The subquotients in this $\Delta$-filtration can be described by Lemma \ref{ch3:lem_Ek-Verma-spcase1}. 
Since the functor $E_k$ is exact, the $\Delta$-filtration of $^vP^{\mu'}_R$ induces a $\Delta$-filtration of $^vP^{\mu}_R\simeq E_{k}(^vP^{\mu'}_R)$.  Thus the image of $z$ by
$$
^vZ_{\mu',R}\to\End(^vP^{\mu'}_R)\to\End(^vP^{\mu}_R)
$$
acts on the subquotients of the $\Delta$-filtration of $^vP^{\mu}_R$ in the following way: it acts on the subquotient $\Delta^{w(1_{\mu})}_R$ of $^vP^{\mu}_R$ by the multiplication by $z_w$. On the other hand, the image of $z$ by
$$
^vZ_{\mu',R}\to{^vZ}_{\mu,R}\to\End(^vP^{\mu}_R)
$$
acts on the subquotients in the same way. This proves the statement because an element of $\End(^vP^{\mu}_R)$ is determined by its action on the subquotients of a $\Delta$-filtration of $^vP^{\mu}_R$.

\end{proof}

\subsection{The functor $\bbV$}
Now, we assume that $v$ is an arbitrary elements of $\widehat W$.  
We have a functor
$$
\bbV_{\mu,R}\colon {^vO}_{\mu,R}\to \mod(^vZ_{\mu,R}), \quad M\mapsto \Hom(^vP_R^\mu,M).
$$
Set $^vZ_{\mu}=\bbC\otimes_R{^vZ}_{\mu,R}$ and $^vZ=\bbC\otimes_R{^vZ}_{R}$. By \cite[Prop.~2.6]{Fie-cen} we have $\bbC\otimes_R{^vP}^\mu_{R}={^vP}^\mu$. Next, \cite[Prop.~2.7]{Fie-cen} yields an algebra isomorphism $^vZ_{\mu}\simeq \End(^vP_{\mu})$. Now, consider the functor
$$
\bbV_{\mu}\colon {^vO}_{\mu}\to \mod(^vZ_{\mu}),\quad M\mapsto \Hom(^vP^\mu,M).
$$

A Koszul grading on the category $^vO_{\mu}$ is constructed in \cite{SVV}. Let us denote by $^v\widetilde O_{\mu}$ the graded version of this category.

The functor $\bbV$ above has following properties.

\smallskip
\begin{prop}
\label{ch3:prop_Vk}
$(a)$ The functor $\bbV_{\mu,R}$ is fully faithful on $^vO_{\mu,R}^\Delta$.

$(b)$ The functor $\bbV_{\mu}$ is fully faithful on projective objects in $O_{\mu}$.

$(c)$ The functor $\bbV_{\mu}$ admits a graded lift $\widetilde\bbV_{\mu}\colon {^v\widetilde O}_{\mu}\to \grmod(^vZ_{\mu})$.
\end{prop}
\begin{proof}[Proof]
Part $(a)$ is \cite[Prop.~2]{Fie-str} (1). Part $(b)$ is \cite[Prop.~4.50]{SVV} $(b)$. Part $(c)$ is given in the proof of \cite[Lem.~5.10]{SVV}.
\end{proof}

\subsection{The cohomology of Schubert varieties}

All cohomology groups in this section have coefficients in $\bbC$.

Set $G=GL_N$. Let $B\subset G(\bbC((t)))$ be the standard Borel subgroup. Let $P_\mu\subset G(\bbC((t)))$ be the parabolic subgroup with Lie algebra $\widehat\bfp_{\mu}$.
Let $X_\mu$ be the partial affine flag ind-scheme $G(\bbC((t)))/P_\mu$. The affine Bruhat cells in $X_\mu$ are indexed by $J_\mu$. For $w\in J_\mu$ we denote by $X_{\mu,w}$ (resp. $\overline X_{\mu,w}$) the corresponding finite dimensional affine Bruhat cell (resp. Schubert variety). Note that we have $X_{\mu,w}\simeq \bbC^{\ell(w)}$. The following statement is proved in \cite[Prop.~4.43~(a)]{SVV}.

\smallskip
\begin{lem}
Assume $v\in J_\mu\cap \widetilde W$. There is an isomorphism of graded algebras between $^vZ_{\mu}$ and the cohomology $H^*(\overline X_{\mu,v})$.
\qed
\end{lem}

\smallskip
Now we are going to extend the notions $X_{\mu,w}$ and $\overline X_{\mu,w}$ to an arbitrary $w\in J_\mu$ in order to get an extended version of the previous lemma.

Let $\pi$ be the cyclic shift the of Dynkin diagram of type $A_{N-1}^{(1)}$ that takes the root $\alpha_i$ to the root $\alpha_{i-1}$ for $i\in\bbZ/N\bbZ$. It yields an automorphism $\pi\colon G\to G$. 
Then for each $n\in\bbZ$ we have a parabolic subgroup $\pi^n(P_\mu)\subset G(\bbC((t)))$. Recall that the symbol $\pi$ also denotes an element of $\widehat W$, see Section \ref{ch3:subs_ext-aff}. Let $X^n_\mu$ be the partial affine flag ind-scheme defined in the same way as $X_\mu$ with respect to the parabolic subgroup $\pi^n(P_\mu)\subset G(\bbC((t)))$ instead of $P_\mu$. In particular, we have $X_\mu=X^0_\mu$. Let us use the abbreviation  $\pi^n(W_\mu)$ for the subgroup $\pi^nW_\mu\pi^{-n}$ of $\widetilde W$. Note that the group $\pi^n(W_\mu)$ is the Weyl group of the Levi of $\pi^n(P_\mu)$. The Bruhat cells and the Schubert varieties in $X_\mu^n$ are indexes by the shortest representatives of the cosets in $\widetilde W/\pi^n(W_{\mu})$. For such a representative $w$ let $X_{\mu,w}^n$ (resp. $\overline X_{\mu,w}^n$) be the Bruhat cell (resp. Schubert variety) in $X^n_\mu$.

Assume that $v\in J_\mu$. Then $v$ has a unique decomposition of the form $v=w\pi^n$, such that $w$ is minimal in $w\pi^{n}(W_{\mu})$. Then we set $X_{\mu,v}=X^n_{\mu,w}$ and $\overline X_{\mu,v}=\overline X^n_{\mu,w}$. Note that for $v\in J_\mu$ we have $X_{\mu,v}\simeq \bbC^{\ell(v)}$. 
We get the following generalization of the lemma above.

\smallskip
\begin{lem}
\label{ch3:lem_coh-Schub-extended}
Assume $v\in J_\mu$. There is an isomorphism of graded algebras between $^vZ_{\mu}$ and the cohomology $H^*(\overline X_{\mu,v})$.
\end{lem}
\begin{proof}
Consider the decomposition $v=w\pi^n$ as above.  By definition, the truncated category $^vO_{\mu}$ is a Serre subcategory of $\calO_{\pi^n(1_\mu)}$. It is generated by modules $L^{x\pi^n(1_\mu)}$, where $x\in \widetilde W$ is such that $x\leqslant w$. Note also that the stabilizer of the weight $\pi^n(1_\mu)$ in $\widetilde W$ is $\pi^n(W_\mu)$. Then, by \cite[Prop.~4.43~(a)]{SVV}, we have an isomorphism of graded algebras $^vZ_{\mu}=H^*(\overline X_{\mu,w}^n)$. On the other hand the variety $\overline X_{\mu,v}$ is defined as $\overline X_{\mu,w}^n$.
\end{proof}

\smallskip
Now, assume that $v\in J_{\mu',+}$. Recall that in this case we have an inclusion of algebras ${^vZ}_{\mu',R}\subset {^vZ}_{\mu,R}$ because of the assumption $W_{\mu}\subset W_{\mu'}$. We want to show that after the base change we get an inclusion of algebras ${^v}Z_{\mu'}\subset {^vZ}_{\mu}$. However, this is not obvious because the functor $\bbC\otimes_R\bullet$ is not left exact. But this fact can be justified using geometry. The injectivity of the homomorphism ${^v}Z_{\mu'}\to {^vZ}_{\mu}$ is a consequence of Lemma \ref{ch3:lem_Zmu-Zmu'-mod} below.

Denote by $w_\mu$ the longest elements in $W_{\mu}$. The shortest elements in $vW_{\mu}$ and $vW_{\mu'}$ are respectively $vw_\mu$ and $vw_{\mu'}$. By Lemma \ref{ch3:lem_coh-Schub-extended}, we have algebra isomorphisms ${^vZ}_{\mu}\simeq H^*(\overline X_{\mu,vw_\mu})$ and ${^vZ}_{\mu'}\simeq H^*(\overline X_{\mu',vw_{\mu'}})$. 

The group $\widetilde W$ is a Coxeter group. In particular we have a length function $\ell\colon \widetilde W\to \bbN$. We can extend it to $\widehat W$ be setting $\ell(w\pi^n)=\ell(w)$ for each $n\in\bbZ$ and $w\in\widetilde W$. Now we are ready to prove the following result.

\smallskip
\begin{lem}
\label{ch3:lem_Zmu-Zmu'-mod}
There is the following isomorphism of graded $^vZ_{\mu'}$-modules
$$
^vZ_{\mu}\simeq \bigoplus_{r=0}^{\mu_{k+1}}{^vZ}_{\mu'} \langle 2r\rangle.
$$
\end{lem}
\begin{proof}
Let $J_{\mu'}^\mu$ be the set of shortest representatives of classes in $W_{\mu'}/W_\mu$.
We have the following decomposition into affine cells 
$$
\overline X_{\mu,vw_\mu}=\coprod_{w\in {^{v}J}_\mu}X_{\mu,w}=\coprod_{w\in {^{v}J}_{\mu'}}\coprod_{x\in J_{\mu'}^\mu}X_{\mu,wx}.
$$
This yields
$$
\begin{array}{rcl}
^vZ_{\mu}&\simeq & H^*(\overline X_{\mu,vw_\mu})\\
&\simeq&\bigoplus_{w\in {^vJ}_\mu}H^*(X_{\mu,w})\langle 2\ell(vw_\mu)-2\ell(w)\rangle\\
&\simeq& \bigoplus_{w\in {^vJ}_{\mu'}}\bigoplus_{x\in J_{\mu'}^\mu}H^*(X_{\mu',wx})\langle 2\ell(vw_\mu)-2\ell(w)-2\ell(x)\rangle.
\end{array}
$$
We also have
$\overline X_{\mu,v}=\coprod_{w\in {^vJ}_\mu} X_{\mu,w}$. 
This implies
$$
\begin{array}{rcl}
^vZ_{\mu'}&\simeq& H^*(\overline X_{\mu',vw_{\mu'}})\\
&\simeq& \bigoplus_{w\in {^vJ}_{\mu'}}H^*(X_{\mu',w})\langle 2\ell(vw_{\mu'})-2\ell(w)\rangle
\end{array}
$$

Note that we have $\ell(w_{\mu'})-\ell(w_\mu)=\mu_{k+1}$. Moreover, for each $w\in {^vJ}_{\mu'}$ and $x\in J_{\mu'}^\mu$ the variety $X_{\mu,wx}$ is an affine fibration over $X_{\mu',w}$. This implies
$$
^vZ_{\mu}\simeq \bigoplus_{x\in J_{\mu'}^\mu}{^vZ}_{\mu'}\langle 2\ell(vw_\mu)-2\ell(vw_{\mu'})-2\ell(x) \rangle=\bigoplus_{r=0}^{\mu_{k+1}}{^vZ}_{\mu'}\langle 2r \rangle.
$$
\end{proof}

\smallskip
We will write $\Ind$ and $\Res$ for the induction and restriction functors
$$
\Ind\colon \mod(^vZ_{\mu'})\to\mod(^vZ_{\mu}),\qquad \Res\colon \mod(^vZ_{\mu})\to\mod(^vZ_{\mu'}).
$$
We fix the graded lifts of $\widetilde\Res$ and $\widetilde\Ind$ of the functors $\Res$ and $\Ind$ in the following way
$$
\widetilde\Res(M)=M\langle -\mu_{k+1} \rangle,\qquad\widetilde\Ind(M)={^vZ}_{\mu}\otimes_{{^vZ}_{\mu'}} M.
$$

Now, Lemma \ref{ch3:lem_Zmu-Zmu'-mod} implies the following.

\smallskip
\begin{coro}
\label{ch3:coro_Res-Ind-adj}
$(a)$ The pair of functors $(\Res,\Ind)$ is biadjoint.

$(b)$ The pairs of functors $(\widetilde\Ind,\widetilde\Res\langle{\mu_{k+1}}\rangle)$ and $(\widetilde\Res,\widetilde\Ind\langle{-\mu_{k+1}}\rangle)$ are adjoint.

$(c)$
$$
\widetilde\Res\circ\widetilde\Ind =\Id^{\oplus [\mu_{k+1}+1]_q}:= \bigoplus_{r=0}^{\mu_{k+1}}\Id\langle{2r-\mu_{k+1}}\rangle,
$$
where $\Id$ is the identity endofunctor of the category $\grmod(Z_{\mu'})$.
\qed
\end{coro}

\subsection{Graded lifts of the functors}

As above we assume $W_\mu\subset W_{\mu'}$ and that $v\in J_{\mu',+}$.

\smallskip
\begin{lem}
\label{ch3:lem_diag-Fk-Res}
The following diagram of functors is commutative
$$
\begin{CD}
{^vO}_{\mu,R}^{\Delta} @>{F_{k}}>> {^vO}_{\mu', R}^{\Delta}\\
@V{\bbV_{\mu, R}}VV @V{\bbV_{\mu',R}}VV\\
\mod(^vZ_{\mu,R}) @>{\Res}>> \mod(^vZ_{\mu',R}).
\end{CD}
$$
\end{lem}
\begin{proof}
Let $M$ be an object in $^vO_{\mu,R}^{\Delta}$.
We have the following chain of isomorphisms of $^vZ_{\mu',R}$-modules.
$$
\begin{array}{lll}
\bbV_{\mu',R}\circ F_{k}(M)&\simeq& \Hom(^vP^{\mu'}_{R},F_{k}(M))\\
&\simeq&\Hom(E_{k}(^vP^{\mu'}_{R}),M)\\
&\simeq&\Hom(^vP^{\mu}_{R},M)\\
&\simeq&\bbV_{\mu,R}(M)\\
\end{array}
$$
Here, the $^vZ_{\mu,R}$-modules in the last two lines are considered as $^vZ_{\mu',R}$-modules with respect to the inclusion $^vZ_{\mu',R}\subset {^vZ}_{\mu,R}$.
The third isomorphism in the chain is an isomorphism of $^vZ_{\mu',R}$-modules by Lemma \ref{ch3:lem_diag-cent-proj-mu'-mu}.
\end{proof}

\smallskip
\begin{lem}
\label{ch3:lem_diag-Ek-Ind}
The following diagram of functors is commutative
$$
\begin{CD}
^vO^{\Delta}_{\mu,R} @<{E_{k}}<< ^vO^\Delta_{\mu',R}\\
@V{\bbV_{\mu,R}}VV @V{\bbV_{\mu',R}}VV\\
\mod(^vZ_{\mu,R}) @<{\Ind}<< \mod(^vZ_{\mu',R}).
\end{CD}
$$
\end{lem}
\begin{proof}
Let $M$ be an object in $^vO^\Delta_{\mu',R}$. We have the following chain of isomorphisms of $^vZ_{\mu,R}$-modules.
$$
\begin{array}{lll}
\bbV_{\mu,R}\circ E_k(M)&\simeq&\Hom(^vP^\mu_R,E_k(M))\\
&\simeq&\Hom(F_k(^vP^\mu_R),M)\\
&\simeq&\Hom(\bbV_{\mu',R}\circ F_k(^vP^\mu_R),\bbV_{\mu',R}(M))\\
&\simeq&\Hom(\Res\circ\bbV_{\mu,R}(^vP^\mu_R),\bbV_{\mu',R}(M))\\
&\simeq&\Hom(^vZ_{\mu,R},\bbV_{\mu',R}(M))\\
&\simeq&\Ind\circ\bbV_{\mu',R}(M)\\
\end{array}
$$
Here the third isomorphism holds by Proposition \ref{ch3:prop_Vk} $(a)$, the fourth isomorphism holds by Lemma \ref{ch3:lem_diag-Fk-Res}.
The last isomorphism holds because, by Corollary \ref{ch3:coro_Res-Ind-adj} $(a)$, the functor $\Hom_{^vZ_{\mu',R}}(^vZ_{\mu,R},\bullet)$, which is obviously right adjoint to $\Res$, is isomorphic to $\Ind$.
\end{proof}

\smallskip
Now, Lemmas \ref{ch3:lem_diag-Fk-Res}-\ref{ch3:lem_diag-Ek-Ind} imply the following.

\smallskip
\begin{coro}
\label{ch3:coro_diag-F-Res}
The following diagrams of functors are commutative
$$
\begin{CD}
^vO_{\mu} @>{F_{k}}>> ^vO_{\mu'}\\
@V{\bbV_{\mu}}VV @V{\bbV_{\mu'}}VV\\
\mod(^vZ_{\mu}) @>{\Res}>> \mod(^vZ_{\mu'}),
\end{CD}
$$
$$
\begin{CD}
^vO_{\mu} @<{E_{k}}<< ^vO_{\mu'}\\
@V{\bbV_{\mu}}VV @V{\bbV_{\mu'}}VV\\
\mod(^vZ_{\mu}) @<{\Ind}<< \mod(^vZ_{\mu'}).
\end{CD}
$$
\end{coro}
\begin{proof}
Passage to the residue field in Lemma \ref{ch3:lem_diag-Ek-Ind} implies that the diagrams in the statement are commutative on $\Delta$-filtered objects. A standard argument (see for example the proof of Lemma \ref{ch3:lem_gr-lift-E-F+adj-case1}) shows that the commutativity on $\Delta$-filtered objects implies the commutativity. 
\end{proof}
\smallskip
Let $^vO_{\mu}^{\rm proj}$ and $^v\widetilde O_{\mu}^{\rm proj}$ be the full subcategories of projective modules in $^vO_{\mu}$ and $^v\widetilde O_{\mu}$ respectively. The fully faithfulness of the functor $\bbV_{\mu}$ on projective modules implies the fully faithfulness of the functor $\widetilde\bbV_{\mu}$ on projective modules. These functors identify $^vO_{\mu,}^{\rm proj}$ and $^v\widetilde O_{\mu}^{\rm proj}$ with some full subcategories in $\mod(^vZ_{\mu})$ and $\grmod(^vZ_{\mu})$ that we denote $\mod(^vZ_{\mu})^{\rm proj}$ and $\grmod(^vZ_{\mu})^{\rm proj}$ respectively. Since the functor $F_k$ takes projective modules to projective modules, the commutativity of the diagram in Corollary \ref{ch3:coro_diag-F-Res} implies that the functor $\Res$ takes the category $\mod(^vZ_{\mu})^{\rm proj}$ to $\mod(^vZ_{\mu'})^{\rm proj}$. This implies that its graded lift $\widetilde\Res$ takes $\grmod(^vZ_{\mu})^{\rm proj}$ to $\grmod(^vZ_{\mu'})^{\rm proj}$. Similar statements hold for $\Ind$ and $\widetilde\Ind$.



\smallskip
\begin{lem}
\label{ch3:lem_gr-lift-E-F+adj-case1}
$(a)$ The functors $E_k$ and $F_k$ admit graded lifts $\widetilde E_k\colon{^v\widetilde O}_{\mu'} \to {^v\widetilde O}_{\mu}$ and $\widetilde F_k\colon{^v\widetilde O}_{\mu} \to {^v\widetilde O}_{\mu'}$. They can be chosen in such a way that the condition below holds.

$(b)$
The following pairs of functors are adjoint
$$
(\widetilde F_{k},\widetilde E_{k}\langle -\mu_{k+1} \rangle),\quad (\widetilde E_{k},\widetilde F_{k}\langle \mu_{k+1} \rangle).
$$
\end{lem}
\begin{proof}
Let us prove $(a)$. We give the prove only for the functor $F_k$. The proof for $E_k$ is similar. The proof below is similar to the proof of \cite[Lem.~5.10]{SVV}.

As explained above, the functor $\widetilde\Res$ restricts to a functor $\grmod(^vZ_{\mu})^{\rm proj}\to\grmod(^vZ_{\mu'})^{\rm proj}$. Together with the equivalences of categories $^v\widetilde O_{\mu}^{\rm proj}\simeq\grmod(^vZ_{\mu})^{\rm proj}$ and $^v\widetilde O_{\mu'}^{\rm proj}\simeq\grmod(^vZ_{\mu'})^{\rm proj}$ obtained by restricting $\widetilde\bbV_{\mu}$ and $\widetilde\bbV_{\mu'}$ this yields a functor $\widetilde F_k\colon {^v\widetilde O}_{\mu}^{\rm proj}\to {^v\widetilde O}_{\mu'}^{\rm proj}$. Next, we obtain a functor of homotopy categories $\widetilde F_k\colon K^b(^v\widetilde O_{\mu}^{\rm proj})\to K^b(^v\widetilde O_{\mu'}^{\rm proj})$. Since the categories $^v\widetilde O_{\mu}$ and $^v\widetilde O_{\mu'}$ have finite global dimensions, we have equivalences of categories $K^b(^v\widetilde O_{\mu}^{\rm proj})\simeq D^b(^v\widetilde O_{\mu})$ and $K^b(^v\widetilde O_{\mu'}^{\rm proj})\simeq D^b(^v\widetilde O_{\mu'})$.
Thus we get a functor of triangulated categories $\widetilde F_k\colon D^b(^v\widetilde O_{\mu})\to D^b(^v\widetilde O_{\mu'})$. If we repeat the same construction for non-graded categories, we obtain a functor $F_k\colon D^b(^vO_{\mu})\to D^b(^vO_{\mu'})$ that is the same as the functor between the bounded derived categories induced by the exact functor $F_k\colon { ^vO}_{\mu}\to {^vO}_{\mu'}$, see Corollary \ref{ch3:coro_diag-F-Res}. This implies that the following diagram is commutative
$$
\begin{CD}
D^b(^v\widetilde O_{\mu})@>{\widetilde F_k}>>D^b(^v\widetilde O_{\mu'})\\
@V{\rm forget}VV                   @V{\rm forget}VV\\
D^b(^vO_{\mu})@>{F_k}>> D^b(^vO_{\mu'})\\
\end{CD}
$$
Since the bottom functor takes $^vO_{\mu}$ to $^vO_{\mu'}$, the top functor takes $^v\widetilde O_{\mu}$ to $^v\widetilde O_{\mu'}$. This completes the proof of $(a)$.


Now we prove $(b)$.
The functors $\widetilde E_k$ and $\widetilde F_k$ are constructed as unique functors such that we have the following commutative diagrams
\begin{equation}
\label{ch3:eq_diag-E-Res-grad}
\begin{CD}
^vO_{\mu} @>{\widetilde F_{k}}>> ^vO_{\mu'}\\
@V{\widetilde\bbV_{\mu}}VV @V{\widetilde\bbV_{\mu'}}VV\\
\mod(^vZ_{\mu'}) @>{\widetilde\Res}>> \mod(^vZ_{\mu'}),
\end{CD}
\qquad
\begin{CD}
^vO_{\mu} @<{\widetilde E_{k}}<< ^vO_{\mu'}\\
@V{\widetilde\bbV_{\mu}}VV @V{\widetilde\bbV_{\mu'}}VV\\
\mod(^vZ_{\mu}) @<{\widetilde\Ind}<< \mod(^vZ_{\mu'}).
\end{CD}
\end{equation}

By Corollary \ref{ch3:coro_Res-Ind-adj} $(b)$ and Proposition \ref{ch3:prop_Vk} $(b)$, the restrictions of the pairs $(\widetilde F_{k},\widetilde E_{k}\langle -\mu_{k+1} \rangle)$ and $(\widetilde E_{k},\widetilde F_{k}\langle \mu_{k+1} \rangle)$ to the subcategories of projective objects are biadjoint. We can conclude using the lemma below.
\end{proof}

\smallskip
\begin{lem}
\label{ch3:lem_lift-adj-proj}
Let $\calC_1$, $\calC_2$ be abelian categories of finite global dimension and let $\calC_1'$, $\calC_2'$ be the full subcategories of projective objects. Assume that $E\colon \calC_1\to \calC_2$, $F\colon \calC_2\to\calC_1$ are exact functors. Assume that $E$ and $F$ send projective objects to projective objects and denote $E'\colon \calC_1'\to \calC_2'$, $F'\colon \calC_2'\to \calC_1'$ the restrictions of $E$ and $F$. Assume that the pair $(E',F')$ is adjoint. Then the pair $(E,F)$ is adjoint.
\end{lem}
\begin{proof}[Proof]
Let
$$
\varepsilon'\colon E'F'\to{\Id},\qquad
\eta'\colon\Id\to F'E'
$$
be the counit and the unit of the adjoint pair $(E',F')$.

We can extend the functors $E'$ and $F'$ to functors $E'\colon K^b(\calC_1')\to K^b(\calC_2')$ and $F'\colon K^b(\calC_2')\to K^b(\calC_1')$ of the homotopy categories of bounded complexes. The counit $\varepsilon'$ and the unit $\eta'$ extend to natural transformations of functors of homotopy categories. These extended natural transformations still satisfy the properties of the counit and the unit of an adjunction. Thus the extended pair $(E',F')$ is adjoint.

Since the categories $\calC_1$ and $\calC_2$ have finite global dimensions, we have equivalences of categories
\begin{equation}
\label{ch3:eq_K(C')=D(C)}
K^b(\calC_1')\simeq D^b(\calC_1),\qquad K^b(\calC_2')\simeq D^b(\calC_2).
\end{equation}
By construction, the functors
\begin{equation}
\label{ch3:eq_E-F-D(C)}
E\colon D^b(\calC_1)\to D^b(\calC_2),\quad F\colon D^b(\calC_2)\to D^b(\calC_1)
\end{equation}
obtained from functors $E'$ and $F'$ via the equivalences (\ref{ch3:eq_K(C')=D(C)}) coincide with the functors induced from $E\colon \calC_1\to\calC_2$ and $F\colon \calC_2\to\calC_1$. The pair of functors $(E,F)$ in (\ref{ch3:eq_E-F-D(C)}) is adjoint with a counit $\varepsilon$ and a unit $\eta$, obtained from $\varepsilon'$ and $\eta'$. These counit and unit restrict to natural transformations of functors of abelian categories $E\colon \calC_1\to\calC_2$ and $F\colon \calC_2\to\calC_1$. This proves the statement.
\end{proof}

\smallskip
We need the following lemma later.

\smallskip
\begin{lem}
\label{ch3:lem_prod-FE-case1}
We have the following isomorphism of functors
$$
\widetilde F_k\widetilde E_k\simeq \Id^{\oplus [\mu_{k+1}+1]_q}:=\bigoplus_{r=0}^{\mu_{k+1}}\Id\langle 2r-\mu_{k+1} \rangle,
$$
where $\Id$ is the identity endofunctor of the category $^v\widetilde O_{\mu'}$.
\end{lem}
\begin{proof}[Proof]
By Corollary \ref{ch3:coro_Res-Ind-adj} $(c)$ we have $\widetilde \Res\circ\widetilde\Ind\simeq \Id^{\oplus [\mu_{k+1}+1]_q}$. Then the diagrams (\ref{ch3:eq_diag-E-Res-grad}) and Proposition \ref{ch3:prop_Vk} $(b)$ yield an isomorphism $\widetilde F_k\widetilde E_k\simeq \Id^{\oplus [\mu_{k+1}+1]_q}$ on projective modules in $^v\widetilde O_{\mu'}$. This isomorphism can be extended to the category $^v\widetilde O_{\mu'}$ in the same way as in the proof of Lemma \ref{ch3:lem_lift-adj-proj}.
\end{proof}

\subsection{The case $W_{\mu'}\subset W_{\mu}$}
\label{ch3:subs_second-case}
In the sections above we assumed $W_{\mu}\subset W_{\mu'}$ (or equivalently $\mu_k=1$). In this section we announce similar results in the case $W_{\mu'}\subset W_{\mu}$ (or equivalently $\mu_{k+1}=0$). All the proofs are the same as in the previous case but the roles of $E_k$ and $F_k$ should be exchanged.

Here we always assume that $v$ is in $J_{\mu,+}$ (thus also in $J_{\mu',+}$).
In contrast with the situation above, we have $^vZ_{\mu'}\subset {^vZ}_{\mu}$. Consider the induction and the restriction functors $\Ind\colon \mod(^vZ_{\mu'})\to \mod(^vZ_{\mu})$ and $\Res\colon\mod(^vZ_{\mu})\to \mod(^vZ_{\mu'})$.

Similarly to Corollary \ref{ch3:coro_diag-F-Res} we can prove the following statement.

\smallskip
\begin{lem}
The following diagrams of functors are commutative
$$
\begin{CD}
^vO_{\mu} @>{F_{k}}>> ^vO_{\mu'}\\
@V{\bbV_{\mu}}VV @V{\bbV_{\mu'}}VV\\
\mod(^vZ_{\mu}) @>{\Ind}>> \mod(^vZ_{\mu'}),
\end{CD}
$$
$$
\begin{CD}
^vO_{\mu} @<{E_{k}}<< ^vO_{\mu'}\\
@V{\bbV_{\mu}}VV @V{\bbV_{\mu'}}VV\\
\mod(^vZ_{\mu}) @<{\Res}<< \mod(^vZ_{\mu'}).
\end{CD}
$$
\qed
\end{lem}

\smallskip
Next, similarly to Lemmas \ref{ch3:lem_gr-lift-E-F+adj-case1}, \ref{ch3:lem_prod-FE-case1} we can deduce the following result.

\smallskip
\begin{lem}
$(a)$ The functors $E_k$ and $F_k$ admit graded lifts $\widetilde E_k\colon{^v\widetilde O}_{\mu'} \to {^v\widetilde O}_{\mu}$. They can be chosen in such a way that the conditions below hold.

$(b)$ The following pairs of functors are adjoint
$$
(\widetilde F_{k},\widetilde E_{k}\langle \mu_{k}-1 \rangle),\quad (\widetilde E_{k},\widetilde F_{k}\langle -\mu_{k}+1 \rangle).
$$

$(c)$ We have the following isomorphism of functors
$$
\widetilde E_k\widetilde F_k\simeq \Id^{\oplus [\mu_{k}]_q}:=\bigoplus_{r=0}^{\mu_{k}-1}\Id\langle 2r-\mu_{k}+1 \rangle,
$$
where $\Id$ is the identity endofunctor of the category $^v\widetilde O_{\mu}$.
\end{lem}

\section{Koszul duality}
\label{ch3:sec-Koszul}

\subsection{Bimodules over a semisimple basic algebra}
\label{ch3:subs_bimod}
Let $B$ be a $\bbC$-algebra isomorphic to a finite direct sum of copies of $\bbC$. We have $B=\bigoplus_{\lambda\in\Lambda}\bbC e_\lambda$ for some idempotents $e_\lambda$.

\smallskip
\begin{df}
Let $\bmod(B)$ be the category of finite dimensional $(B,B)$-bimodules.
\end{df}

\smallskip
A bimodule $M\in \bmod(B)$ can be viewed just as a collection of finite dimensional $\bbC$-vector spaces $e_\lambda M e_\mu$ for $\lambda,\mu\in\Lambda$.
To each bimodule $M\in\bmod(B)$ we can associate a bimodule $M^\star\in\bmod(M)$ as follows $M^\star=\Hom_{\bmod(B)}(M,B\otimes_{\bbC}B)$. The bimodule structure on $M^\star$ is defined in the following way. For $f\in M^\star$, $m\in M$, $b_1,b_2\in B$ we have $b_1fb_2(m)=f(b_2mb_1)$.

\smallskip
\begin{lem}
\label{ch3:lem_bimod}
Assume that $M,N\in \bmod(B)$, $X,Y\in\mod(B)$, $Z\in\mod(B)^{\rm op}$. Then we have the following isomorphisms:

{\rm (a)} $\Hom_{\bmod(B)}(M,N)\simeq\bigoplus_{\lambda,\mu\in\Lambda}\Hom_\bbC(e_\lambda M e_\mu,e_\lambda M e_\mu)$,

{\rm (b)} $\Hom_B(X,Y)\simeq \bigoplus_{\lambda\in\Lambda}\Hom_\bbC(e_\lambda M,e_\lambda M)$,

{\rm (c)} $X\otimes_B Z=\bigoplus_{\lambda\in\Lambda}Xe_\lambda\otimes_\bbC e_\lambda Z$,

{\rm (d)} $e_\lambda M^\star e_\mu\simeq (e_\mu Me_\lambda)^*$, where $\bullet^*$ is the usual duality for $\bbC$-vector spaces,

{\rm (e)} $\Hom_B(M^\star\otimes_BX,Y)\simeq\Hom_B(X,M\otimes_BY)$,

{\rm (f)} $(M\otimes_B N)^\star\simeq N^\star\otimes_B M^\star$.
\end{lem}
\begin{proof}[Proof]
Parts (a), (b), (c) are obvious. Part (d) follows from (a). Part (e) follows from (b), (c), (d). Part (f) follows from (c), (d).
\end{proof}


\subsection{Quadratic dualities}
\label{ch3:subs_quad-dual}

Let $A=\oplus_{n\in \bbN}A_n$ be a finite dimensional $\bbN$-graded algebra over $\bbC$. Assume that $A_0$ is semisimple and basic. Let $T_{A_0}(A_1)=\bigoplus_{n\in N}A_1^{\otimes n}$ be the tensor algebra of $A_1$ over $A_0$, here $A_1^{\otimes n}$ means $A_1\otimes_{A_0}A_1\otimes_{A_0}\cdots \otimes_{A_0}A_1$ with $n$ components $A_1$. The algebra $A$ is said to be \emph{quadratic} if it is generated by elements of degree $0$ and $1$ with relations in degree $2$, i.e., the kernel of the obvious map $T_{A_0}(A_1)\to A$ is generated by elements in $A_1\otimes_{A_0}A_1$.

\smallskip
\begin{df}
Consider the ($A_0$,$A_0$)-bimodule morphism $\phi\colon A_1\otimes_{A_0}A_1\to A_2$ given by the product in $A$. Let $\phi^\star \colon A_2^\star\to A_1^\star\otimes_{A_0} A_1^\star$ be the dual morphism to $\phi$, see Lemma \ref{ch3:lem_bimod}, here $\bullet^\star$ is as in Section \ref{ch3:subs_bimod} with respect to $B=A_0$.
The \emph{quadratic dual algebra} to $A$ is the quadratic algebra $A^!=T_{A_0}(A_1^\star)/(\Im~\phi^\star)$.
\end{df}

\smallskip
\begin{rk}
In the previous definition we do not assume that the algebra $A$ is quadratic itself. However, if it is true, we have a graded $\bbC$-algebra isomorphism $(A^!)^!\simeq A$.
\end{rk}

\smallskip

Let $\calC$ be an abelian category such that its objects are graded modules.
Denote by $\Com^\downarrow(\calC)$ the category of complexes $X^\bullet$ in $\calC$ such that the $j$th graded component of $X^i$ is zero when $i>>0$ or $i+j<<0$. Similarly, let $\Com^\uparrow(\calC)$ the category of complexes $X^\bullet$ in $\calC$ such that the $j$th graded component of $X^i$ is zero when $i<<0$ or $i+j>>0$. Denote by $D^\downarrow(\calC)$ and  $D^\uparrow(\calC)$ the corresponding derived categories of such complexes.
We will use the following abbreviations
$$
D^\downarrow(A)=D^\downarrow(\grmod(A)),\quad D^\uparrow(A)=D^\uparrow(\grmod(A)),\quad
D^b(A)=D^b(\grmod(A)).
$$

In the situation above we have the following functors
$\calK\colon D^\downarrow(A)\to D^\uparrow(A^!)$ and $\calK'\colon D^\uparrow(A^!)\to D^\downarrow(A)$ called \emph{quadratic duality functors}. See \cite[Sec.~5]{MOS} for more details.

\subsection{Koszul algebras}
\label{ch3:subs_Koszul-alg}
Let $A=\bigoplus_{n\in\bbN}A_n$ be a finite dimensional $\bbN$-graded $\bbC$-algebra such that $A_0$ is semisimple. We identify $A_0$ with the left graded $A$-module $A_0\simeq A/{\oplus_{n>0}A_n}$.

\smallskip
\begin{df}
The graded algebra $A$ is \emph{Koszul} if the left graded $A$-module $A_0$ admits a projective resolution $\cdots\to P^2\to P^1\to P^0\to A_0$ such that $P^r$ is generated by its degree $r$ component.
\end{df}

\smallskip
If $A$ is Koszul, we consider the graded $\bbC$-algebra $A^!=\Ext^*_A(A_0,A_0)^{\rm op}$ and we call it the \emph{Koszul dual} algebra to $A$.
The following is well-known, see \cite{BGS}.

\smallskip
\begin{prop}
\label{ch3:prop_Koszul-duality}
Let $A$ be a Koszul $\bbC$-algebra. Assume that $A$ and $A^!$ are finite dimensional. Then, the following holds.

$(a)$ The algebra $A$ is quadratic. The Koszul dual algebra $A^!$ coincides with the quadratic dual algebra.

$(b)$ The algebra $A^!$ is also Koszul and there is a graded algebra isomorphism $(A^!)^!\simeq A$.

$(c)$ There is an equivalence of categories
$$
\calK\colon {D}^b(A)\to {D}^b(A^!), \qquad M\mapsto \RHom_A(A_0,M).
$$
 \qed
\end{prop}

\smallskip
If $A$ is Koszul, then the functors $\calK$ and $\calK'$ from the previous section are mutually inverse. Moreover, the equivalence $\calK$ of bounded derived categories in Proposition \ref{ch3:prop_Koszul-duality} $(c)$ is the restriction of the functor $\calK$ from the previous section.

\smallskip
\begin{df}
Let $A$ and $B$ be Koszul algebras. We say that the functor $\Phi\colon D^b(A)\to D^b(B)$ is Koszul dual to the functor $\Psi\colon D^b(A^!)\to D^b(B^!)$ if the following diagram of functor is commutative
$$
\begin{CD}
D^b(A) @>{\Psi}>> D^b(B)\\
@V{\calK}VV                @V{\calK}VV\\
D^b(A^!)@>{\Phi}>> D^b(B^!).
\end{CD}
$$
\end{df}

\subsection{Categories of linear complexes}
In this section we recall some results from \cite{MOS} about linear complexes. Let $A$ be as in Section \ref{ch3:subs_quad-dual}.

\smallskip
\begin{df}
Let $\mathcal{LC}(A)$ be the category of complexes $\cdots\to\calX^{k-1}\to\calX^{k}\to\calX^{k+1}\to\cdots$ of projective modules in $\grmod(A)$ such that for each $k\in\bbZ$ each indecomposable direct factor $P$ of $\calX^k$ is a direct factor of $A\langle k\rangle$.
\end{df}

\smallskip
\begin{prop}
There is an equivalence of categories $\epsilon_A\colon \mathcal{LC}(A)\simeq \grmod(A^!)$.\qed
\end{prop}

\smallskip
Let us describe the construction of $\epsilon^{-1}_A$. Let $M=\oplus_{n\in\bbZ}M_n$ be in $\grmod(A^!)$. The graded $A^!$-module structure yields morphisms of $A_0$-modules $f'_n\colon A^!_1\otimes M_n\to M_{n+1}$ for each $n\in\bbZ$. We have
\begin{eqnarray*}
\Hom_{A_0}(A^!_1\otimes_{A_0} M_n,M_{n+1})&=&\Hom_{A_0}(M_n,(A^!_1)^\star\otimes_{A_0} M_{n+1})\\
&=&\Hom_{A_0}(M_n,A_1\otimes_{A_0} M_{n+1}).
\end{eqnarray*}

Let $f_n\colon\Hom_{A_0}(M_n,A_1\otimes_{A_0} M_{n+1})$ be the image of $f'_n$ by the chain of isomorphisms above.

We have $\epsilon_A^{-1}(M)=\cdots\stackrel{\partial_{k-2}}{\to}\calX^{k-1}\stackrel{\partial_{k-1}}{\to}\calX^{k}\stackrel{\partial_k}{\to}\calX^{k+1}\stackrel{\partial_{k+1}}{\to}\cdots$ with $\calX^k=A\langle k\rangle\otimes_{A_0} M_k$ and
$$
\partial_k\colon A\langle k\rangle\otimes_{A_0} M_k\to A\langle k+1 \rangle\otimes_{A_0} M_{k+1},\quad a\otimes m\mapsto (a\otimes \Id)(f_k(m)).
$$

The quadratic duality functor discussed in the previous section can be characterized as follows, see \cite[Prop.~21]{MOS}.

\smallskip
\begin{lem}
\label{ch3:lem_Kos-Tot}
Up to isomorphism of functors, the following diagram is commutative:
$$
\begin{diagram}
\node[2]{D^\uparrow(\mathcal{LC}(A))} \arrow{sw,l}{\mathrm{Tot}} \\
\node{D^\downarrow(A)}
\node[2]{D^\uparrow(A^!)} \arrow[2]{w,b}{\calK'} \arrow{nw,t}{\epsilon^{-1}_A}\\
\end{diagram},
$$
where $\mathrm{Tot}$ is the functor taking the total complex.
\qed
\end{lem}

\subsection{The main lemma about Koszul dual functors}
\label{ch3:subs_key-lem}
Let $\{e_\lambda;\lambda\in\Lambda\}$ be the set of indecomposable idempotents of $A_0$, i.e., we have $A_0=\bigoplus_{\lambda\in\Lambda}\bbC e_\lambda$. Denote by $e^!_\lambda$ the corresponding idempotent of $A^!_0$ via the identification $A_0\simeq A^!_0$. For each subset $\Lambda'\subset\Lambda$ set $e_{\Lambda '}=\sum_{\lambda\in\Lambda '}e_\lambda$. Consider the graded algebras
$$
A_{\Lambda'}=e_{\Lambda'}Ae_{\Lambda'}, \qquad _{\Lambda'}A=A/(e_{\Lambda\backslash\Lambda'}).
$$
Similarly, we can define $A^!_{\Lambda'}$ and $_{\Lambda'}A^!$.

We have a functor $F\colon\grmod(A_{\Lambda'})\to\grmod(A)$, $M\mapsto Ae_{\Lambda'}\otimes_{A_{\Lambda'}} M$. Note also that the category $\grmod(_{\Lambda'}A^!)$ can be viewed as a subcategory of $\grmod(A^!)$ containing modules that are killed by $e_{\Lambda\backslash\Lambda'}$. Let $\iota\colon\grmod(_{\Lambda'}A^!)\to \grmod(A^!)$ be the inclusion. The following proposition is proved in \cite[Thm.~28]{MOS}.

\smallskip
\begin{prop}
\label{ch3:prop_dual-F-G}
{\rm (a)} The quadratic dual algebra to $A_{\Lambda'}$ is isomorphic to $_{\Lambda'}A^!$.

{\rm (b)} The following diagram commutes up to isomorphism of functors.
$$
\begin{CD}
D^\downarrow(A) @<\calK'<< D^\uparrow(A^!)\\
@AFAA                      @A\iota AA\\
D^\downarrow(A_{\Lambda'}) @<\calK'<< D^\uparrow(_{\Lambda'}A^!)\\
\end{CD}
$$
\end{prop}
\begin{proof}[Idea of proof of (b)]
By Lemma \ref{ch3:lem_Kos-Tot} it is enough to proof the commutativity of the following diagram.
$$
\begin{CD}
\mathcal{LC}(A) @<\epsilon_A^{-1}<< \grmod(A^!)\\
@AFAA                      @A\iota AA\\
\mathcal{LC}(A_{\Lambda'}) @<\epsilon_{A_{\Lambda'}}^{-1}<< \grmod(_{\Lambda'}A^!)\\
\end{CD}
$$
\end{proof}

\smallskip
We can generalize this result as follows.

\smallskip
\begin{lem}
\label{ch3:lem_key}
Let $A'$ be a finite dimensional $\bbN$-graded $\bbC$-algebra. Assume that for some subset $\Lambda'\subset\Lambda$ there is a graded (unitary) homomorphism $\psi\colon A'\to A_{\Lambda'}$ such that
\begin{itemize}
    \item[\rm{(a)}] $\psi$ is an isomorphism in degrees $0$ and $1$,
    \item[\rm{(b)}] $\psi$ induces an isomorphism between the kernel of $A'_1\otimes_{A'_0}A'_1\to A'_2$ and the kernel of $(A_{\Lambda'})_1\otimes_{(A_{\Lambda'})_0}(A_{\Lambda'})_1\to (A_{\Lambda'})_2$.
\end{itemize}
Then the quadratic dual of $A'$ is isomorphic to $_{\Lambda'}A$.

Consider the graded $(A,A')$-bimodule $Ae_{\Lambda'}$, where the right $A'$-module structure is obtained from the right $A_{\Lambda'}$-module structure using $\psi$. Consider the functor $T\colon \grmod(A')\to\grmod(A)$, $M\mapsto Ae_{\Lambda'}\otimes_{A'}M$. Then the following diagram commutes up to an isomorphism of functors.
$$
\begin{CD}
D^\downarrow(A) @<\calK'<< D^\uparrow(A^!)\\
@ATAA                      @A\iota AA\\
D^\downarrow(A') @<\calK'<< D^\uparrow(_{\Lambda'}A^!)\\
\end{CD}
$$
\end{lem}
\begin{proof}[Proof]
By definition, the quadratic dual of $A'$ depends only on the algebra $A'_0$, the $(A'_0,A'_0)$-bimodule $A'_1$ and the kernel of $A'_1\otimes_{A'_0}A'_1\to A'_2$. Thus the quadratic dual algebras of $A'$ and $A_{\Lambda'}$ are isomorphic. Finally, Proposition \ref{ch3:prop_dual-F-G} (a) implies that the quadratic dual of $A'$ is isomorphic to $_{\Lambda'}A$.

Now, by Lemma \ref{ch3:lem_Kos-Tot} is enough to prove the commutativity of the following diagram up to an isomorphism of functors.
$$
\begin{CD}
\mathcal{LC}(A) @<\epsilon^{-1}_A<< \grmod(A^!)\\
@ATAA                      @A\iota AA\\
\mathcal{LC}(A') @<\epsilon^{-1}_{A'}<< \grmod(_{\Lambda'}A^!)\\
\end{CD}
$$

By analogy with the definition of the functor $T$, consider the functor $\Phi\colon \grmod(A')\to\grmod(A_{\Lambda'})$, $M\mapsto A_{\Lambda'}\otimes_{A'}M$. For each $\lambda\in\Lambda'$ let $e'_\lambda$ be the idempotent in $A'_0$ such that $\psi(e'_\lambda)=e_\lambda$. We have $\Phi(A'e'_\lambda)=A_{\Lambda'}e_\lambda$ for each $\lambda\in\Lambda'$. In particular $\Phi$ induces a bijection between the indecomposable direct factors of $A'$ and $A_{\Lambda'}$. Thus $\Phi$ induces a functor $\Phi\colon \mathcal{LC}(A')\to \mathcal{LC}(A_{\Lambda'})$. Note that by definition the boundary maps in the complexes of the category $\mathcal{LC}(\bullet)$ are of degree $1$. Thus, by (a) and (b) the functor $\Phi$ induces an equivalence of categories $\Phi\colon \mathcal{LC}(A')\to \mathcal{LC}(A_{\Lambda'})$.


Consider the following diagram, where the functor $F$ is as before Proposition \ref{ch3:prop_dual-F-G}.
$$
\begin{CD}
\mathcal{LC}(A) @<\Id<< \mathcal{LC}(A) @<\epsilon^{-1}_A<< \grmod(A^!)\\
@ATAA           @AFAA                              @A\iota AA\\
\mathcal{LC}(A') @<\Phi^{-1}<< \mathcal{LC}(A_{\Lambda'}) @<\epsilon^{-1}_{A_{\Lambda'}}<< \grmod(_{\Lambda'}A^!)\\
\end{CD}
$$
The right square commutes by the proof of Proposition \ref{ch3:prop_dual-F-G} and the commutativity of the left square is obvious. To conclude we need only to check that $\epsilon^{-1}_{A'}=\Phi^{-1}\circ\epsilon^{-1}_{A_{\Lambda'}}$.

Let us check that $\Phi\circ\epsilon^{-1}_{A'}=\epsilon^{-1}_{A_{\Lambda'}}$.
This is clear on objects because
$$
\begin{array}{ccccccccccc}
\epsilon^{-1}_{A_{\Lambda'}}(M)&= \cdots &\stackrel{\partial'_{k-1}}{\to}& A_{\Lambda'}\langle k \rangle \otimes_{(A_{\Lambda'})_0}M_k &\stackrel{\partial'_k}{\to}& A_{\Lambda'}\langle k+1 \rangle \otimes_{(A_{\Lambda'})_0}M_{k+1}&\stackrel{\partial'_{k+1}}{\to}&\cdots,\\
\epsilon^{-1}_{A'}(M)&= \cdots &\stackrel{\partial''_{k-1}}{\to}& A'\langle k \rangle \otimes_{A'_0}M_k &\stackrel{\partial''_k}{\to}& A'\langle k+1 \rangle \otimes_{A'_0}M_{k+1}&\stackrel{\partial''_{k+1}}{\to}&\cdots.
\end{array}
$$

The boundary maps are defined as follows
$$
\begin{array}{rll}
\partial'_k\colon& A_{\Lambda'}\langle k\rangle\otimes_{(A_{\Lambda'})_0} M_k\to A_{\Lambda'}\langle k+1 \rangle\otimes_{(A_{\Lambda'})_0} M_{k+1},&\quad a\otimes m\mapsto (a\otimes \Id)(f^1_n(m)),\\
\partial''_k\colon& A'\langle k\rangle\otimes_{A'_0} M_k\to A'\langle k+1 \rangle\otimes_{A'_0} M_{k+1},&\quad a\otimes m\mapsto (a\otimes \Id)(f^2_n(m)),
\end{array}
$$
where $f^1_n\colon M_n\to (A_{\Lambda'})_1\otimes_{(A_{\Lambda'})_0} M_{n+1}$ and $f^2_n\colon M_n\to A'_1\otimes_{A'_0} M_{n+1}$ are defined in the same way as $f_n$ in the definition of $\epsilon^{-1}$. Thus it is also clear that $\Phi$ commutes with the boundary maps.
\end{proof}

\smallskip
\begin{rk}
\label{ch3:rk_cond-b}
Condition (b) is necessary only to deduce that $(A')^!\simeq {(A_{\Lambda'})^!}$. Without this condition we know only that the algebra $(A')^!$ is isomorphic to a quotient of ${(A_{\Lambda'})^!}$. Thus condition (b) can be replaced by the requirement $\dim (A')^!=\dim {_{\Lambda'}A^!}$.
\end{rk}

\smallskip
We can reformulate Lemma \ref{ch3:lem_key} in the following way.

\smallskip
\begin{coro}
\label{ch3:coro_main-Koszul}
Let $A'$ be an $\bbN$-graded finite dimensional $\bbC$-algebra with basic $A'_0$ such that the indecomposable idempotents of $A'_0$ are parameterized by a subset $\Lambda'$ of $\Lambda$, i.e., we have $A'_0=\bigoplus_{\lambda\in\Lambda'}\bbC e'_\lambda$. Assume that $\dim (A')^!=\dim {_{\Lambda'}A^!}$. Assume also that there is an exact functor $T\colon \grmod(A')\to\grmod(A)$ such that
\begin{itemize}
    \item[\rm{(a)}] $T(A'e'_\lambda)= Ae_\lambda$ $\forall \lambda\in\Lambda'$,
    \item[\rm{(b)}] the functor $T$ yields an isomorphism $\Hom_{A'}(A'e'_\lambda\langle 1\rangle,A'e'_\mu)\simeq \Hom_{A}(Ae_\lambda\langle 1\rangle,Ae_\mu)$.
\end{itemize}
Then the quadratic dual for $A'$ is $_{\Lambda'}A^!$ and the following diagram commutes up to isomorphism of functors.
$$
\begin{CD}
D^\downarrow(A) @<\calK'<< D^\uparrow(A^!)\\
@ATAA                      @A\iota AA\\
D^\downarrow(A') @<\calK'<< D^\uparrow(_{\Lambda'}A^!)\\
\end{CD}
$$
\end{coro}
\begin{proof}[Proof]

Condition (a) implies that the functor $T$ yields a homomorphism of graded algebras $\psi\colon A'\to A_{\Lambda'}$. Moreover, condition (b) implies that $\psi$ satisfies condition (a) of Lemma \ref{ch3:lem_key}. Finally, the assumption $\dim (A')^!=\dim {_{\Lambda'}A}$ implies that $\psi$ satisfies condition (b) of Lemma \ref{ch3:lem_key}, see Remark \ref{ch3:rk_cond-b}. The functor $T$ hare can be identified with the functor $T=Ae_{\Lambda'}\otimes_{A'}\bullet$ in the statement of Lemma \ref{ch3:lem_key}, see \cite[Lem.~3.4]{Str}. Thus the statement follows from Lemma \ref{ch3:lem_key}.
\end{proof}

\subsection{Zuckerman functors}
\label{ch3:subs_Zuck}
Fix $v\in \widehat W$. Let $\nu_1$ and $\nu_2$ be two different parabolic types such that $W_{\nu_1}\subset W_{\nu_2}$. By definition of the parabolic category $\calO$, there is an inclusion of categories ${^v}O^{\nu_2}_\mu\subset {^v}O^{\nu_1}_\mu$. We denote by $\inc$ the inclusion functor. We may write $\inc=\inc_{\nu_2}^{\nu_1}$ to specify the parameters. The functor $\inc$ admits a left adjoint functor $\tr$. For $M\in {^v}O^{\nu_1}_\mu$, the object $\tr(M)$ is the maximal quotient of $M$ that is in ${^v}O^{\nu_2}_\mu$, see Lemma \ref{ch3:lem_trunc-categ-gen} $(a)$. We call the functor $\tr$ the \emph{parabolic truncation} functor. We may write $\tr_{\nu_1}^{\nu_2}$ to specify the parameters.

Now, we assume that $\nu_1$ and $\nu_2$ are two arbitrary parabolic types. Then there is a parabolic type $\nu_3$ such that we have $W_{\nu_3}=W_{\nu_1}\cap W_{\nu_2}$. The \emph{Zuckerman functor} $\Zuc_{\nu_1}^{\nu_2}$ (or simply $\Zuc$) is the composition $\Zuc_{\nu_1}^{\nu_2}=\tr_{\nu_3}^{\nu_2}\circ\inc_{\nu_1}^{\nu_3}$.

The parabolic inclusion functor is exact. The parabolic truncation functor is only right exact. This implies that the Zuckerman functor is right exact.

Now, we are going to grade Zuckerman functors. Let $^vA^\nu_\mu$ be the endomorphism algebra of the minimal projective generator of $^vO_\mu^\nu$ (or simply $^vA_\mu$ in the non-parabolic case). We have $^vO_\mu^\nu\simeq \mod(^vA_\mu^\nu)$. The Koszul grading on $^vA^\nu_\mu$ is constructed in \cite{SVV}. The graded version $^v\widetilde O_\mu^\nu$ of $^vO_\mu^\nu$ is the category $\grmod(^vA^\nu_\mu)$. Moreover, the algebra $^vA_\mu^\nu$ is the quotient of $^vA_\mu$ by a homogeneous ideal $I_\nu$. By construction, the grading on $^vA_\mu^\nu$ is induced from the grading on  $^vA_\mu$. Assume that $\nu_1$ and $\nu_2$ are such that $W_{\nu_1}\subset W_{\nu_2}$. Then we have $I_{\nu_1}\subset I_{\nu_2}$. This implies that the graded algebra $^vA_\mu^{\nu_2}$ is isomorphic to the quotient of the graded algebra $^vA^{\nu_1}_\mu$ by the homogeneous ideal $I_{\nu_2}/I_{\nu_1}$. This yields an inclusion of graded categories ${^v}\widetilde O^{\nu_2}_\mu\subset {^v}\widetilde O^{\nu_1}_\mu$. Let us denote by $\widetilde\inc_{\nu_2}^{\nu_1}$ (or simply $\widetilde\inc$) the inclusion functor. It is a graded lift of the functor $\inc$. Similarly, its left adjoint functor $\widetilde\tr_{\nu_1}^{\nu_2}$ is a graded lift of the functor $\tr$, see Remark \ref{ch3:rk_grade-trunc-gen}. Thus we get graded lifts $\widetilde\Zuc_{\nu_1}^{\nu_2}$ of the Zuckerman functor $\Zuc_{\nu_1}^{\nu_2}$ for arbitrary parabolic types $\nu_1$ and $\nu_2$.

Similarly, we can define the parabolic inclusion functor, the parabolic truncation functor, the Zuckerman functor and their graded versions for the affine category $\calO$ at a positive level.

\subsection{The Koszul dual functors in the category $O$}
\label{ch3:subs_dual-funct-in-O}

As above, we assume $W_{\mu}\subset W_{\mu'}$. Set $J^\nu_\mu=\{w\in J_\mu;~w(1_\mu)\in P^\nu\}$. Note that the inclusion $J_{\mu'}\subset J_\mu$ induces an inclusion $J^\nu_{\mu'}\subset J^\nu_\mu$. For $v\in\widehat W$ we set $^vJ^\nu_\mu=\{w\in J^\nu_\mu;~w\leqslant v\}$.

As in Section \ref{ch3:sec_gr-lifts} we assume that we have $W_\mu\subset W_{\mu'}$.  Fix a parabolic type $\nu=(\nu_1,\cdots,\nu_l)\in X_l[N]$. 

Assume $v\in J_{\mu'}^\nu w_{\mu'}$. The functors
$$
F_k\colon {^vO}_{\mu}\to {^vO}_{\mu'},\qquad E_k\colon {^vO}_{\mu'}\to {^vO}_{\mu}
$$
restrict to functors of parabolic categories
$$
F_k\colon {^vO}^\nu_{\mu}\to {^vO}^\nu_{\mu'},\qquad E_k\colon {^vO}^\nu_{\mu'}\to {^vO}^\nu_{\mu}.
$$
The restricted functors still satisfy the properties announced in Lemmas \ref{ch3:lem_gr-lift-E-F+adj-case1}, \ref{ch3:lem_prod-FE-case1}.



Assume that $w\in {^vJ}^\nu_\mu$. Let $^vP^{w(1_\mu)}$ be the projective cover of $L^{w(1_\mu)}$ in $^vO^\nu_\mu$. (Note that we do not indicate the parabolic type $\nu$ in our notations for modules to simplify the notations.) We fix the grading on $L^{w(1_\mu)}$ such that it is concentrated in degree zero when we consider $L^{w(1_\mu)}$ as an $^vA^\nu_\mu$-module (see Section \ref{ch3:subs_Zuck} for the definition of $^vA^\nu_\mu$). A standard argument shows that the modules $^vP^{w(1_\mu)}$ and $\Delta^{w(1_\mu)}$ admit graded lifts. (The graded lift of $^vP^{w(1_\mu)}$ can be constructed as the projective cover of the graded lift of $L^{w(1_\mu)}$ in $^v\widetilde O^\nu_\mu$. The existence of graded lifts of projective modules implies the existence of graded lifts of Verma modules, see \cite[Cor.~4]{MO}.) We fix the graded lifts of $^vP^{w(1_\mu)}$ and $\Delta^{w(1_\mu)}$ such that the surjections $^vP^{w(1_\mu)}\to L^{w(1_\mu)}$ and $\Delta^{w(1_\mu)}\to L^{w(1_\mu)}$ are homogeneous of degree zero, see also Lemma \ref{ch3:lem-grad-unique}.

The following lemma is stated in the parabolic category $O$.

\smallskip
\begin{lem}
\label{ch3:lem_Ek-on-proj-spcase1}
$(a)$ For each $w\in {^vJ}_{\mu'}^\nu$, we have $E_k({^vP}^{w(1_{\mu'})})={^vP}^{w(1_{\mu})}$.

$(b)$ For each $w\in {^vJ}_{\mu}^\nu$, we have
$$
F_k(L^{w(1_{\mu})})=
\left\{
\begin{array}{ll}
L^{w(1_{\mu'})} &\mbox{\rm if }  w\in {^vJ}^\nu_{\mu'},\\
0 &\mbox{\rm else}.
\end{array}
\right.
$$
\end{lem}
\begin{proof}
First, we prove $(a)$ in the non-parabolic situation (i.e., for $\nu=(1,1,\cdots,1)$). The modules $E_k({^vP}^{w(1_{\mu'})})$ and ${^vP}^{w(1_{\mu})}$ are both projective. Thus it is enough to show that their classes in the Grothendieck group are the same. To show this, we compare the multiplicities of Verma modules in the $\Delta$-filtrations of $E_k({^vP}^{w(1_{\mu'})})$ and ${^vP}^{w(1_{\mu})}$.

We need to show that for each $x\in {^vJ}_{\mu'}$ we have
$$
[E_k({^vP}^{w(1_{\mu'})}),\Delta^{x(1_{\mu'})}]=[{^vP}^{w(1_{\mu})},\Delta^{x(1_{\mu})}].
$$
By Lemma \ref{ch3:lem_Ek-Verma-spcase1}, for each $x\in {^vJ}_\mu$, the multiplicity $[E_k({^vP}^{w(1_{\mu'})}),\Delta^{x(1_{\mu})}]$ is equal to the multiplicity $[{^vP}^{w(1_{\mu'})},\Delta^{x(1_{\mu'})}]$. So, we need to prove the equality
$$
[{^vP}^{w(1_{\mu'})},\Delta^{x(1_{\mu'})}]=[{^vP}^{w(1_{\mu})},\Delta^{x(1_{\mu})}].
$$
The last equality is obvious because both of these multiplicities are given by the same parabolic Kazhdan-Lusztig polynomial. See, for example, \cite[App.~A]{Mak-Koszul} for more details about multiplicities in the parabolic category $\calO$ for $\widehat{\mathfrak{gl}}_N$.

Now, we prove $(b)$. Since the set of simple modules in the parabolic category $O$ is a subset of the set of simple modules of the non-parabolic category $O$, it is enough to prove $(b)$ in the non-parabolic case.

For each $w\in {^vJ}_{\mu}$ and $x\in{^vJ}_{\mu'}$, we have
$$
\begin{array}{rcl}
\Hom(^vP^{x(1_{\mu'})},F_k(L^{w(1_\mu)}))&\simeq&  \Hom(E_k(^vP^{x(1_{\mu'})}),L^{w(1_\mu)})\\
&\simeq&
\Hom(^vP^{x(1_{\mu})},L^{w(1_\mu)}).
\end{array}
$$
This implies that we have $\dim\Hom(^vP^{x(1_{\mu'})},F_k(L^{w(1_\mu)}))=\delta_{x,w}$.
Since $\dim\Hom(^vP^{x(1_{\mu'})},M)$ counts the multiplicity of the simple module $L^{x(1_{\mu'})}$ in the module $M$ (this fact can be proved in the same way as \cite[Thm.~ 3.9~(c)]{BGG}), this proves $(b)$.

Finally, we prove $(a)$ in the parabolic situation. For each $w\in {^vJ}_{\mu'}^\nu$ and each $x\in {^vJ}_{\mu}^\nu$ we have
$$
\begin{array}{rcl}
\Hom(E_k({^vP}^{w(1_{\mu'})}),L^{x(1_{\mu})})&\simeq& \Hom({^vP}^{w(1_{\mu'})},F_k(L^{x(1_{\mu})}))\\
&\simeq&
\left\{
\begin{array}{lll}
\Hom({^vP}^{w(1_{\mu'})},L^{x(1_{\mu'})})& \mbox{if } x\in {^vJ}_{\mu'}^\nu \\
0& \mbox{else},
\end{array}
\right.
\end{array}
$$
where the second isomorphism follows from $(b)$. This implies that we have $\dim\Hom(E_k({^vP}^{w(1_{\mu'})}),L^{x(1_{\mu})})=\delta_{w,x}$. Thus we have $E_k({^vP}^{w(1_{\mu'})})\simeq {^vP}^{w(1_{\mu})}$.
\end{proof}

The definitions of the graded lifts $\widetilde E_k$ and $\widetilde F_k$ in Lemma \ref{ch3:lem_gr-lift-E-F+adj-case1} depend on the choice of the graded lift $\widetilde\bbV_\mu$ of $\bbV_\mu$. Note that we have the following isomorphism of ${^vZ}_\mu$-modules $\bbV_\mu(^vP^{\mu})\simeq {^vZ}_\mu$ for all $\mu\in X_e[N]$. By Lemma \ref{ch3:lem-grad-unique}, for each choice of the graded lift $\widetilde\bbV_\mu$, we have $\widetilde\bbV_\mu(^vP^{\mu})\simeq {^vZ}_\mu\langle r\rangle$ for some $r\in\bbZ$. From now on, we always assume that the graded lift $\widetilde\bbV_\mu$ is chosen in such a way that we have an isomorphism of graded ${^vZ}_\mu$-modules $\widetilde\bbV_\mu(^vP^{\mu})\simeq {^vZ}_\mu$ (without any shift $r$).

In the following statement we consider the non-parabolic situation.

\smallskip
\begin{lem}
\label{ch3:lem_Ek-on-Verma-spcase1-grad}
For each $w\in {^vJ_{\mu'}}$, the graded module $\widetilde E_k(\Delta^{w(1_{\mu'})})$ has a graded $\Delta$-filtration with constituents $\Delta^{wz(1_\mu)}\langle\ell(z)\rangle$ for $z\in J^\mu_{\mu'}$.

\end{lem}
\begin{proof}
First, we prove that $\widetilde E_k$ takes the graded anti-dominant projective module to the graded anti-dominant projective module, i.e., that we have $\widetilde E_k(^vP^{\mu'})\simeq {^vP}^{\mu}$.

By Lemma \ref{ch3:lem-grad-unique}, the graded lift of ${^vP}^{\mu}$ is unique up to graded shift. Thus, by Lemma \ref{ch3:lem_Ek-on-proj-spcase1}, we have $\widetilde E_k({^vP}^{\mu'})={^vP}^{\mu}\langle r \rangle$ for some $r\in\bbZ$. We need to prove that $r=0$.

Recall that the graded lift $\widetilde E_k$ of $E_k$ is constructed in the proof of Lemma \ref{ch3:lem_gr-lift-E-F+adj-case1} in such a way that the following diagram is commutative
$$
\begin{CD}
^vO_{\mu} @<{\widetilde E_{k}}<< ^vO_{\mu'}\\
@V{\widetilde\bbV_{\mu}}VV @V{\widetilde\bbV_{\mu'}}VV\\
\mod(^vZ_{\mu}) @<{\widetilde\Ind}<< \mod(^vZ_{\mu'}).
\end{CD}
$$
Moreover, by definition, we have the following isomorphisms of graded modules
$$
\widetilde\bbV_{\mu}({^vP}^{w(1_{\mu})})\simeq{^vZ}_{\mu}, \quad \widetilde\bbV_{\mu'}({^vP}^{w(1_{\mu'})})\simeq{^vZ}_{\mu'},\quad
\widetilde\Ind(^vZ_{\mu'})={^vZ}_{\mu}.
$$
This implies that we have $r=0$.

Now we prove the statement of the lemma. The module $\widetilde E_k(\Delta^{w(1_{\mu'})})$ has a graded $\Delta$-filtartion because it has a $\Delta$-filtration as an ungraded module, see \cite[Rem.~2.13]{Mak-Koszul}. The constituents (up to graded shifts) are $\Delta^{wz(1_\mu)}$, $z\in W_{\mu'}/W_{\mu}$ by Lemma \ref{ch3:lem_Ek-Verma-spcase1}. We need only to identify the shifts. The graded multiplicities of Verma modules in projective modules are given in terms of Kazhdan-Lusztig polynomials in \cite[App.~A]{Mak-Koszul}. In particular, \cite[Lem.~A.4~(d)]{Mak-Koszul} implies that, for each $w\in {^vJ}_\mu$, the module $\Delta^{w(1_\mu)}$ appears as a constituent in a graded $\Delta$-filtration of $^vP^\mu$ once with the graded shift by $\ell(w)$. Similarly, for each $w\in{^vJ}_{\mu'}$, the module $\Delta^{w(1_{\mu'})}$ appears as a constituent in a graded $\Delta$-filtration of $^vP^{\mu'}$ once with the graded shift by $\ell(w)$. Now, since $\widetilde E_k({^vP}^{\mu'})\simeq {^vP}^{\mu}$, we see that, for each $w\in {^vJ}_{\mu'}$ and each $z\in J_{\mu'}^{\mu}$, the module $\Delta^{wz(1_\mu)}$ appears in the $\Delta$-filtration of $\widetilde E_k(\Delta^{w(1_{\mu'})})$ with the graded shift by $\ell(z)$.
\end{proof}

\smallskip
In the following lemma me consider the general (i.e., parabolic) situation.
\smallskip
\begin{lem}
\label{ch3:lem_Ek-on-proj-grad-spcase1}
For each $w\in {^vJ}_{\mu'}^\nu$, we have $\widetilde E_k({^vP}^{w(1_{\mu'})})={^vP}^{w(1_{\mu})}$.
\end{lem}
\begin{proof}
By Lemmas \ref{ch3:lem-grad-unique} and \ref{ch3:lem_Ek-on-proj-spcase1}, we have $\widetilde E_k({^vP}^{w(1_{\mu'})})={^vP}^{w(1_{\mu})}[r]$ for some integer $r$. We must show that the shift $r$ is zero.

First, we prove this in the non-parabolic case. 
The module $\Delta^{w(1_\mu')}$ (resp. $\Delta^{w(1_\mu)}$) is contained in each $\Delta$-filtration of ${^vP}^{w(1_{\mu'})}$ (resp. ${^vP}^{w(1_{\mu})}$) only once and without a graded shift. Moreover, by Lemma \ref{ch3:lem_Ek-on-Verma-spcase1-grad} the module $\Delta^{w(1_\mu)}$ is contained in each $\Delta$-filtration of $\widetilde E_k(\Delta^{w(1_\mu')})$ only once and without a graded shift. This implies that the graded shift $r$ is zero.

The parabolic case follows from the non-parabolic case. Really, the projective covers of simple modules in the parabolic category $O$ are quotients of protective covers in the non-parabolic category $O$ (see Lemma \ref{ch3:lem_trunc-categ-gen} $(b)$). Thus the shift $r$ should be zero in the parabolic case because it is zero in the non-parabolic case.
\end{proof}

Let us check that the functor $\widetilde E_k\colon {^v\widetilde O}^\nu_{\mu}\to {^v\widetilde O}^\nu_{\mu'}$
satisfies the hypotheses of Corollary \ref{ch3:coro_main-Koszul}. Condition $(a)$ follows from Lemma \ref{ch3:lem_Ek-on-proj-grad-spcase1}.

Let $P$ and $Q$ be projective covers of simple modules in $^v\widetilde O_{\mu}$ graded as above. To check $(b)$, we have to show that we have an isomorphism
$$
\Hom(\widetilde E_{k}(P)\langle 1\rangle,\widetilde E_{k}(Q))\simeq \Hom(P\langle 1 \rangle,Q).
$$
We have
$$
\begin{array}{rcl}
\Hom(\widetilde E_{k}(P)\langle 1\rangle,\widetilde E_{k}(Q))&\simeq &\Hom(P,\widetilde F_{k+1}\widetilde E_{k+1}(Q)\langle \mu_{k+1}-1 \rangle)\\
&\simeq&\Hom(P,[\mu_{k+1}+1]_q(Q)\langle \mu_{k+1}-1 \rangle)\\
&\simeq&\Hom(P,Q\langle -1 \rangle)\bigoplus \oplus_{r=1}^{\mu_{k+1}}\Hom(P,Q\langle 2r-1 \rangle)\\
&\simeq&\Hom(P\langle 1 \rangle,Q).
\end{array}
$$
Here the first isomorphism follows from Lemma \ref{ch3:lem_gr-lift-E-F+adj-case1} $(b)$, the second isomorphism follows from Lemma \ref{ch3:lem_prod-FE-case1}. The last isomorphism holds because $\Hom(P,Q\langle r \rangle)$ is zero for $r>0$ because the $\bbZ$-graded algebra
$$
\End(\bigoplus_{w\in {^vJ^\nu_\mu}}{^vP}^{w(1_{\mu'})})
$$
has zero negative homogeneous components (as it is Koszul).

For each $\mu=(\mu_1,\cdots,\mu_e)$ we set $\mu^{\rm op}=(\mu_e,\cdots,\mu_1)$.
We can define the positive level version $O_{\mu,+}^\nu$ of the category $O_{\mu}^\nu$ in the following way. For each $\lambda\in P$ we set $\widetilde\lambda^+=\lambda+z_\lambda\delta+(e-N)\Lambda_0$,
where $z_\lambda=(\lambda,2\rho+\lambda)/2{e}$. For each $\lambda\in P^\nu$ denote by $^+\Delta(\lambda)$ the Verma module with highest weight $\widetilde\lambda^+$ and denote by $^+L(\lambda)$ its simple quotient. We will also abbreviate $^+\Delta^\lambda={^+\Delta}(\lambda-\rho)$ and $^+L^\lambda={^+L}(\lambda-\rho)$.
 Let $O_{\mu,+}^\nu$ be the Serre subcategory of $O^\nu$ generated by the simple modules $^+L^\lambda$ for $\lambda\in P^\nu[\mu^{\rm op}]$. Similarly to the negative $e$-action of $\widehat W$ on $P$ described in Section \ref{ch3:subs_ext-aff} we can consider the positive $e$-action on $P$. We define the positive $e$-action in the following way: the element $w\in\widehat W$ sends $\lambda$ to $-w(-\lambda)$ (where $w(-\lambda)$ corresponds to the negative $e$-action). The notion of the positive $e$-action of $\widehat W$ on $P$ is motivated by the fact that the map
$$
P\to \widehat\bfh^*,\quad \lambda\mapsto\widetilde{\lambda-\rho}^++\widehat\rho
$$
is $\widehat W$-invariant. 
We say that an element $\lambda\in P$ is $e$-\emph{dominant} if we have $\lambda_1\geqslant\lambda_2\geqslant\cdots\geqslant\lambda_N\geqslant\lambda_1-e$. Fix an $e$-dominant element $1^+_\mu\in P[\mu^{\rm op}]$. (We can take for example $1^+_\mu=(e^{\mu_1},\cdots,1^{\mu_e})$). Note that the stabilizer of $1^+_\mu$ in $\widehat W$ with respect to the positive $e$-action is $W_\mu$. From now on, each time when we write $w(1^+_\mu)$ we mean the positive $e$-action on $P$ and each time when we write $w(1_\mu)$ we mean the negative $e$-action.

Recall that $J_{\mu,+}$ is the subset of $\widehat W$ containing all $w$ such that $w$ is maximal in $wW_\mu$. Set $J^\nu_{\mu,+}=\{w\in J_{\mu,+};~w(1^+_\mu)\in P^\nu\}$. Note that the inclusion $J_{\mu'}\subset J_\mu$ induces an inclusion $J^\nu_{\mu'}\subset J^\nu_\mu$. For $v\in\widehat W$ we set $^vJ^\nu_\mu=\{w\in J^\nu_\mu;~w\leqslant v\}$ and $^vJ^\nu_{\mu,+}=\{w\in J^\nu_{\mu,+};~w\leqslant v\}$.

We have the following lemma. 

\smallskip
\begin{lem}
\label{ch3:lem_comb_index_set_O}
$(a)$ There is a bijection $J_\mu^\nu\to J_{\nu,+}^\mu$ given by $w\mapsto w^{-1}$.

$(b)$ For each $v\in J_\mu^\nu$, there is a bijection $^vJ_\mu^\nu\to {^{v^{-1}}}J_{\nu,+}^\mu$ given by $w\mapsto w^{-1}$.
\end{lem}
\begin{proof} 
Part $(a)$ follows from \cite[Cor.~3.3]{SVV}. Part $(b)$ follows from part $(a)$.
\end{proof}

\smallskip
Similarly to the truncated version $^vO_{\mu}^\nu$ of $O_\mu^\nu$, we can define the truncated version $^vO_{\mu,+}^\nu$ of $O_{\mu,+}^\nu$. We define $^vO_{\mu,+}^\nu$ as the Serre quotient of $O_{\mu,+}^\nu$, where we kill the simple module $^+L^{w(1^+_\mu)}$ for each $w\in J_{\mu,+}^\nu-{^vJ}_{\mu,+}^\nu$.

By \cite[Thm.~3.12]{SVV}, for $v\in J_\mu^\nu$, the category ${^v\widetilde O}_{\mu}^\nu$ is Koszul dual to the category ${^{v^{-1}}\widetilde O}_{\nu,+}^\mu$. The bijection between the simple modules in ${^v\widetilde O}_{\mu}^\nu$ and the indecomposable projective modules in ${^{v^{-1}}\widetilde O}_{\nu,+}^\mu$ given by the Koszul functor $\calK$ is such that for each $w\in {^vJ}_\mu^\nu$ the module $L^{w(1_\mu)}$ corresponds to the projective cover of $^+L^{w^{-1}(1^+_\nu)}$.  

We should make a remark about our notation. Usually, we denote by $e$ the number of components in $\mu$ and we denote by $l$ the number of components in $\nu$. So, when we exchange the roles of $\mu$ and $\nu$ and we consider the category $O_{\nu,+}^\mu$, we mean that this category is defined with respect to the level $l-N$ (and not $e-N$).

Now, assume again that $v$ is in $J_{\mu'}^\nu w_{\mu'}$. Then we have $vw_\mu\in J_\mu^\nu$ and $vw_{\mu'}\in J_{\mu'}^\nu$. In this case the Koszul dual categories to $^vO_\mu^\nu$ and $^vO_{\mu'}^\nu$ are  $^{w_\mu v^{-1}}O_{\nu,+}^\mu$ and $^{w_{\mu'}v^{-1}}O_{\nu,+}^{\mu'}$.

\smallskip
\begin{lem}
\label{ch3:lem-can_change_trunc_pos_level}
$(a)$ We have 
$$^{w_{\mu'}v^{-1}}J_{\nu,+}^{\mu'}={^{w_\mu v^{-1}}}J_{\nu,+}^\mu\cap J_{\nu,+}^{\mu'}.$$

$(b)$ We have 
$$^{w_{\mu'}v^{-1}}J_{\nu,+}^{\mu'}=^{w_{\mu}v^{-1}}J_{\nu,+}^{\mu'}.$$
\end{lem}
\begin{proof}


Let us prove $(a)$. By Lemma \ref{ch3:lem_comb_index_set_O} the statement is equivalent to $$^{vw_{\mu'}}J^{\nu}_{\mu'}={^{vw_\mu}}J^{\nu}_\mu\cap J^{\nu}_{\mu'}.$$ Moreover, by definition, we have $^{vw_{\mu'}}J^{\nu}_{\mu'}={^v}J^{\nu}_{\mu'}$ and ${^{vw_\mu}}J^{\nu}_\mu={^{v}}J^{\nu}_\mu$. Thus, the statement is equivalent to $^{v}J^{\nu}_{\mu'}={^{v}}J^{\nu}_\mu\cap J^{\nu}_{\mu'}$. The last equality is obvious.

Part $(b)$ follows from part $(a)$.

\end{proof}

\smallskip
Now, put $u=w_{\mu}v^{-1}$. The discussion above together with Lemma \ref{ch3:lem-can_change_trunc_pos_level} shows that the Koszul dual categories to to $^vO_\mu^\nu$ and $^vO_{\mu'}^\nu$ are  $^{u}O_{\nu,+}^\mu$ and $^{u}O_{\nu,+}^{\mu'}$.

We get the following result.

\smallskip
\begin{thm}
\label{ch3:thm_dual-func-O-spcase1}
Assume that we have $W_{\mu}\subset W_{\mu'}$.

$(a)$ The functor $\widetilde F_k\colon D^b(^v\widetilde O^\nu_{\mu})\to D^b(^v\widetilde O^\nu_{\mu'})$ is Koszul dual to the shifted parabolic truncation functor $\widetilde\tr\langle \mu_{k+1} \rangle\colon D^b(^u\widetilde O_{\nu,+}^\mu)\to D^b(^u\widetilde O_{\nu,+} ^{\mu'})$.

$(b)$ The functor $\widetilde E_k\colon D^b(^v\widetilde O^\nu_{\mu'})\to D^b(^v\widetilde O^\nu_{\mu})$ is Koszul dual to the parabolic inclusion functor $\widetilde\inc\colon D^b(^u\widetilde O_{\nu,+}^{\mu'})\to D^b(^u\widetilde O_{\nu,+} ^{\mu})$.
\end{thm}
\begin{proof}
We have checked above that the functor $\widetilde E_k\colon {^v\widetilde O}^\nu_{\mu}\to {^v\widetilde O}^\nu_{\mu'}$ satisfies the hypotheses of Corollary \ref{ch3:coro_main-Koszul}. Thus Corollary \ref{ch3:coro_main-Koszul} implies part $(b)$. Part $(a)$ follows from part $(b)$ by adjointness.
\end{proof}

\smallskip
Similarly to the situation $W_{\mu}\subset W_{\mu'}$, we can do the same in the situation $W_{\mu}\subset W_{\mu'}$ (see also Section \ref{ch3:subs_second-case}). In this case we should take $v\in J_{\mu}^\nu w_\mu$ and put $u=w_{\mu'}v^{-1}$. We get the following theorem.

\smallskip
\begin{thm}
\label{ch3:thm_dual-func-O-spcase2}
Assume that we have $W_{\mu'}\subset W_{\mu}$.

$(a)$ The functor $\widetilde F_k\colon D^b(^v\widetilde O^\nu_{\mu})\to D^b(^v\widetilde O^\nu_{\mu'})$ is Koszul dual to the parabolic inclusion functor $\widetilde\inc\colon D^b(^u\widetilde O_{\nu,+}^\mu)\to D^b(^u\widetilde O_{\nu,+} ^{\mu'})$.

$(b)$ The functor $\widetilde E_k\colon D^b(^v\widetilde O^\nu_{\mu'})\to D^b(^v\widetilde O^\nu_{\mu})$ is Koszul dual to the shifted parabolic truncation functor $\widetilde\tr\langle \mu_k-1\rangle\colon D^b(^u\widetilde O_{\nu,+}^{\mu'})\to D^b(^u\widetilde O_{\nu,+} ^{\mu})$.
\qed
\end{thm}

\subsection{The restriction to the category $\bfA$}
The goal of this section is to restrict the results of the previous section to category $\bfA$.

We have seen that we can grade the functor $E_k$ and $F_k$ for category $O$ when we have $W_\mu\subset W_{\mu'}$ or $W_{\mu'}\subset W_\mu$. Let us show that in this cases we can also grade similar functors for the category $\bfA$. We have $\bfA^\nu[\alpha]\subset {^vO^\nu_\mu}$ and $\bfA^\nu[\alpha+\alpha_k]\subset {^vO^\nu_{\mu'}}$. Denote by $h$ the inclusion functor from $\bfA^\nu[\alpha]$ to ${^vO^\nu_\mu}$. Abusing the notation, we will use the same symbol for the inclusion functor from $\bfA^\nu[\alpha+\alpha_k]$ to ${^vO^\nu_{\mu'}}$. Let $h^*$ and $h^!$ be the left and right adjoint functors to $h$. 
The functor $F_k$ for the category $\bfA$ is defined as the restriction of the functor $F_k$ for the category $O$. This restriction can be written as $h^!F_kh$. The functor $E_k$ for the category $O$ does not preserve the category $\bfA$ in general. The functor $E_k$ for the category $\bfA$ is defined in \cite[Sec.~5.9]{RSVV} as $h^*E_kh$. It is easy to see that we can grade the functor $h$ and its adjoint functors in the same way as we graded Zuckerman functors. Thus we obtain graded lifts $\widetilde E_k$ and $\widetilde F_k$ of the functors $E_k$ and $F_k$ for the category $\bfA$. Moreover, we still have the adjunctions $(\widetilde E_k,\widetilde F_k\langle\mu_{k+1}\rangle)$ (when $W_\mu\subset W_{\mu'}$) and $(\widetilde E_k,\widetilde F_k\langle 1-\mu_{k}\rangle)$ (when $W_{\mu'}\subset W_{\mu}$) in the category $\bfA$.

We do not have adjunctions in other direction in general. However, if additionally  we have $\nu_r>|\alpha|$ for each $r\in[1,l]$, then the functors $E_k$ and $F_k$ for the category $\bfA$ are biadjoint by \cite[Lem.~7.6]{RSVV}. This means that there is no difference between $h^*E_kh$ and $h^!E_kh$. Thus we also get the adjunctions $(\widetilde F_k,\widetilde E_k\langle-\mu_{k+1}\rangle)$ (when $W_\mu\subset W_{\mu'}$) and $(\widetilde F_k,\widetilde E_k\langle \mu_{k}-1\rangle)$ (when $W_{\mu'}\subset W_{\mu}$) in the category $\bfA$. (In fact, we always have the adjunctions in both directions if $k\ne 0$ because in this case the functor $E_k$ for the category $\bfA$ is just the restriction of the functor $E_k$ for the category $O$ and similarly for $\widetilde E_k$.)

We start from a general lemma. Let $A$ be a finite dimensional Koszul algebra over $\bbC$. Let $\{e_\lambda;\lambda\in\Lambda\}$ be the set of indecomposable idempotents in $A$. Fix a subset $\Lambda'\subset \Lambda$. Assume that the algebra $_{\Lambda'}A$ (see Section \ref{ch3:subs_key-lem} for the notations) is also Koszul. Then we have an algebra isomorphism $(_{\Lambda'}A)^!\simeq(A^!)_{\Lambda'}$. The graded algebra $_{\Lambda'}A$ is a quotient of the graded algebra $A$ by a homogeneous ideal. In particular we have an inclusion of categories $\iota\colon\grmod(_{\Lambda'}A)\to \grmod(A)$. Moreover, there is a functor
$$
\tau\colon \grmod(A^!)\to \grmod((A^!)_{\Lambda'}),\qquad  M\mapsto e^{!}_{\Lambda'}M.
$$
The functors $\iota$ and $\tau$ are both exact. They yield functors between derived categories $\iota\colon D^b(_{\Lambda'}A)\to D^b(A)$ and $\tau\colon D^b(A^!)\to D^b((A^!)_{\Lambda'})$.

Since the algebra $A$ is Koszul, there is a functor $\calK\colon D^b(A)\to D^b(A^!)$ defined by $\calK=\RHom(A_0,\bullet)$, see Section \ref{ch3:subs_Koszul-alg}. We will sometimes write $\calK_A$ to specify the algebra $A$.

In the following lemma we identify  $(_{\Lambda'}A)^!=(A^!)_{\Lambda'}$.

\smallskip
\begin{lem}
\label{ch3:lem_calK-subcat}
We have the following isomorphism of functors $D^b(_{\Lambda'}A)\to D^b((A^!)_{\Lambda'})$
$$
\calK_{_{\Lambda'}A}\simeq \tau\circ \calK_A\circ \iota.
$$
\end{lem}
\begin{proof}
For a complex $M\in D^b(_{\Lambda'}A)$, we have
$$
\begin{array}{rcl}
\tau\circ \calK_A\circ i(M)&\simeq&\tau(\RHom_A(A_0,M))\\
&\simeq&\RHom_A(e_{\Lambda'}A_0,M)\\
&\simeq&\RHom_{_{\Lambda'}A}({(_{\Lambda'}A})_0,M)\\
&\simeq&\calK_{_{\Lambda'}A}(M).
\end{array}
$$

\end{proof}

\smallskip
Fix $\alpha\in Q^+_e$. Consider the category $\bfA^\nu[\alpha]$ as in Section \ref{ch3:subs_cat-bfA}. Let $\mu$ be such that $\bfA^\nu[\alpha]$ is a subcategory of $O^\nu_\mu$. (Then $\bfA^\nu[\alpha+\alpha_k]$ is a subcategory of $O^\nu_{\mu'}$.) Assume that we have $W_{\mu}\subset W_{\mu'}$. Assume that $v\in J_{\mu'}^\nu w_{\mu'}$ is such that $\bfA^\nu[\alpha]$ is a subcategory of $^vO^\nu_\mu$ and $\bfA^\nu[\alpha+\alpha_k]$ is a subcategory of $^vO^\nu_{\mu'}$. Put $u=w_\mu v^{-1}$. The category $\bfA^\nu[\alpha]$ is also Koszul. Denote by $\widetilde\bfA^\nu[\alpha]$ its graded version. The Koszul dual category to $\bfA^\nu[\alpha]$ is a Serre quotient of the category $^uO_{\nu,+}^\mu$ (see \cite[Rem.~3.15]{Mak-Koszul}). Let us denote this quotient and its graded version by $\bfA_+^\mu[\alpha]$ and $\widetilde\bfA_+^\mu[\alpha]$ respectively. (We will also use similar notations for $\bfA^\nu[\alpha+\alpha_k]$.)

First, we prove the following lemma.

\smallskip
\begin{lem}
\label{ch3:lem-Zuck_defined_A+}
Assume that we have $W_{\mu}\subset W_{\mu'}$ and $k\ne 0$.

$(a)$
The inclusion of categories ${^uO}_{\nu,+}^{\mu'}\subset {^uO}_{\nu,+}^{\mu}$ yields an inclusion of categories $\bfA_+^{\mu'}[\alpha+\alpha_k]\subset \bfA_+^\mu[\alpha]$.

$(b)$ 
The inclusion of categories ${^u\widetilde O}_{\nu,+}^{\mu'}\subset {^u\widetilde O}_{\nu,+}^{\mu}$ yields an inclusion of categories $\widetilde\bfA_+^{\mu'}[\alpha+\alpha_k]\subset \widetilde\bfA_+^\mu[\alpha]$.

Assume that we have $W_{\mu}\supset W_{\mu'}$ and $k\ne 0$.

$(c)$
The inclusion of categories ${^uO}_{\nu,+}^{\mu}\subset {^uO}_{\nu,+}^{\mu'}$ yields an inclusion of categories $\bfA_+^{\mu}[\alpha]\subset \bfA_+^{\mu'}[\alpha+\alpha_k]$.

$(d)$ 
The inclusion of categories ${^u\widetilde O}_{\nu,+}^{\mu}\subset {^u\widetilde O}_{\nu,+}^{\mu'}$ yields an inclusion of categories $\widetilde\bfA_+^{\mu'}[\alpha]\subset \widetilde\bfA_+^{\mu'}[\alpha+\alpha_k]$.
\end{lem}
\begin{proof}
Denote by $p_1$ and $p_2$ respectively the quotient functors 
$$
p_1\colon{^uO}_{\nu,+}^{\mu'}\to\bfA_+^{\mu'}[\alpha+\alpha_k],\qquad p_2\colon{^uO}_{\nu,+}^{\mu}\to \bfA^\mu_+[\alpha].
$$
To prove $(a)$ and $(b)$, it is enough to prove that each simple module in ${^uO}_{\nu,+}^{\mu'}$ is killed by the functor $p_1$ if and only if it is killed by the functor $p_2$. We can get the combinatorial   description of the simple modules killed by $p_1$ and $p_2$ respectively using \cite[Rem.~2.18]{Mak-Koszul}.

For each $w\in {^vJ}_{\mu'}^{\nu}$ (resp. $w\in {^vJ}_{\mu}^{\nu}$), the simple module $^+L^{w^{-1}(1^+_{\nu})}$ is killed by $p_1$ (resp. $p_2$) if and only if the simple module $L^{w(1_{\mu'})}\in {^v}O_{\mu'}^\nu$ is not in $\bfA^\nu[\alpha+\alpha_k]$ (resp. the simple module $L^{w(1_\mu)}\in {^v}O_{\mu}^\nu$ is not in $\bfA^\nu[\alpha]$). So, we need to show that for each $w\in {^vJ}_{\mu'}^\nu$ the module $L^{w(1_{\mu'})}\in {^v}O_{\mu'}^\nu$ is in $\bfA^\nu[\alpha+\alpha_k]$ if and only if the module $L^{w(1_\mu)}\in {^v}O_{\mu}^\nu$ is in $\bfA^\nu[\alpha]$. Finally, we have to show that for each $w\in {^vJ}_{\mu'}^\nu$ we have $w(1_{\mu'})\geqslant \rho_\nu$ if and only if we have $w(1_{\mu})\geqslant \rho_\nu$. (Here the order is as in Section \ref{ch3:subs_rank-ch-A}.)

It is obvious that $w(1_{\mu})\geqslant \rho_\nu$ implies $w(1_{\mu'})\geqslant \rho_\nu$ because we have $w(1_{\mu'})\geqslant w(1_{\mu})$. Now, let us show the inverse statement. Note that we have $w(1_{\mu'})=w(1_\mu)+\epsilon_r$, where $r\in[1,N]$ is the unique index such that $w(1_\mu)_r\equiv k$ mod $e$. Assume that we have $w(1_{\mu'})\geqslant \rho_\nu$ but not $w(1_{\mu})\geqslant \rho_\nu$. Then we have $w(1_{\mu'})_r=(\rho_\nu)_r$. Assume first that $(\rho_\nu)_r\ne 1$. In particular this implies $r<N$. Since the weight $w(1_\mu)$ is in $P^\nu$, we have 
$$
w(1_{\mu'})_{r+1}=w(1_\mu)_{r+1}<w(1_\mu)_r=(\rho_\nu)_r-1=(\rho_\nu)_{r+1}.
$$ 
This contradicts to $w(1_{\mu'})\geqslant \rho_\nu$. Now, assume that we have $(\rho_\nu)_r=1$. Since we have $(\rho_\nu)_r\equiv w(1_{\mu'})_r\equiv k+1$ mod $e$, this implies $k=0$. This contradicts with the assumption $k\ne 0$. This proves the statement.

The proof of $(c)$, $(d)$ is similar to the proof of $(a)$ and $(b)$.
\end{proof}


\smallskip
In the case $W_{\mu}\subset W_{\mu'}$, $k\ne 0$, the lemma above allows us to define the parabolic inclusion functor $\inc\colon\bfA_+^{\mu'}[\alpha+\alpha_k]\to\bfA_+^{\mu}[\alpha]$ and the parabolic truncation functor $\tr\colon\bfA_+^{\mu}[\alpha]\to \bfA_+^{\mu'}[\alpha+\alpha_k]$ and their graded versions $\widetilde\inc$ and $\widetilde\tr$. Similarly, in the case $W_{\mu}\supset W_{\mu'}$, $k\ne 0$, the lemma above allows us to define the parabolic inclusion functor $\inc\colon\bfA_+^{\mu}[\alpha]\to\bfA_+^{\mu'}[\alpha+\alpha_k]$ and the parabolic truncation functor $\tr\colon\bfA_+^{\mu'}[\alpha+\alpha_k]\to \bfA_+^{\mu}[\alpha]$ and their graded versions $\widetilde\inc$ and $\widetilde\tr$.

\smallskip
\begin{thm}
\label{ch3:thm_dual-func-A-spcases}
Assume that we have $W_{\mu}\subset W_{\mu'}$.

$(a)$ The functor $\widetilde F_k\colon D^b(\widetilde\bfA^\nu[\alpha])\to D^b(\widetilde\bfA^\nu[\alpha+\alpha_k])$ is Koszul dual to the shifted parabolic truncation functor $\widetilde\tr\langle\mu_{k+1}\rangle \colon D^b(\widetilde\bfA^\mu_+[\alpha])\to D^b(\widetilde\bfA^{\mu'}_+[\alpha+\alpha_k])$.

$(b)$ The functor $\widetilde E_k\colon D^b(\widetilde\bfA^\nu[\alpha+\alpha_k])\to  D^b(\widetilde\bfA^\nu[\alpha])$ is Koszul dual to the parabolic inclusion functor $\widetilde\inc\colon D^b(\widetilde\bfA^{\mu'}_+[\alpha+\alpha_k])\to D^b(\widetilde\bfA^\mu_+[\alpha])$.

Now, assume that we have $W_{\mu'}\subset W_{\mu}$.

$(c)$ The functor $\widetilde F_k\colon D^b(\widetilde\bfA^\nu[\alpha])\to D^b(\widetilde\bfA^\nu[\alpha+\alpha_k])$ is Koszul dual to the parabolic inclusion functor $\widetilde\inc\colon D^b(\widetilde\bfA^\mu_+[\alpha])\to D^b(\widetilde\bfA^{\mu'}_+[\alpha+\alpha_k])$.

$(d)$ The functor $\widetilde E_k\colon D^b(\widetilde\bfA^\nu[\alpha+\alpha_k])\to  D^b(\widetilde\bfA^\nu[\alpha])$ is Koszul dual to the shifted parabolic truncation functor $\widetilde\tr\langle \mu_{k}-1\rangle\colon D^b(\widetilde\bfA^{\mu'}_+[\alpha+\alpha_k])\to D^b(\widetilde\bfA^\mu_+[\alpha])$.
\end{thm}

\begin{proof}
Let us prove $(b)$.

Let $v\in J_{\mu'}^\nu w_{\mu'}$ be such that $\bfA^\nu[\alpha]$ is a subcategory of $^vO^\nu_{\mu'}$ and $\bfA^\nu[\alpha+\alpha_k]$ is a subcategory of $^vO^\nu_{\mu'}$. Then the same is true for graded versions. Denote by $i$ the inclusion functor from $\widetilde\bfA^\nu[\alpha]$ to $^v\widetilde O^\nu_{\mu}$. Let $\tau\colon ^u\widetilde O^\mu_{\nu,+}\to \widetilde\bfA^\mu_+[\alpha]$ be the natural quotient functor.

Consider the following diagram
$$
\begin{CD}
D^b(\widetilde\bfA^\mu_+[\alpha]) @<{\widetilde\inc}<< D^b(\widetilde\bfA^{\mu'}_+[\alpha+\alpha_k])\\
@A{\tau}AA                             @A{\tau}AA\\
D^b(^v\widetilde O^\mu_{\nu,+}) @<{\widetilde\inc}<< D^b(^v\widetilde O^{\mu'}_{\nu,+})\\
@A{\calK}AA                             @A{\calK}AA\\
D^b(^v\widetilde O^\nu_{\mu}) @<{E_k}<< D^b(^v\widetilde O^\nu_{\mu'})\\
@A{i}AA                    @A{i}AA\\
D^b(\widetilde\bfA^\nu[\alpha]) @<{E_k}<< D^b(\widetilde\bfA^\nu[\alpha+\alpha_k]).
\end{CD}
$$
The commutativity of the top and bottom rectangles is obvious. The commutativity of the middle rectangle follows from Theorem \ref{ch3:thm_dual-func-O-spcase1} $(b)$. Now, by Lemma \ref{ch3:lem_calK-subcat}, the big rectangle in the diagram above yields the following commutative diagram
$$
\begin{CD}
D^b(\widetilde\bfA^\mu_+[\alpha]) @<{\rm\widetilde inc}<< D^b(\widetilde\bfA^{\mu'}_+[\alpha+\alpha_k])\\
@A{\calK}AA                    @A{\calK}AA\\
D^b(\widetilde\bfA^\nu[\alpha]) @<{F_k}<< D^b(\widetilde\bfA^\nu[\alpha+\alpha_k]).
\end{CD}
$$
This proves $(b)$.

Part $(a)$ follows from $(b)$ by adjointness. We can prove $(c)$ in the same way as $(b)$, using Theorem \ref{ch3:thm_dual-func-O-spcase2} $(a)$. Part $(d)$ follows from $(c)$ by adjointness.

\end{proof}

\subsection{Zuckerman functors for the category $\bfA_+$}
\label{ch3:subs_Zuck-A}

Fix $u\in \widehat W$. The Zuckerman functor $\Zuc\colon {^u}O_{\nu,+}^\mu\to {^u}O_{\nu,+}^{\mu'}$ (see Section \ref{ch3:subs_Zuck}) is a composition of a parabolic inclusion functor and a parabolic truncation functor ${^u}O_{\nu,+}^\mu\stackrel{\inc}{\to} {^u}O_{\nu,+}^{\mu''}\stackrel{\tr}{\to} {^u}O_{\nu,+}^{\mu'}$, where the parabolic type $\mu''$ is chosen such that $W_{\mu''}=W_{\mu}\cap W_{\mu'}$ (in fact, we can take $\mu''=\overline\mu^0$). Now we are going to give an analogue of the Zuckerman functor for the category $\bfA_+$, i.e., we want to define a functor $\Zuc_k^+\colon\bfA^{\mu}_+[\alpha]\to \bfA^{\mu'}_+[\alpha+\alpha_k]$. (Recall that the categories $\bfA^{\mu}_+[\alpha]$ and $\bfA^{\mu'}_+[\alpha+\alpha_k]$ are Serre quotients of ${^u}O_{\nu,+}^\mu$ and ${^u}O_{\nu,+}^{\mu'}$ respectively for $u$ big enough.) The main difficulty to give such a definition is that we have no obvious candidate to replace the category ${^u}O_{\nu,+}^{\mu''}$. 

Let us write $\overline\bfA$ instead of $\bfA$ to indicate that the category is defined with respect to $e+1$ instead of $e$.
Let us identify $\bfA^{\mu}_+[\alpha]\simeq \overline\bfA^{\overline\mu}_+[\beta+\overline\alpha]$ and $\bfA^{\mu'}_+[\alpha+\alpha_k]\simeq \overline\bfA^{\overline\mu'}_+[\beta+\overline\alpha+\overline\alpha_k+\overline\alpha_{k+1}]$ (see Proposition \ref{ch3:prop_equiv-A-e-e+1}).  Assume that we have $k\ne 0$. Then by Lemma \ref{ch3:lem-Zuck_defined_A+} we have the following inclusion of categories
$$
\overline\bfA^{\overline\mu}_+[\beta+\overline\alpha]\subset \overline\bfA^{\overline\mu^0}_+[\beta+\overline\alpha+\overline\alpha_k]\supset \overline\bfA^{\overline\mu'}_+[\beta+\overline\alpha+\overline\alpha_k+\overline\alpha_{k+1}].
$$

Now, we define the Zuckerman functor $\Zuc_k^+\colon\bfA^{\mu}_+[\alpha]\to \bfA^{\mu'}_+[\alpha+\alpha_k]$ as the composition of the parabolic inclusion functor with the parabolic truncation functor 
$$
\bfA^{\mu}_+[\alpha]\stackrel{\inc}{\to} \overline\bfA^{\overline\mu^0}_+[\beta+\overline\alpha+\overline\alpha_k]\stackrel{\tr}{\to}\bfA^{\mu'}_+[\alpha+\alpha_k].
$$ 
We define the Zuckerman functor $\Zuc_k^-\colon \bfA^{\mu'}_+[\alpha+\alpha_k]\to \bfA^{\mu}_+[\alpha]$ in a similar way. We can also define the graded version $\widetilde\Zuc_k^{\pm}$ of the Zuckerman functors by replacing the functors $\inc$ and $\tr$ by their graded versions $\widetilde\inc$ and $\widetilde\tr$. Unfortunately this approach does not allow to define the Zuckerman functors for $k=0$ because of the assumption $k\ne 0$ in Lemma \ref{ch3:lem-Zuck_defined_A+}. The definition of the Zuckerman functors for $k=0$ will be given in Section \ref{ch3:subs_k=0}.

\subsection{The Koszul dual functors in the category $\bfA$}

\begin{thm}
\label{ch3:thm-final-F-E-dual-Zuck}
Assume that we have $\nu_r>|\alpha|$ for each $r\in[1,l]$, $e>2$ and $k\ne 0$.

$(a)$ The functor $F_k\colon \bfA^\nu[\alpha]\to\bfA^\nu[\alpha+\alpha_k]$ has a graded lift $\widetilde F_k$ such that the functor $\widetilde F_k\colon D^b(\widetilde\bfA^\nu[\alpha])\to D^b(\widetilde\bfA^\nu[\alpha+\alpha_k])$ is Koszul dual to the shifted Zuckerman functor $\widetilde\Zuc_k^+\langle\mu_{k+1} \rangle\colon D^b(\widetilde\bfA^\mu_+[\alpha])\to D^b(\widetilde\bfA^{\mu'}_+[\alpha+\alpha_k])$.

$(b)$ The functor $E_k\colon\bfA^\nu[\alpha+\alpha_k]\to\bfA^\nu[\alpha]$ has a graded lift such that the functor $\widetilde E_k\colon D^b(\widetilde\bfA^\nu[\alpha+\alpha_k])\to D^b(\widetilde\bfA^\nu[\alpha])$ is Koszul dual to the shifted Zuckerman functor $\widetilde\Zuc_k^-\langle \mu_{k}-1 \rangle\colon D^b(\widetilde\bfA^\mu_+[\alpha])\to D^b(\widetilde\bfA^{\mu'}_+[\alpha+\alpha_k])$.
\end{thm}
\begin{proof}
By Theorem \ref{ch3:thm_decomp_Fk-A} we have the following commutative diagram
$$
\begin{diagram}
\node{\overline\bfA^\nu[\beta+\overline\alpha]} \arrow{e,t}{\overline F_k}
\node{\overline\bfA^\nu[\beta+\overline\alpha+\overline\alpha_k]} \arrow{e,t}{\overline F_{k+1}} 
\node{\overline\bfA^\nu[\beta+\overline\alpha+\overline\alpha_k+\overline\alpha_{k+1}]} \arrow{s,r}{} \\
\node{\bfA^\nu[\alpha]} \arrow{n,l}{} 
\arrow[2]{e,b}{F_k} \node[2]{\bfA^\nu[\alpha+\alpha_k]}
\end{diagram}
$$

Here the vertical maps are some equivalences of categories.
By unicity of Koszul grading (see \cite[Cor.~2.5.2]{BGS}) there exist unique graded lifts of vertical maps such that they are equivalences of graded categories and they respect the chosen grading of simple modules (i.e., concentrated in degree $0$). Moreover, the top horizontal maps have graded lifts because for a suitable $v$ we have
$$
\overline\bfA^\nu[\beta+\overline\alpha]\subset  {^v\overline O}_{\overline\mu}^\nu,\qquad  \overline\bfA^\nu[\beta+\overline\alpha+\overline\alpha_k]\subset  {^v\overline O}_{\overline\mu^0}^\nu,\qquad \overline\bfA^\nu[\beta+\overline\alpha+\overline\alpha_k+\overline\alpha_{k+1}]\subset  {^v\overline O}_{\overline\mu'}^\nu
$$
and $W_{\overline\mu}\supset W_{\overline\mu^0}\subset W_{\overline\mu'}$. This implies that there is a graded version $\widetilde F_k$ of the functor $F_k$ such that it makes the graded version of the diagram above commutative.

Since the categories $\bfA^\nu[\alpha]$ and $\overline\bfA^\nu[\beta+\overline\alpha]$ are equivalent, their Koszul dual categories are also equivalent. We can chose the equivalences $(\bfA^\nu[\alpha])^!\simeq \bfA^\mu_+[\alpha]$ and $(\overline\bfA^\nu[\beta+\overline\alpha])^!\simeq \bfA^\mu_+[\alpha]$ in such a way that the vertical map in the diagram is Koszul dual to the identity functor. We can do the same with the categories in the right part of the diagram above.

By Theorem \ref{ch3:thm_dual-func-A-spcases}, the left top functor in the graded version of the diagram above is Koszul dual to the parabolic inclusion functor $\widetilde\inc$ and the top right functor in the diagram is Koszul dual to the graded shift $\widetilde\tr\langle\mu_{k+1}\rangle$ of the parabolic truncation functor. By definition (see Section \ref{ch3:subs_Zuck}), the Zuckerman functor is the composition of the parabolic inclusion and the parabolic truncation functors. This implies that the functor $\widetilde F_k\colon D^b(\widetilde\bfA^\nu[\alpha])\to D^b(\widetilde\bfA^\nu[\alpha+\alpha_k])$ is Koszul dual to the shifted Zuckerman functor $\widetilde\Zuc_k^+\langle\mu_{k+1}\rangle$. This proves $(a)$.

We can prove $(b)$ in the same way. By adjointness, the diagram above yields a similar diagram for the functor $E$. This diagram allows to grade the functor $E_k$. Then we deduce the Koszul dual functor to $E_k$ in the same way as in $(a)$.
\end{proof}

\subsection{The case $k=0$}
\label{ch3:subs_k=0}
Now, we are going to get an analogue of Theorem \ref{ch3:thm-final-F-E-dual-Zuck} in the case $k=0$. The main difficulty in this case is that we cannot define Zuckerman functors for the category $\bfA_+$ in the same was as in Section \ref{ch3:subs_Zuck-A} because Lemma \ref{ch3:lem-Zuck_defined_A+} fails. To fix this problem we replace the category $\bfA$ by a smaller category $A$.

Assume that we have $k=0$ and $W_{\mu}\supset W_{\mu'}$. In particular this implies $\mu_1=0$.

Let $A^\nu[\alpha+\alpha_0]$ be the Serre subcategory of $\bfA^\nu[\alpha+\alpha_0]$ generated by simple modules $L^{\lambda}$ such that the weight $\lambda\in P$ has no coordinates equal to $1$. It is a highest weight subcategory.

\smallskip
\begin{rk}
\label{ch3:rk-gr_from_bfA_to_A}
$(a)$
The category $A^\nu[\alpha+\alpha_0]$ inherits the Koszul grading from the category $\bfA^\nu[\alpha+\alpha_0]$ in the following way. We know that there is a Koszul algebra $A$ such that $\bfA^\nu[\alpha+\alpha_0]\simeq \mod(A)$. Let $\{e_\lambda;\lambda\in\Lambda\}$ be the set of indecomposable idempotents of $A_0$. Then by \cite[Lem.~2.17]{Mak-Koszul} there is a subset $\Lambda'\subset \Lambda$ such that we have $\bfA^\nu[\alpha+\alpha_0]\simeq \mod(_{\Lambda'}A)$ (see Section \ref{ch3:subs_key-lem} for the notations). Moreover, the Koszul dual algebra to $_{\Lambda'}A$ is $A^!_{\Lambda'}$.

Since, we have $\mod(A^!)\simeq \bfA^{\mu'}_+[\alpha+\alpha_0]$, the Koszul dual category $A^{\mu'}_+[\alpha+\alpha_0]$ to $A^\nu[\alpha+\alpha_0]$ is a Serre quotient of $\bfA^{\mu'}_+[\alpha+\alpha_0]$. The quotient functor $$a\colon \bfA^{\mu'}_+[\alpha+\alpha_0]\to A^{\mu'}_+[\alpha+\alpha_0]$$ can be seen as the functor 
$$
a\colon\mod(A^!)\to\mod(A^!_{\Lambda'}),\qquad M\mapsto e^!_{\Lambda'}M.
$$

$(b)$ The left adjoint functor  $b\colon A^{\mu'}_+[\alpha+\alpha_0]\to \bfA^{\mu'}_+[\alpha+\alpha_0]$ to $a$ can be seen as
$$
b\colon \mod(A^!_{\Lambda'})\to \mod(A^!),\qquad M\mapsto A^!e^!_{\Lambda'}\otimes _{A^!_{\Lambda'}}M.
$$
The functors $a$ and $b$ have obvious graded lifts
$$
\widetilde a\colon \widetilde\bfA^{\mu'}_+[\alpha+\alpha_0]\to \widetilde A^{\mu'}_+[\alpha+\alpha_0],\qquad \widetilde b\colon \widetilde A^{\mu'}_+[\alpha+\alpha_0]\to \widetilde\bfA^{\mu'}_+[\alpha+\alpha_0].
$$

By Proposition \ref{ch3:prop_dual-F-G}, the functor $\widetilde b$ is Koszul dual to the inclusion functor $\widetilde A^\nu[\alpha+\alpha_0]\to \widetilde\bfA^\nu[\alpha+\alpha_0]$. Then, by adjointness, the functor $\widetilde a$ is Koszul dual to the right adjoint functor to the inclusion functor above.  
\end{rk}

\smallskip
It is easy to see from the action of $F_0$ on Verma modules (see Proposition \ref{ch3:prop_functors-on-O-gen} $(e)$) that the image of the functor $F_0\colon \bfA^\nu[\alpha]\to \bfA^\nu[\alpha+\alpha_0]$ is in $A^\nu[\alpha+\alpha_0]$. Moreover, recall from Section \ref{ch3:subs_cat-bfA} that the functor $E_0\colon O_{\mu'}^\nu\to O_{\mu}^{\nu}$ does not take $\bfA^\nu[\alpha+\alpha_0]$ to $\bfA^\nu[\alpha]$. (The reader should pay attention to the fact that the functor $E_0$ for the category $\bfA$ is not defined as the restriction of the functor $E_0$ for the category $O$.) However, it is easy to see from the action of $E_0$ on Verma modules (see Proposition \ref{ch3:prop_functors-on-O-gen} $(e)$) that the functor $E_0$ for the category $O$ takes  $A^\nu[\alpha+\alpha_0]$ to $\bfA^\nu[\alpha]$. Thus we get a functor $E_0\colon A^\nu[\alpha+\alpha_0]\to\bfA^\nu[\alpha]$. This functor also coincides with the restriction of the functor $E_0\colon\bfA^\nu[\alpha+\alpha_0]\to\bfA^\nu[\alpha]$ to the category $A^\nu[\alpha+\alpha_0]$. 


The following statement can be proved in the same way as Lemma \ref{ch3:lem-Zuck_defined_A+}.

\smallskip
\begin{lem}
\label{ch3:lem-Zuck_defined_A+_k=0}
Assume that we have $W_{\mu}\supset W_{\mu'}$.

$(a)$
The inclusion of categories ${^uO}_{\nu,+}^{\mu}\subset {^uO}_{\nu,+}^{\mu'}$ yields an inclusion of categories $\bfA_+^{\mu}[\alpha]\subset A_+^{\mu'}[\alpha+\alpha_0]$.

$(b)$ 
The inclusion of categories ${^u\widetilde O}_{\nu,+}^{\mu'}\subset {^u\widetilde O}_{\nu,+}^{\mu}$ yields an inclusion of categories $\widetilde \bfA_+^{\mu'}[\alpha+\alpha_0]\subset \widetilde A_+^\mu[\alpha]$.

\qed
\end{lem}

The lemma above allows us to define the inclusion and the truncation functors $\inc\colon \bfA_+^{\mu}[\alpha]\to A_+^{\mu'}[\alpha+\alpha_0]$, $\tr\colon A_+^{\mu'}[\alpha+\alpha_0]\to \bfA_+^{\mu}[\alpha]$ and their graded versions $\widetilde\inc$, $\widetilde\tr$.

We still assume $k=0$ but we do not assume $W_{\mu}\supset W_{\mu'}$ any more. We define the Zuckerman functors $\Zuc^{\pm}_0$ for this case. Let us identify $\bfA_+^{\mu}[\alpha]\simeq \overline\bfA_+^{\overline\mu}[\beta+\overline\alpha]$ and $\bfA_+^{\mu'}[\alpha+\alpha_k]\simeq \overline\bfA_+^{\overline\mu'}[\beta+\overline\alpha+\overline\alpha_k+\overline\alpha_{k+1}]$. By Lemmas \ref{ch3:lem-Zuck_defined_A+}, \ref{ch3:lem-Zuck_defined_A+_k=0} we have the following inclusions of categories
$$
\overline\bfA_+^{\overline\mu}[\beta+\overline\alpha]\subset \overline A_+^{\overline\mu^0}[\beta+\overline\alpha+\overline\alpha_0],\qquad \overline\bfA_+^{\overline\mu^0}[\beta+\overline\alpha+\overline\alpha_0]\supset \overline\bfA_+^{\overline\mu'}[\beta+\overline\alpha+\overline\alpha_0+\overline\alpha_{1}].
$$
We define the Zuckerman functor $\Zuc_0^+\colon \bfA_+^{\mu}[\alpha]\to \bfA_+^{\mu'}[\alpha+\alpha_0]$ as the composition 
$$
\bfA_+^{\mu}[\alpha]\stackrel{\inc}{\to}\overline A_+^{\overline\mu^0}[\beta+\overline\alpha+\overline\alpha_0]\stackrel{b}{\to}\overline \bfA_+^{\overline\mu^0}[\beta+\overline\alpha+\overline\alpha_0]\stackrel{\tr}{\to} \bfA_+^{\mu'}[\alpha+\alpha_0].
$$
Similarly, we define the Zuckerman functor $\Zuc_0^-\colon \bfA_+^{\mu'}[\alpha+\alpha_0]\to \bfA_+^{\mu}[\alpha]$ as the composition
$$
\bfA_+^{\mu'}[\alpha+\alpha_0]\stackrel{\inc}{\to}\overline \bfA_+^{\overline\mu^0}[\beta+\overline\alpha+\overline\alpha_0]\stackrel{a}{\to}\overline A_+^{\overline\mu^0}[\beta+\overline\alpha+\overline\alpha_0]\stackrel{\tr}{\to}\bfA_+^{\mu}[\alpha].
$$
Replacing the functors $\inc$, $\tr$, $a$, $b$ by their graded versions $\widetilde\inc$, $\widetilde\tr$, $\widetilde a$, $\widetilde b$ yields graded versions $\widetilde\Zuc_0^+$ and $\widetilde\Zuc_0^-$ of the Zuckerman functors.

Now, similarly to Theorem \ref{ch3:thm_dual-func-A-spcases} we can prove the following.

\smallskip
\begin{thm}
\label{ch3:thm_dual-func-A-spcases-k=0}
Assume that we have $k=0$ and $W_{\mu}\supset W_{\mu'}$.

$(a)$ The functor $\widetilde F_0\colon D^b(\widetilde\bfA^\nu[\alpha])\to D^b(\widetilde A^\nu[\alpha+\alpha_0])$ is Koszul dual to the parabolic inclusion functor $\widetilde\inc\colon D^b(\widetilde\bfA^\mu_+[\alpha])\to D^b(\widetilde A^{\mu'}_+[\alpha+\alpha_0])$.

$(b)$ The functor $\widetilde E_0\colon D^b(\widetilde A^\nu[\alpha+\alpha_0])\to  D^b(\widetilde\bfA^\nu[\alpha])$ is Koszul dual to the shifted parabolic truncation functor $\widetilde\tr\langle \mu_{0}-1\rangle\colon D^b(\widetilde A^{\mu'}_+[\alpha+\alpha_0])\to D^b(\widetilde\bfA^\mu_+[\alpha])$.
\qed
\end{thm}

\smallskip
Finally, we get an analogue of Theorem \ref{ch3:thm-final-F-E-dual-Zuck} in the case $k=0$. 

\smallskip
\begin{thm}
Assume that we have $\nu_r>|\alpha|$ for each $r\in[1,l]$ and $e>2$.

$(a)$ The functor $F_0\colon \bfA^\nu[\alpha]\to\bfA^\nu[\alpha+\alpha_0]$ has a graded lift $\widetilde F_0$ such that the functor $\widetilde F_0\colon D^b(\widetilde\bfA^\nu[\alpha])\to D^b(\widetilde\bfA^\nu[\alpha+\alpha_0])$ is Koszul dual to the shifted Zuckerman functor $\widetilde\Zuc_0^+\langle\mu_{1} \rangle\colon D^b(\widetilde\bfA^\mu_+[\alpha])\to D^b(\widetilde\bfA^{\mu'}_+[\alpha+\alpha_0])$.

$(b)$ The functor $E_0\colon\bfA^\nu[\alpha+\alpha_0]\to\bfA^\nu[\alpha]$ has a graded lift $\widetilde E_0$ such that the functor $\widetilde E_0\colon D^b(\widetilde\bfA^\nu[\alpha+\alpha_0])\to D^b(\widetilde\bfA^\nu[\alpha])$ is Koszul dual to the shifted Zuckerman functor $\widetilde\Zuc_0^-\langle \mu_{0}-1 \rangle\colon D^b(\widetilde\bfA^\mu_+[\alpha])\to D^b(\widetilde\bfA^{\mu'}_+[\alpha+\alpha_0])$.
\end{thm}
\begin{proof}
The proof is similar to the proof of Theorem \ref{ch3:thm-final-F-E-dual-Zuck}. To prove $(a)$ we should consider the diagram as in the proof of Theorem \ref{ch3:thm-final-F-E-dual-Zuck} with an additional term.
$$
\begin{diagram}
\node{\overline\bfA^\nu[\beta+\overline\alpha]} \arrow{e,t}{\overline F_0}
\node{\overline A^\nu[\beta+\overline\alpha+\overline\alpha_0]}  \arrow{e,t}{}
\node{\overline\bfA^\nu[\beta+\overline\alpha+\overline\alpha_0]} \arrow{e,t}{\overline F_{1}} 
\node{\overline\bfA^\nu[\beta+\overline\alpha+\overline\alpha_0+\overline\alpha_{1}]} \arrow{s,r}{} \\
\node{\bfA^\nu[\alpha]} \arrow{n,l}{} 
\arrow[3]{e,b}{F_0} \node[3]{\bfA^\nu[\alpha+\alpha_0]}
\end{diagram}
$$

We prove $(b)$ in the same way by considering the diagram obtained from the diagram above by adjointness. Note that in this case we have the adjunction $(F_0,E_0)$ (and not only $(E_0,F_0)$) because of the assumption on $\nu$. 
\end{proof}

\section*{Acknowledgements}
I am grateful for the hospitality of the Max-Planck-Institut f\"ur Mathematik in Bonn, where a big part of this work is done.
I would like to thank \'Eric Vasserot for his guidance and helpful
discussions during my work on this paper. I would also like to thank C\'edric Bonnaf\'e for his comments on an earlier version of this paper.

\end{document}